\documentclass[a4paper,10pt,reqno]{amsart} 
\usepackage[body={13cm,21cm}]{geometry}
\usepackage[initials]{amsrefs}

\usepackage{amssymb,amscd}
\usepackage[all]{xy}

\usepackage[mathscr]{eucal}
\usepackage{enumerate}

\numberwithin{equation}{section}
\theoremstyle{plain}
\newtheorem{thm}{Theorem}[section]
\newtheorem{cor}[thm]{Corollary}
\newtheorem{lem}[thm]{Lemma}
\newtheorem{prop}[thm]{Proposition}
\newtheorem*{problem*}{Problem}
\theoremstyle{definition}
\newtheorem{Def}[thm]{Definition}
\theoremstyle{remark}

\newtheorem{rem}[thm]{Remark}
\newtheorem*{claim*}{Claim}
\newtheorem*{exa*}{Example}
\newtheorem*{rem*}{Remark}
\newtheorem*{rems*}{Remarks}
\newtheorem*{fact*}{Fact}

\providecommand{\R}[1]{\mathrm{#1}}

\DeclareMathOperator{\Br}{Br}

\DeclareMathOperator{\disc}{disc}

\DeclareMathOperator{\Ext}{Ext}

\DeclareMathOperator{\Pic}{Pic}

\DeclareMathOperator{\Spin}{Spin}

\def\br{   {\rm Br\,}       }
\def \brun   { {   {\rm Br}_1\,    }} 
\def\Br{{\rm Br\,}}

\def\pic{{\rm Pic\,}}
\def\hom{{\rm Hom}}
\def\Q{{\mathbb Q}}
\def\Z{{\mathbb Z}}
\def\R{{\mathbb R}}

\def\N{{\mathbb N}}
\def\G{{\mathbb G}} 
\def\P{{\mathbb P}}
\def\A{{\mathbb A}}

\def\k{{\overline k}}
\def\ker {{\rm  Ker}} 
\def\hom{{\rm Hom}}
\def\g{\mathfrak{g}}

\def\oi{\hskip1mm {\buildrel \simeq \over \rightarrow} \hskip1mm}

\newcommand{\etale}{{\textup{\'et}}} 

\DeclareFontFamily{U}{wncy}{}
\DeclareFontShape{U}{wncy}{m}{n}{%
<5>wncyr5%
<6>wncyr6%
<7>wncyr7%
<8>wncyr8%
<9>wncyr9%
<10>wncyr10%
<11>wncyr10%
<12>wncyr6%
<14>wncyr7%
<17>wncyr8%
<20>wncyr10%
<25>wncyr10}{}
\DeclareMathAlphabet{\cyr}{U}{wncy}{m}{n}
\begin{document}

\title[Brauer--Manin obstruction and integral quadratic forms]
{Brauer--Manin obstruction for integral points of homogeneous spaces and 
representation by integral quadratic forms}
 \author{Jean-Louis Colliot-Th\'el\`ene}
\address{C.N.R.S., UMR  8628, Math\'ematiques, B\^atiment 425, Universit\'e   Paris-Sud,
F-91405 Orsay, France}
\email{jlct@math.u-psud.fr}
\author{Fei Xu}
\address{Academy of Mathematics and System Science, CAS,
Beijing 100080, P.R.China}
\email{xufei@math.ac.cn}
\date{\footnotesize December 12, 2007}

\maketitle

 {\it   De toute fa\c con, les consid\'erations expos\'ees ici se pr{\^e}tent \`a des
g\'en\'era\-lisations vari\'ees, qu'il n'entrait pas dans notre propos d'examiner pour l'instant.
 
\hskip 10cm A. Weil  \ \cite{W}
}
 

\section*{Introduction} \label{sec.intro}

Representation of an integral quadratic form, of rank $n$,  by another integral quadratic form, of rank  $m  \geq n$,  has been 
a subject of investigation for many years. The most natural question is that of the representation
of an integer by a given integral quadratic form ($n=1$, $m$ arbitrary).

Scattered in the literature one finds many examples where the problem can be solved locally,
that is over the reals and over all the rings $\Z_{p}$ of $p$-adic integers, but the problem cannot
be solved over $\Z$: these are counterexamples to a local-global principle for the problem of
integral representation.
One here encounters  such concepts as that of  ``spinor exceptions''.

It is the purpose of the present paper to give a conceptual framework for 
entire series of  such counterexamples.

The situation resembles the one which Manin encountered in 1970 regarding the classical Hasse principle, namely the question of existence of a rational point on a variety defined over the rationals
when one knows that there are points in each completion of $\Q$. Manin analyzed most of the
then known counterexamples by means of the Brauer group of varieties.

Our key tool is a straightforward variation on the Brauer--Manin obstruction,
which we call the integral Brauer--Manin obstruction.
We are over a number field $k$, with ring of integers $O$, and we are interested
in the integral points of a certain $O$-scheme ${\bf X}$ associated to the representation problem.
An important point is that even though we are interested in the set of integral points
of the  scheme ${\bf X}$ the Brauer group which we use is the Brauer group of the
$k$-variety $X={\bf X} \times_{O}k$, and not the Brauer group of ${\bf X}$, as one would naively
imagine. This obstruction is defined in \S \ref{sec.notation}.

In this paper we restrict  attention to the problem of representation 
of a form $g$ of rank $n\geq 1$ by  a form $f$ of rank $m \geq 3$. 
Then the $k$-varieties which underly the problem are  homogeneous spaces of spinor groups.

In  \S 
\ref{sec.Brauer}
we discuss
Brauer groups and  Brauer pairings on  homogeneous spaces of connected
linear algebraic groups over an arbitrary field.

In the next two sections we discuss rational and integral points on homogeneous spaces under  an arbitrary connected linear group $G$.  There are two types of results, depending on whether the geometric stabilizer is connected (\S \ref{sec.BMobsconnected}) or whether it is a finite commutative group (\S \ref{sec.BMobsfinite}).  
For integral points, the main results, Theorem \ref{obs.BM.strong} and Theorem \ref{thm.strongapp.finite},
assert that the integral Brauer--Manin obstruction to
the existence of an integral point is the only obstruction provided the group $G$ is simply connected
and satisfies an isotropy condition at  the archimedean places. These conditions ensure that the group
$G$ satisfies the strong approximation theorem.

The tools we use have already been used most efficiently in the
study of rational points, by Sansuc  \cite{S1981} and Borovoi \cite{Bo1996}.
Our results on rational points are very slight extensions of their results.
Some results on integral points  already appear in a paper by Borovoi and Rudnick \cite{BR1995}.
It has recently
 come to our attention that further related results appear in the appendix of a paper of
Erovenko and Rapinchuk \cite{ER2006}.

For the representation problem of quadratic forms over the integers, the 
isotropy condition is that  the form $f$ is ``indefinite'', i.e. isotropic at some archimedean completion of  the number field $k$. In this case, if there is no integral  Brauer--Manin obstruction,
as defined in the present paper, then the form $g$ is represented by $f$ over the integers.
As may be expected, the cases $m-n \geq 3$ (Theorem \ref{atleast3}), $m-n=2$ (Theorem  \ref{thm.obs.codim2.indefinite}) and $1 \geq m-n \geq 0$ (Theorem  \ref{differenceatmostone})
each require a separate discussion. For $m-n \geq 3$,  the geometric stabilizer is simply connected, there are no Brauer--Manin obstructions,
the result on representation by indefinite forms is classical. For  $m-n=2$, the geometric stabilizer is a $1$-dimensional torus. For $1 \geq m-n \geq 0$, the geometric stabilizer is $\Z/2$.

In \S  \ref{genera} 
 we compare the results obtained in \S  \ref{sec.rep.quad.integers}   with classical results in terms of  genera and spinor genera. We give a new look at the notion of  ``spinor exception''.

In \S  \ref{sec.rep.quad.ex} we show how various examples in the literature can all be interpreted in terms of the integral Brauer--Manin obstruction. We also illustrate the general results by
a theorem a special case of which is: Let $f(x_{1},\dots,x_{n})$, resp.  $l(x_{1},\dots,x_{n})$,
be a polynomial of total degree 2, resp. total degree 1, with coefficients in $\Z$. Assume that $l$ does not divide $f$,
  that the affine $\Q$-variety defined by $f=l=0$ is smooth, and that its set of real points is noncompact.
For $n \geq 5$, the existence of solutions to $f=l=0$  in all $\Z_{p}$ implies the existence of solutions in $\Z$. For $n=4$, this is so if and only if there is no Brauer--Manin obstruction (a condition which in this case  can easily be checked).  

In \S  \ref{sec.sum.three.squares}   we apply the technique to recover a theorem characterizing sums of three squares in the ring of integers of an imaginary quadratic field --  without using   computation of integral spinor norms and   Gauss genus theory.
 In an appendix (\S \ref{appendix}), 
this kind of argument also enables Dasheng Wei and the second named author to establish a local-global principle for sums of three squares in the ring of integers of an arbitary cyclotomic field.
 
 \bigskip

{\it Acknowledgments} The work on this paper started at the
Conference on Arithmetic Geometry and Automorphic 
Forms  held at Nankai Institute of Mathematics, Tianjin  in August 2005.
It was pursued  
over e-mail and during stays of the first named author at Tsinghua University, Beijing  (August 2005)
and at the Morning Centre of Mathematics, Beijing  (April 2007) and of the second named author
at I.H.\'E.S, Bures-sur-Yvette  (July 2006). We thank David Harari and Mikhail Borovoi for comments on
our text.

\section{Notation; the integral Brauer--Manin obstruction} \label{sec.notation}

For an arbitrary scheme $X$, with structural sheaf $O_{X} $, we let $\pic X =H^1_{Zar}(X,O_{X}^*)$ denote the Picard group of $X$. Given a sheaf $\mathcal F$ on the \'etale site of $X$, we
let   $H^r_{\etale}(X,\mathcal F)$ denote the \'etale cohomology groups of the sheaf $\mathcal F$.
There is a natural isomorphism $\pic X \simeq H^1_{\etale}(X,\G_{m})$, where
$\G_{m}$ is the \'etale sheaf associated to the multiplicative group $\G_{m}$ over $X$.
The Brauer group of $X$ is $\br X = H^2_{\etale}(X,\G_{m})$.
For background on the Brauer group, we  refer to Grothendieck's expos\'es \cite{Gro}.

Let $k$ be a field. A $k$-variety is a separated $k$-scheme of finite type.
Given a $k$-variety $X$ and a field extension $K/k$ we set
$X_{K}=X \times_{k}K$. 
For $K={\overline k}$  a separable closure of $k$ we write ${\overline X}=X \times_{k}{\overline k}$.
We let $K[X]=H^0(X_{K},O_{X_{K}})$ be the ring of
global functions on $X_{K}$. We let $K[X]^*$ be the group of units in that ring.
We let $X(K)$ be the set of $K$-rational points of $X$,
that is $X(K)={\rm Hom}_{k}({\rm Spec}(K), X)$.

Let $k$ be a number field and $O$ its ring of integers.  
Let $\Omega_{k}$ denote the set of  places of $k$.
For $v \in \Omega_{k}$ we let $k_{v}$ denote the completion of $k$ at $v$.
For $v$ nonarchimedean we let $O_{v}$ denote the ring of integers in $k_{v}$.
For each place $v$, class field theory yields an 
embedding ${\rm inv}_{v}: \br k_{v} \hookrightarrow \Q/\Z$.

Let $S$ be a finite set of places of $k$ containing all archimedean places.  
Let ${\bf X}$ be a scheme separated and of finite type over the ring $O_{S}$ of $S$-integers of $k$. 
The set ${\bf X}(O_{S})$ is the set of $S$-integral points of $\bf X$.

Let $X ={\bf X} \times_{O_{S}}k$.
 For any commutative integral $O$-algebra $A$
with $O_{S} \subset A$ 
with field of fractions $F$, the natural map
$${\bf X}(A)= {\rm Hom}_{O_{S}}({\rm Spec} A, {\bf X}) \to X(F)= {\rm Hom}_{k}({\rm Spec} F,X)$$
is an injection. If ${\bf Y} \subset {\bf X}$ denotes the schematic closure of
$X$ in ${\bf X}$ then ${\bf Y}(A) ={\bf X}(A)$. 
For all problems considered here, we could thus replace ${\bf X}$ (which may have very bad
special fibres)  by the flat $O$-scheme ${\bf Y}$.

We may thus view the set ${\bf X}(O_{S})$ 
as a subset of $X(k)$ and
for each  place $v \notin S$ of  $k$ we may view
 ${\bf X}(O_{v})$ as a subset of $X(k_{v})$.
The latter space is given the topology induced by that of the local field $k_{v}$.
As one easily checks,  ${\bf X}(O_{v})$ is open in $X(k_{v})$.
If $U \subset X$ is a dense Zariski open set of $X$,
if $X/k$ is smooth, the implicit function theorem imples that  $U(k_{v})$ is dense in $X(k_{v})$.
This implies that ${\bf X}(O_{v}) \cap U(k_{v})$ is dense in ${\bf X}(O_{v})$.

An ad\`ele of the $k$-variety $X$ is a family $\{x_{v}\} \in \prod_{v \in \Omega_{k}} X(k_{v})$
such that for almost all $v$, the point  $x_{v}$ belongs to ${\bf X}(O_{v})$.
This definition does not depend on the model ${\bf X}/ O_{S}$.
The set of ad\`eles of $X$ is denoted $X({\mathcal A}_{k})$.
There is a natural diagonal embedding $X(k) \subset X({\mathcal A}_{k})$.

There is a natural pairing between   $X({\mathcal A}_{k})$  
and the Brauer group $\br X =H^2_{\etale}(X,\G_{m})$ of $X$:
$$X({\mathcal A}_{k}) \times \br X \to \Q/\Z.$$

$$( \{x_{v}\} , \alpha) \mapsto \sum_{v} {\rm inv}_{v}(\alpha(x_{v})).$$
Any element  of $X(k)$ is in the kernel of that pairing. The image of $\br k \to \br X$
is in the right kernel of that pairing. 

When $X/k$ is proper, for a given element 
$\alpha \in \br X$, there  exists a finite set $S_{\alpha}$ of places of $k$ such that
for   $v \notin S_{\alpha}$ and    any  $M_{v } \in X(k_{v})$,  $\alpha(M_{v})=0$.
This enables one to produce  and analyse counterexamples to the Hasse principle.
For background on the Brauer--Manin obstruction, we refer the reader to
 \cite{Sko} and the literature cited there.

When $X/k$ is not proper,  for $\alpha \in \br X$, there is in general no such finite set
$S_{\alpha}$ of places, the pairing seems to be useless. The situation changes when one restricts
attention to integral points.

The above pairing   induces a pairing
 $$\left[ \prod_{v \in S} X(k_{v}) \times \prod_{v \notin S} {\bf X}(O_{v}) \right] \times \br X \to \Q/\Z$$
 which vanishes on the image of $X(k)$ in the left hand side and vanishes  on
 the image of $\br k \to \br X$ in the right hand side. That is, we have an induced  pairing
 $$\left[ \prod_{v \in S} X(k_{v}) \times \prod_{v \notin S} {\bf X}(O_{v}) \right]   \times \br X / \br k \to \Q/\Z.$$
 
In the present context, the Brauer--Manin set 
$$ (  \prod_{v \in S} X(k_{v}) \times \prod_{v \notin S} {\bf X}(O_{v}) )^{\br X}$$
is by definition the left kernel of either of the above pairings.
We have the inclusions 
$${\bf X}(O_{S}) \subset   (\prod_{v \in S} X(k_{v}) \times \prod_{v \notin S} {\bf X}(O_{v}) )^{\br X} \subset  (\prod_{v \in S} X(k_{v}) \times \prod_{v \notin S} {\bf X}(O_{v}) ). $$
 
If the product $(\prod_{v \in S} X(k_{v}) \times \prod_{v \notin S} {\bf X}(O_{v}) ) $ is not empty  but the Brauer--Manin set is empty,
we say there is a Brauer--Manin obstruction to the existence of an $S$-integral point on $\bf X$.

For a given $\alpha \in \br X$ there exists a finite
set  $S_{\alpha, {\bf X}}$ of places $v$ of $k$ with $S \subset S_{\alpha, {\bf X}}$ such that for any $v \notin S_{\alpha, {\bf X}}$ and $M_{v } \in {\bf X}(O_{v})$,  we have $\alpha(M_{v})=0$.
For a given $\alpha \in \br X$, computing the image of the evaluation map
$$ev({\alpha}) : (\prod_{v \in S} X(k_{v}) \times \prod_{v \notin S} {\bf X}(O_{v}) )  \to \Q/\Z$$
is thus reduced to a finite amount of computations. 
 
If  the quotient $\br X/ \br k$ is finite,  only finitely many computations
are needed to decide if   the Brauer--Manin set is empty or not.

\bigskip

Cohomology in this paper will mostly be \'etale cohomology. Over a field $k$
with separable closure $\overline k$ and Galois group $\g={\rm Gal}({\overline k}/k)$, \'etale cohomology
is just Galois cohomology. For a continuous discrete $\g$-module $M$, we shall
denote by $H^r(\g,M)$ or $H^r(k,M)$ the Galois cohomology groups.
Given a linear algebraic group $H$ over $k$, one has the pointed cohohomogy set 
$H^1(k,H)$. This set classifies (right) principal homogeneous spaces under $H$,
up to nonunique isomorphism.
Over an arbitrary $k$-scheme $X$,
 one has the pointed cohomology set $H^1_{\etale}(X,H)$.
 This set classifies (right) principal homogeneous spaces  over $X$ under the group $H$,
 up to nonunique isomorphism. In this relative context, (right) principal homogeneous spaces
 will be referred to as torsors.

\section{Brauer groups and Brauer pairing for homogeneous spaces}   \label{sec.Brauer}

In this whole section, $k$ denotes a field of characteristic zero,
  $\overline k$ an algebraic closure of $k$, and $\g$ the Galois group 
  of $\overline k$ over $k$.
  
For any $k$-variety $X$,  let  $$\brun X = \ker [\br X \to \br {\overline X}].$$

Let us first  recall  some known results.

\begin{lem}  \label{Leray}
One has a natural exact sequence
$$0 \to H^1(\g,\k[Y]^*) \to  \pic Y \to (\pic {\overline Y})^\g \to H^2(\g,\k[Y]^*)
\to \hskip5cm$$ $$\hskip5cm  \brun Y \to  H^1(\g,\pic {\overline Y}) \to  H^3(\g,\k[Y]^*) $$
and the last map in this sequence is zero if the natural map
$ H^3(\g,\k[Y]^*) \to H^3_{\etale}(Y,\G_{m})$ is injective.
\end{lem}
\begin{proof} This is the exact sequence of terms of low degree attached to the spectral sequence
$E_{2}^{pq} =H^p(k,H^q_{\etale}({\overline Y},\G_{m})) \Longrightarrow H^n_{\etale}(Y,\G_{m})$.
\end{proof}

\begin{prop}(Sansuc)  \label{prop.Sansuc}
Let   $H/k$ a connected linear algebraic group.
Let ${\hat H}$ denote the (geometric) character group of $H$. 
 This is a finitely generated, $\Z$-free, discrete Galois module.
Let $X$ be a smooth connected $k$-variety and 
$p: Y  \to  X$ 
  be a torsor over $X$ under $H$. 
There is a natural exact sequence of abelian groups
\begin{equation}
0 \to k[X]^*/k^* \to k[Y]^*/k^* \to {\hat H}(k) \to \pic X \to \pic Y \to \pic H \to \br X \to \br Y 
\end{equation}
In this sequence, the abelian groups $k[X]^*/k^*$,  $k[Y]^*/k^*$ and ${\hat H}(k)$  
are finitely generated and free, and the group $\pic H$ is finite.
\end{prop}

\begin{proof} This is   Prop.  6.10 of \cite{S1981} . 
\end{proof}

The map  $\nu(Y)(k) : {\hat H}(k) \to \pic X$ is
 the obvious map : given a character  $\chi $ of $H$
and the $H$-torsor $Y$ over $X$, one produces 
  a $\G_{m}$-torsor over $X$ by the change of structural group defined by $\chi$.
We let $\nu(Y) : {\hat H}  \to \pic {\overline X} $ be the associated Galois-equivariant map over
$\k$.

\medskip

Any  extension  
$$1 \to \G_{m} \to H_{1}  \to H \to 1$$
of a group $H$ by the torus $\G_{m}$
defines a $\G_{m}$-torsor over $H$ hence a class in $\Pic H$.
The induced map $$\Ext_{k-gp}(H,\G_{m})  \to \Pic H$$
is functorial in the group $H$. 

 Assume that  $H$ is a connected linear algebraic group.
Then  the extension is automatically central,  and  the above map is an isomorphism
  $$\Ext^c_{k-gp}(H,\G_{m,k}) \oi \Pic H$$ 
   (see \cite{CTflasque}, Cor. 5.7).
Here $\Ext^c_{k-gp}(H,\G_{m})$ denotes  the abelian group
 of isomorphism classes of central extensions of $k$-algebraic groups of $H$ by $\G_{m}$.

\medskip

Let $H$ be an algebraic group over $k$, let $X$ be a  $k$-variety and $p: Y \to X$ be a torsor  over $X$ under $H$. There is an associated class $\xi$ in the cohomology set $H^1_{\etale}(X,H)$. 
Given any central extension of algebraic groups
$$1 \to \G_{m} \to H_{1}  \to H \to 1$$
there is a natural exact sequence of pointed sets
$$ H^1_{\etale}(X,H_{1})  \to H^1_{\etale}(X,H) \to H^2_{\etale}(X,\G_{m}) = \Br X.$$

We thus have a natural pairing
$$H^1_{\etale}(X,H) \times \Ext^c_{k-gp}(H,\G_{m})  \to \Br X$$
and this pairing is linear 
on the right hand side,   functorial in the $k$-scheme $X$ and functorial in 
the $k$-group $H$. To the torsor $Y$ there is thus associated  
a homomorphism \ of abelian groups
$$\rho_{tors}(Y)  : \Ext^c_{k-gp}(H,\G_{m})  \to  \Br X.$$

If $H$ is connected this map induces
 a homomorphism
$$\delta_{tors}(Y)  :  \Pic H  \to  \Br X.$$

\begin{prop}\label{dodgingSansuc} 
Let $H$ be a connected linear algebraic group over $k$.
Let $H_{1},X,Y$ be as above. Assume $\Pic Y =0$ and
  $Y(k) \neq \emptyset$.
Then $\rho_{tors}(Y)$ and $\delta_{tors}(Y) $ are injective.
\end{prop}
\begin{proof}
If the class of the central extension
$$1 \to \G_{m} \to H_{1}  \to H \to 1$$
is in the kernel of $ \rho_{tors}(Y)$ then
  the class of $Y$ in $H^1_{\etale}(X,H)$ is in the kernel of the map
$H^1_{\etale}(X,H) \to H^2_{\etale}(X,\G_{m}) = \Br X.$
There thus exists a torsor $Z/X$   under the
$k$-group $H_{1}$  such that the the $H$-torsor $Z \times^{H_{1}}H/X$ is
$H$-isomorphic to the $H$-torsor $Y/X$.
The projection map $Z \to Y$ makes $Z$ into a $\G_{m}$-torsor over $Y$. Since 
$\Pic Y =0$, there exists a $k$-morphism  $\sigma : Y \to Z$ which is a section of
the projection $p : Z \to Y$.
Fix $y \in Y(k)$. Let $z=\sigma(y) \in Z(k)$ and let $x \in X(k)$ be the image of
$y$ under $Y \to X$. 
Taking fibres over $x$, we get 
an $H_{1}$-torsor $Z_{x}$ over $k$ with the $k$-point $z$,
an $H$-torsor $Y_{x}$ with the $k$-point $y$,
a projection $Z_{x} \to Y_{x}$ compatible with the actions of $H_{1}$ and $H$
and a section $\sigma_{x} : Y_{x} \to Z_{x}$ sending $y$ to $z$.
The   $k$-homomorphism $H_{1} \to H$ thus admits a schemetheoretic section $\tau : H \to H_{1}$.

At this point we can appeal to  the injectivity of $\Ext^c_{k-gp}(H,\G_{m,k}) \oi \Pic H$
   (see \cite{CTflasque}, Cor. 5.7) to conclude. Alternatively, we observe that
 section $\tau$ sends the unit element of $H$
 to the unit element of $H_{1}$. It is a priori just
 a $k$-morphism of varieties.  
Because $H_{1}$ is 
a central extension of the connected $k$-group $H$ by a torus, this implies that
$\tau$ is a homomorphism of algebraic groups (this is a consequence of Rosenlicht's lemma, see \cite{CTflasque}, proof of Prop. 3.2). Thus the central extension is split.
\end{proof}
\rem
  If we assume $X$ is geometrically integral, hence also $Y$, or if $H$ is characterfree,
one may dispense with the assumption $Y(k) \neq \emptyset$.

 With notation as in the proposition, Sansuc's proposition \ref{prop.Sansuc} 
yields an exact sequence $\Pic Y \to \Pic H \to \Br X$. If one could prove that
the map $ \Pic H \to \Br X$ in that sequence coincides (up to a sign) with the
map $\delta_{tors}(Y)$, this would give another proof of Prop. \ref{dodgingSansuc}.
 \endrem

\begin{prop}\label{Kraft} 
Let $G$ be a connected linear algebraic group over $k$ and let $H \subset G$ be a closed subgroup, not necessarily connected. 
Let $X=G/H$. 
Then there is a natural exact sequence
$${\hat G}(k)  \to {\hat H}(k)  \to \Pic X \to \Pic G.$$
\end{prop}
\begin{proof}
See Prop.  3.2 of  the paper {\it The Picard group of a $G$-variety},  by H. Knop, H-P. Kraft and T. Vust
in \cite{KSS}, p. 77--87. The proof there is given over an algebraically closed field. One checks that it
extends to the above statement.
The map ${\hat H}(k)  \to \Pic X$ is the map $\nu(G)(k)$ associated to the $H$-torsor $G$ over $X=G/H$.
\end{proof}

\begin{prop} \label{pic.br.simply.connected}
Let   $G/k$ be a semisimple simply connected group.
Let $Y/k$ be a $k$-variety. Assume there exists an isomorphism of $\k$-varieties
${\overline G} \simeq {\overline  Y}$.
 We have
\begin{enumerate}[\rm(i)]
\item
The natural map
$k^* \to k[Y]^*$ is an isomorphism.
\item
 $\pic Y= 0.$
\item
 The natural map $\br k \to \br Y$ is injective.
 \item  If $Y(k)\neq \emptyset$  or
 if $k$ is a local or global field,  the map $\br k \to \br Y$  is an isomorphism.
 \end{enumerate}
\end{prop}
\begin{proof} First consider the case $Y=G$ semisimple simply connected and $k=\k$. 
In this case it is well known that $k^*=k[G]^*$
and that $\pic G=0	$. Let us give details for the slightly less known 
vanishing of
 $\br {\overline G}$.

 One reduces to the case ${\overline k}=\mathbb C$
and uses $\pi_{1}{\overline G}=0$ and $\pi_{2}{\overline G}=0$ (\'Elie Cartan).
The universal coefficient theorem then implies $H^2_{top}({\overline G},\Z/n)=0$ for
any positive integer~$n$. The comparison theorem then implies 
$H^2_{\etale}({\overline G},\mu_{n})=0$ for all $n$, hence ${}_{n}\br {\overline G}=0$
for all $n$, hence $\br {\overline G} =0$ since that group is a torsion group.

All statements now follow from the results over $\k$, Lemma \ref{Leray} and   the
vanishing of $H^3(\g,\k^*)$ for $k$ local or global.
\end{proof}

Let  $H$ a  (not necessarily connected)  
 linear algebraic group over $k$.
Let $M=H^{mult}$ denote the maximal quotient of $H$ which is a group of multiplicative type.
Let ${\hat M}={\hat H}$ denote the (geometric) character group of $H$. This is a finitely generated  discrete  $\g$-module, which we view as a commutative $k$-group scheme locally of finite type.  It coincides with the (geometric) character group of $H^{mult}$. If $H$ is connected, $M=H^{mult}$ is a $k$-torus, which is then denoted $H^{tor}$, and
${\hat H}$ is $\Z$-torsionfree.

\begin{prop} \label{prop.tautology}
Let $Y \to X$ be an $H$-torsor.
With notation as above, the  diagram
 \begin{equation}
\begin{array}{cccccccccc}
X(k)&\times & \Br X&\to &\Br k\\
\downarrow{ev_{Y}}& &\uparrow{\rho_{tors}(Y) }   &&||\\
H^1(k,H) &\times & \Ext^c_{k-gp}(H,\G_{m}) &\to &\Br k\\
\downarrow&&\uparrow&&||\\
H^1(k,M) &\times &\Ext^c_{k-gp}(M,\G_{m}) &\to &\Br k\\
|| &&\uparrow&&||\\
H^1(k,M) &\times &\Ext _{k-abgp}(M,\G_{m})&\to &\Br k\\
|| &&\uparrow{\simeq}&&||\\
H^1(k,M) &\times &H^1(k,{\hat M})&\to &\Br k\\
\end{array} \label{small1}
\end{equation}
is commutative. 
\end{prop}
\begin{proof} Commutativity of the first diagram follows from functoriality.
Commutativity of the  second and third diagrams, where $\Ext _{k-abgp}(M,\G_{m})$
denotes the group of isomorphism classes of extensions of $M$ by $\G_{m}$ in the category
of abelian $k$-group schemes, is also a matter of functoriality.
For the last diagram, 
we refer to  \S 0 of \cite{Milne1986} 
for the definition of the  map $H^1(k,{\hat M}) \to 
\Ext _{k-abgp}(M,\G_{m})$, which actually is an isomorphism, and for the fact that the diagram is commutative,
the bottom pairing being given by the cup-product (see particularly \cite{Milne1986} Prop.  0.14).
\end{proof}

The group $\Ext _{k-abgp}(M,\G_{m})$ classifies extensions 
$$ 1 \to \G_{m} \to E \to M \to 1,$$
where $E$ is a commutative algebraic group over $k$.
The group  $E$ is then  a $k$-group of multiplicative type.
Over $\k$ any such extension is split.

The injective map  $\Ext _{k-abgp}(M,\G_{m}) \to \Ext^c_{k-gp}(M,\G_{m}) $
has for its image the group of central extensions of $M$ by $\G_{m}$ which split over $\k$.

Thus the composite map 
 $$H^1(k,{\hat M}) \to \Ext _{k-abgp}(M,\G_{m}) \to \Ext^c_{k-gp}(M,\G_{m}) \to \Ext^c_{k-gp}(H,\G_{m})$$
 has its image in the subgroup of extensions split over $\k$, and the composite map
  $$H^1(k,{\hat M}) \to \Ext _{k-abgp}(M,\G_{m}) \to \Ext^c_{k-gp}(M,\G_{m}) \to \Ext^c_{k-gp}(H,\G_{m}) \to \Br X$$
  has its image in the subgroup $\brun X \subset \br X$.
  Since $\hat M \simeq \hat H$ this gives rise to a map
 $$\theta(Y) : H^1(k,{\hat H})  \to \brun X.$$

 \begin{prop}\label{homologicalalgebra}
 Let $H$ be a not necessarily connected linear algebraic group, let $M=H^{mult}$
  and let $Y \to X$ be an $H$-torsor.
 Let $\nu(Y) : {\hat H} \to \pic {\overline X}$ be the associated homomorphism.
 The induced map
 $H^1(k,{\hat H} ) \to H^1(k,\pic {\overline X})$
 coincides with the composite of the map $\theta(Y) : H^1(k,{\hat H})  \to \brun X$
 with the map $\brun X \to H^1(k,\pic {\overline X})$ in  Lemma \ref{Leray}.
  \end{prop}
  \begin{proof}   To prove this proposition one may replace the $H$-torsor $Y$
  by the $M$-torsor $Y \times^HM$. In other words, it is enough to prove the proposition
  in the case $H=M$ is a $k$-group of multiplicative type. The map $\theta(Y)$ here is just the composite
  $ H^1(k,\hat{M})  \buildrel \simeq \over \rightarrow Ext_{k-abgp}(M,\G_{m}) \to \brun X$, where the middle group is the group of abelian extensions of $M$ by $\G_{m}$, and the second map is given by the $M$-torsor $Y/X$. The proof reduces to a compatibility of spectral sequences
  with cup-products.
      \end{proof}

\begin{prop} \label{prop.connectedtautology}
Let $H$ be a connected linear algebraic group over the field $k$.
With notation as above, the group $M=H^{mult}$ is a torus. Any algebraic group extension of $H$ by $\G_{m}$ is central
and   any algebraic group extension of the $k$-torus $M$ by $\G_{m}$ is commutative.
There are   natural  compatible isomorphisms of finite groups
$$\Ext_{k-gp}(H,\G_{m})  \oi \Pic H$$ 
and 
$$\Ext_{k-gp}(M,\G_{m})  \oi \Pic M.$$ 
Let $Y \to X$ be an $H$-torsor.
The   diagram in the previous proposition yields a commutative diagram
 \begin{equation}
\begin{array}{ccccc}
X(k)&\times & \Br X&\to &\Br k\\
\downarrow{ev_{Y}}&&\uparrow{\delta_{tors}(Y) }   &&||\\
H^1(k,H) &\times & \Pic H &\to &\Br k\\
\downarrow&&\uparrow&&||\\
H^1(k,M) &\times & {\rm Ker} [\Pic H \to  \Pic {\overline H}] &\to &\Br k\\
|| &&\uparrow{\simeq}&&||\\
H^1(k,M) &\times &H^1(k,{\hat M})&\to &\Br k\\
\end{array} \label{small2}
\end{equation}
\end{prop}

\begin{proof}  
The first two statements are well known. 
The two isomorphisms have been discussed above.

That $\Pic H$ is a finite group
is a well known fact (see \cite{CTflasque}, Rappel 0.5, Proposition 3.3,  
Proposition 6.3 and the literature cited there).
Part of the right vertical map in the previous diagram now reads
$$H^1(k,{\hat M}) \simeq \Pic M \to \Pic H.$$
The statement now follows 
from Proposition \ref{prop.tautology}
provided we show that the map $\Pic M \to \Pic H$
induces an isomorphism
$\Pic M \to {\rm Ker} [\Pic H \to  \Pic {\overline H}].$
We have the exact sequence of connected algebraic groups
$$1 \to H_{1} \to H \to M \to 1,$$
where $H_{1}$ is smooth connected characterfree algebraic group.
Applying Proposition \ref{prop.Sansuc} 
to this sequence we get the exact sequences
$$ 0 \to \Pic M \to \Pic H \to \Pic H_{1}$$
and
$$ 0 \to \Pic {\overline M} \to \Pic {\overline H}  \to \Pic {\overline H}_{1}$$
which simply reads $ \Pic {\overline H}  \hookrightarrow  \Pic {\overline H}_{1}$.
Moreover Rosenlicht's lemma and Lemma \ref{Leray} show that the obvious map
$\Pic H_{1} \to \Pic {\overline H}_{1}$ is injective. This is enough to conclude.
\end{proof}

\begin{prop}\label{picbrHconnected}
Let   $G$ be a semisimple, simply connected algebraic group over $k$
and $H \subset G$ a connected $k$-subgroup. 
Let $X=G/H$.  Projection $G \to G/H$ makes $G$ into a right $H$-torsor $Y \to X$.
Let $\br_{*} X \subset \br X$ be the group of elements vanishing at the point of $X(k)$
which is the
image of $1 \in G(k)$. Projection $\Br X \to \Br X/\Br k$ induces an isomorphism
$\Br_{* }X \to \Br X/\Br k$.

(i) The  natural map
${\nu(G) : \hat H}(k) \to \pic X$ is an isomorphism. 

(ii) The map $\delta_{tors}(G) : 
\pic H \to  \br X$ attached to the torsor $ G \to G/H=X$ induces   isomorphisms $$\delta'_{tors}(G)  : \pic H \oi \Br_{*} X  \simeq   \Br X/\Br k.$$

(iii) Let $X_{c}$ be a smooth compactification of $X$. There is an isomorphism
between $\Br X_{c}$ and the group of elements of $H^1(\g, {\hat H} )$
whose restriction to each procyclic subgroup of $\g$ is zero.
\end{prop}
\begin{proof}
For (i) use either Proposition \ref{prop.Sansuc} or \ref{Kraft}.

Let us prove  (ii). The map $\delta_{tors}(G) : 
\pic H \to  \br X$ sends $\pic H$ to   $\br_{*} X \subset \br X$.
By Proposition \ref{dodgingSansuc}, the map $\delta_{tors}(G) : 
\pic H \to  \br X$ is injective.  We thus have an injective homomorphism
$\pic H \to \Br_{*} X$. By Proposition \ref{pic.br.simply.connected}, $\Pic G=0$ and  the natural  map  $\Br  k \to  \Br G$
is an isomorphism. Sansuc's Proposition \ref{prop.Sansuc} gives {\it some} isomorphism
$\pic H \simeq \Br_{*}X$. Since the group $\pic H$ is finite,
 we conclude that
 $\delta_{tors}(G) : \pic H \to \Br_{*} X$ is an isomorphism.

Statement (iii) is a special case of the main Theorem   of \cite{CTK}.
\end{proof}

\begin{rem}\label{transcendant}
 If the connected group $H$ has no characters, then the map $\Pic H \to \pic {\overline H}$ is
injective. As soon as $\Pic H \neq 0$, we thus get ``transcendental elements'' in the Brauer group of $X$,
i.e. elements of the Brauer group of $X$ whose image in $\br {\overline X}$ is nonzero.
\end{rem}

 \begin{prop}\label{picbrHdisconnected}
Let $G$ be a semisimple, simply connected algebraic group over $k$
and $H \subset G$ be a closed  $k$-subgroup, not necessarily connected.
Let $X=G/H$. 
Then 

(i) The natural map
$\nu(G) : {\hat H}(k) \to \pic X$ is an isomorphism. 

(ii)  
The   map
  $\theta(G) : H^1(k,{\hat H}) \to  \brun X $ 
  induces an isomorphism
  $H^1(k,{\hat H}) \simeq \brun X /\br k.$
\end{prop}
\begin{proof}
  Propositions \ref{Kraft} and \ref{pic.br.simply.connected} give  isomorphisms
 $ \nu(G)(k) :  {\hat  H} (k)  \oi   \pic X$, and $\nu(G) : {\hat  H}  \oi   \pic {\overline X}$.
 The first one gives (i), the second one induces an isomorphism
$H^1(k,{\hat H})  \oi  H^1(k,\pic {\overline X}).$ 
 Proposition \ref{pic.br.simply.connected}  gives
 ${\overline k}^* = {\overline k}[X]^*$.  By  Lemma \ref{Leray}  this implies 
  $\brun X/ \br k \oi H^1(k,\pic {\overline X}).$ 
 Combining this with Proposition \ref{homologicalalgebra} 
 we get
  (ii).
  \end{proof}

\section{The  Brauer--Manin obstruction for rational and integral points of homogeneous spaces with connected stabilizers} \label{sec.BMobsconnected}

Let $k$ be a number field and $H/k$ be a connected 
linear algebraic group.
Since $H$ is connected, the  image of the diagonal map $H^1(k,H) \to \prod_{v}H^1(k_{v},H)$ 
lies in the subset $ \oplus_{v}H^1(k_{v},H)$ of elements which are equal to the trivial class $1 \in H^1(k_{v},H)$ for
all but a finite number of places $v$ of $k$. For each place $v$, the pairing
$$H^1(k_{v}, H) \times \pic H_{k_{v}} \to \br k_{v} \subset \Q/\Z$$
from Proposition  \ref{prop.connectedtautology}  induces a map $H^1(k_{v}, H) \to \hom(\pic H,  \Q/\Z)$.

The following theorem is essentially due  to Kottwitz (\cite{K1986} 2.5, 2.6;
 see also \cite{BR1995}). It extends the Tate--Nakayama
theory (case when $H$ is a torus).
With the maps as defined above, a proof of the theorem  is given in
\cite{CTflasque}, Thm. 9.4 (handling the real places is a delicate point,
in \cite{CTflasque} one refers to an argument  of Borovoi).

\begin{thm}\label{Kottwitzrecipro} 
Let $k$ be a number field and $H$ a connected linear algebraic group over $k$.
The above maps induce a natural exact sequence of pointed sets
$$H^1(k,H)\to  \oplus_{v} H^1(k_{v}, H) \to \hom(\pic H, \Q/\Z).$$
\end{thm}

Let  $G$ be a connected linear algebraic group over $k$ and 
$H \subset G$  a connected subgroup. Let $X=G/H$.
Projection $G \to G/H$ makes $G$ into a right torsor over
$X$ under the group $H$.

We have the following natural commutative diagram
\begin{equation}
\begin{array}{ccccc}
G(k)                  &      \to          &          G({\mathcal A}_k )  &  & \\
\downarrow      &                   & \downarrow&          &      \\
X(k)                       &  \to            &     X({\mathcal A}_k )       &  \to  & \hom(\br X/\br k, \Q/\Z) \\
\downarrow      &                   & \downarrow&          &  \downarrow{}    \\
H^1(k,H) &     \to               &  \oplus_{v} H^1(k_{v}, H)&      \to      &  \hom(\pic H, \Q/\Z)    \\
\downarrow      &                   & \downarrow&          &      \\
H^1(k,G) & \to &  \oplus_{v} H^1(k_{v}, G). &&
\end{array} \label{big3}
\end{equation}
In this diagram the two left vertical sequences are exact sequences of pointed sets
(\cite{SerreCG}, Chap. I, \S 5.4, Prop.  36).
The map $\hom(\br X/\br k, \Q/\Z) \to  \hom(\pic H, \Q/\Z) $ is induced by the map
$\delta _{tors}(G) : \Pic H \to \br X$.
The commutativity of this diagram follows from Proposition  \ref{prop.connectedtautology}.
\medskip

A finite set $S$ of places of $k$ will be called {\it big enough for} $(G,H)$ if it contains
all the archimedean places,  there exists
a closed immersion of smooth affine $O_{S}$-group schemes 
with connected fibres
${\bf H} \subset {\bf G}$  extending the embedding
 $H \subset G$ and the quotient  $O_{S}$-scheme  ${\bf X} ={\bf G}/{\bf H}$ exists
 and is separated.
 There always exists such a finite set $S$.

\begin{thm}\label{brauer.kottwitz} Let $k$ be a number field, $H \subset G$ connected linear algebraic groups over $k$ and $X=G/H$.

 (i) With notation as above, the kernel of the  map
$$X({\mathcal A}_k) \to \hom(\br X/\br k, \Q/\Z)$$ is included in the kernel of the composite map
$$X({\mathcal A}_k) \to   \oplus_{v} H^1(k_{v}, H)      \to        \hom(\pic H, \Q/\Z).$$ 

(ii) If $G$ is semisimple and simply connected,
the kernels of these two maps coincide.

(iii) 
Let the finite set $S$ of places be big enough for $(G,H)$.
  A point $\{M_{v}\}_{v \in \Omega_{k}}$ in $\prod_{v \in S}X(k_{v}) \times
 \prod_{v \notin S}{\bf X}(O_{v})$ is in the kernel of the   composite map
 $$X({\mathcal A}_k) \to   \oplus_{v} H^1(k_{v}, H)      \to        \hom(\pic H, \Q/\Z)$$
 if and only if the point $\{M_{v}\}_{v \in S}$ is in the kernel of the composite map
 $$ \prod_{v \in S}X(k_{v} ) \to   \prod_{v\in S} H^1(k_{v}, H)      \to        \hom(\pic H, \Q/\Z).$$
\end{thm}
\begin{proof} 
Statement (i) follows from  diagram (\ref{big3}).

If $G$ is semisimple and  simply connected,   Proposition  \ref{picbrHconnected} 
implies  that the composite map $\delta'_{tors}(G) : \pic H \to \br X \to \br X/\br k$  
is an isomorphism. This gives (ii).

In the situation of (iii) for each $v \notin S$ the composite map
${\bf X}(O_{v})\to X(k_{v}) \to H^1(k_{v},H) $ factorizes as
${\bf X}(O_{v })   \to H^1_{\etale}(O_{v}, {\bf H})  \to H^1(k_{v},H)$
and $H^1_{\etale}(O_{v}, {\bf H})=1$ by Hensel's lemma together with
Lang's theorem.
\end{proof}

\medskip

In the  case where the group $G$ is semisimple and simply connected,
the hypotheses of the next three theorems are fulfilled.
In that case these theorems are due to M. Borovoi 
\cite{Bo1996}  
and M. Borovoi and Z. Rudnick \cite{BR1995}.

\begin{thm} (compare \cite{BR1995}, Thm. 3.6) \label{thm.orbit}
Let  $G$ be a connected linear algebraic group over a number field $k$ and 
$H \subset G$  a connected $k$-subgroup. Let $X=G/H$.
Assume ${\cyr X}^1(k,G)=0.$ If $\{M_{v}\}_{v \in \Omega} \in X({\mathcal A}_k)$ is orthogonal to
the image of the (finite) group $\pic H$ in $\br X$ with respect to the Brauer--Manin pairing, then there exist
$\{g_{v}\} \in G({\mathcal A}_k)$ and $M \in X(k)$ such that for each place $v$ of $k$
$$  g_{v}M=M_{v} \in X(k_{v}).$$
\end{thm}
\begin{proof} This immediately follows from     diagram (\ref{big3}).
\end{proof}

\begin{thm} \label{thm.sha.app}
Let  $G$ be a connected linear algebraic group over a number field $k$ and 
$H \subset G$  a connected $k$-subgroup. Let $X=G/H$.
Assume ${\cyr X}^1(k,G)=0$ and assume that $G$ satisfies weak approximation.
 
 (a) Let $\{M_{v}\}_{v \in \Omega} \in X({\mathcal A}_k)$ be orthogonal to the image of   the (finite) group $\pic H$
in $\br X$ with respect to the Brauer--Manin pairing. Then for each finite set $S$ of places
of $k$ and open sets $U_{v} \subset X(k_{v})$ with $M_{v } \in U_{v}$ there exists
$M \in X(k)$ such that $M \in U_{v}$ for $v \in S$.

(b) (Borovoi)  If $G$ is semisimple
 and simply connected and $H$ is geometrically characterfree, then $X$ satisfies weak approximation.
\end{thm}
\begin{proof} 
(a) Let $\{g_{v}\} \in G({\mathcal A}_k)$ and $M \in X(k)$  be as in the conclusion of the previous theorem. If $g \in G(k)$ is close enough to each $g_{v}$ for $v \in S$, then $gM \in X(k)$
belongs to each $U_{v}$ for each $v \in S$.

(b) Since $H$ is geometrically characterfree, Lemma \ref{Leray} and Rosenlicht's  lemma ensure that the natural
map $\Pic H \to \pic {\overline H}$ is injective. This implies that for any field $K$ containing $k$
the natural map of finite abelian groups $\Pic H \to \Pic H_{K}$ is injective. Thus the dual map
${\rm Hom}(\Pic H_{K},\Q/\Z) \to {\rm Hom}(\Pic H,\Q/\Z)$ is onto.
For any nonarchimedean place $w$ of $k$ the natural map
$$H^1(k_{w} , H) \to {\rm Hom}(\Pic H_{k_{w}},\Q/\Z)$$ is a bijection
(Kottwitz, see \cite{CTflasque}, Thm. 9.1 (ii)).
 Since $G$ is semisimple and simply connected and $w$
nonarchimedean,  $H^1(k_{w},G)=1$ (Kneser) hence the map $X(k_{w}) \to H^1(k_{w},H)$ is onto.
Thus the composite map 
$$X(k_{w}) \to  H^1(k_{w},H) \to {\rm Hom}(\Pic H_{k_{w}},\Q/\Z) \to {\rm Hom}(\Pic H,\Q/\Z)$$ is onto.
Let  $S$ be a finite set of places of $k$, let $\{M_{v}\}_{v \in S} \in \prod_{v\in S} X(k_{v})$ 
and for each place $v$ let $U_{v} \subset X(k_{v})$ be a neighbourhood of $M_{v}$.
Let $\varphi \in {\rm Hom}(\Pic H,\Q/\Z)$ be the image of $\{M_{v}\}_{v \in S} $
under the map
$\prod_{v \in S} X(k_{v}) \to {\rm Hom}(\Pic H,\Q/\Z)$.
Choose a 
nonarchimedean place $w \notin S$ and a point    $M_{w} \in X(k_{w})$
whose image in the group ${\rm Hom}(\Pic H_{k_{w}},\Q/\Z)$ induces $-\varphi \in {\rm Hom}(\Pic H,\Q/\Z)$.
At places $v$  not in $S \cup \{w\}$ take $M_{v} \in X(k_{v})$ to be the image of $1 \in G(k_{v})$
under the projection map $G(k_{v}) \to X(k_{v})$. Then the family $\{M_{v}\}_{v \in \Omega_{k}}$
satisfies the hypothesis in (a), which is enough to conclude.
\end{proof}

\begin{thm}\label{br.3.5}
Let  $G$ be a connected linear algebraic group over   a number field $k$ and 
$H \subset G$  a connected $k$-subgroup. Let $X=G/H$.
Let $X_{c}$ be a smooth compactification of $X$. The closure of the image of the diagonal map
$X_{c}(k) \to X_{c}({\mathcal A}_k)$ is exactly the Brauer--Manin set $X_{c}({\mathcal A}_k)^{\br X_{c}}$
consisting of elements of $X_{c}({\mathcal A}_k)$ which are orthogonal to $\br X_{c}$.
\end{thm}
\begin{proof}
 After changing both $G$ and $H$ one may assume
that the group $G$ is ``quasi\-trivial''  (see   \cite{CTK}, Lemme 1.5). For any such group $G$, weak approximation holds, and
${\cyr X}^1(k,G)=0$ (see \cite{CTflasque}, Prop. 9.2).
 Let $\{M_{v}\} \in X_{c}({\mathcal A}_k)$ be orthogonal to 
$\br X_{c}$. 
Since ${\overline G}$ is a rational variety, the smooth, projective, geometrically integral
variety $X_{c}$ is geometrically unirational. This implies that the 
 quotient $\br X_{c}/\br k$ is finite.
 Any element of $X_{c}({\mathcal A}_k)^{\br X_{c}}$
  may thus be approximated by an $\{M_{v}\} \in X({\mathcal A}_k)$
  which is orthogonal to $\br X_{c}$. 
 Let $S$ be a finite set of places of $k$.
The  group $\pic H$ is finite, its image  $B \subset \br X$ is thus a finite group
and we have $\{M_{v}\} \in X({\mathcal A}_k)^{B \cap \br X_{c}}$. 
According to a theorem of Harari (\cite{Harari}, Cor. 2.6.1, see also \cite{CTPest}, Thm. 1.4), there exists a family $\{P_{v}\} \in X({\mathcal A}_k)$
with $P_{v}=M_{v}$ for $v \in S$  
which is orthogonal to $B \subset \br X$.
By Theorem \ref{thm.sha.app}  we may find
$M \in X(k)$ as close as we wish to each $M_{v}$ for $v \in S$.
\end{proof}
\begin{rem}
Some of the results in 
   \cite{Bo1996},  \cite{Bo1999} and the references quoted therein are not covered by 
   the previous two theorems.
M. Borovoi tells us that in Theorem \ref{br.3.5} one may replace $\br X_{c}$ by $\brun X_{c}$.
\end{rem}

We now discuss integral points of homogeneous spaces. 
Let $Y $ be a variety over a number field $k$.
Let $S_{0}$ be a finite set of places of $k$ (these may be arbitrary places of $k$).
We let $O_{S_{0}}$ denote the subring of elements $x$ of $k$ which are $v$-integral
at each nonarchimedean place $v$ not in $S_{0}$.
One says that $Y$ satisfies strong approximation with respect to $S_{0}$
if the diagonal image of $Y(k)$ in the set of $S_{0}$-ad\`eles of $Y$ is dense.
The $S_{0}$-ad\`eles of $Y$  is the subset of $\prod_{v \notin S_{0}} Y(k_{v})$ consisting of elements
which are integral
at almost all places of $k$. This set is equipped with a natural restricted topology.
The definition does not depend on the choice of an integral model for $Y$.
For a discussion of various properties of strong approximation, see \cite{PR1994}, \S 7.1.
If $G/k$ is a connected linear algebraic group, $G$ satisfies strong approximation 
with respect to $S_{0}$ if and only if $G(k).(\prod_{v\in S_{0}}G(k_{v})) $ is dense in
 the group of all ad\`eles $G({\mathcal A}_k)$.

\medskip

On first reading the statement of the following theorem, the reader is invited to take $S$ to be just
 the set of archimedean places. 

\begin{thm} \label{obs.BM.strong}
Let $k$ be a number field, $O$ its ring of integers, $S$ a finite set of places
of $k$ containing all archimedean places, $O_{S}$ the ring of $S$-integers.
Let $G/k$ be a semisimple,   simply connected group. 
Let $H \subset G$ be a connected subgroup.
Let ${\bf X} $ be a separated $O_{S}$-scheme of finite type such that $X={\bf X}\times_{O_{S}}k$
is $k$-isomorphic to $G/H$.

(a) For a point  $\{M_{v}\} \in    \prod_{v \in S} X(k_{v}) \times
\prod_{v \notin S}{\bf X}(O_{v}) \subset X({\mathcal A}_k) $ the following conditions
are equivalent:

(i) it is in the kernel of the map 
$X({\mathcal A}_k) \to \hom(\br X, \Q/\Z)$;

(ii) it is in the kernel of the composite map
 $$X({\mathcal A}_k) \to \oplus_{v}  H^1(k_{v},H) \to \hom( \pic H, \Q/Z).$$

Let  $S_{1}$ with $S \subset S_{1}$ be   big enough for $(G,H)$ and ${\bf X}$, more precisely assume that
there exists a semisimple $O_{S_{1}}$-group scheme ${\bf G}$ extending $G/k$, a
smooth, fibrewise connected,   $O_{S_{1}}$-subgroup  scheme ${\bf H} \subset {\bf G}$ over $O_{S_{1}}$
extending $H$ and   an $O_{S_{1}}$- isomorphism ${\bf X}\times_{O_{S}} O_{S_{1 }} \simeq {\bf G}/ {\bf H}$ 
 extending $X \simeq G/H$. 
 Then the above conditions are equivalent to :

 (iii)  the point $\{M_{v}\}_{v \in S_{1}}$ is in the kernel of the composite map
$$ \prod_{v \in S_{1}}X(k_{v} ) \to   \prod_{v\in S_{1}} H^1(k_{v}, H)      \to        \hom(\pic H, \Q/\Z).$$

(b) Let $S_{0}  $ be a finite set of places of $k$ such that 
for each almost simple $k$-factor $G'$ of $G$ there exists a place $v \in S_{0} $ 
such that $G'(k_{v})$ is  not compact.  
Let
$S_{2}$ be a finite set of places of $k$.
 If 
$\{M_{v}\} \in    \prod_{v \in S} X(k_{v}) \times
\prod_{v \notin S}{\bf X}(O_{v}) \subset X({\mathcal A}_k) $ 
  satisfies one of the three  conditions
 above, 
 then there exists 
  $M \in {\bf X}(O_{S_{0} \cup S})$ arbitrarily close to each $M_{v}$
for $v \in S_{2} \setminus S_{0}$.  In particular, we then have ${\bf X}(O_{S \cup S_{0}}) \neq \emptyset$.
 \end{thm}
\begin{proof}
 The assumption on $G$ ensures that ${\cyr X}^1(k,G)=0$ and that $G$ satisfies strong approximation
with respect to $S_{0}$   (Kneser \cite{Kn} ; Platonov-Rapinchuk  \cite{PR1994},  chap. 7.4, Teor. 12, russian edition, p. 466 ; english edition, \S 7.4, Thm. 7.12 p. 427). 

Statement (a) then follows from Theorem  \ref{brauer.kottwitz}.

Let us prove (b). We may assume that $S_{2}$ contains $S_{1}$.
According to Theorem \ref{thm.orbit}, whose hypotheses are fulfilled, there exist  $N \in X(k)$
and a family $\{g_{v}\} \in G({\mathcal A}_k)$ such that $g_{v}.N=M_{v}$ for each place $v$.
By strong approximation with respect to $S_{0}$, there exist $g \in G(k)$  
such that $g$ is very close to
$g_{v}$ for $v \in S_{2} \setminus S_{0}$  and  $g \in {\bf G}(O_{v})$ for $v \notin  S_{2} \cup S_{0} $.
Now the point $M=g N \in X(k)$
 belongs to ${\bf X}(O_{v})$ for   $v \notin S_{2}  \cup S_{0}   $
  and it is very close to $M_{v}\in {\bf X}(k_{v})$ for $v \in S_{2}  \setminus S_{0}$.
  Hence it it also belongs to 
  $ {\bf X}(O_{v})$ for  $v \notin S  \cup S_{0}$ (the set
${\bf X}(O_{v})$ is open in $X(k_{v})$). 
 We thus have found $M \in {\bf X}(O_{S \cup S_{0}})$ very close to
   each $M_{v}$ for $v \in S_{2} \setminus S_{0}$.
\end{proof}

 The existence of an
$\{M_{v}\} \in    \prod_{v \in S} X(k_{v}) \times
\prod_{v \notin S}{\bf X}(O_{v}) $
 satisfying hypothesis (i) or (ii)
is  reduced to the existence of an $\{M_{v}\} \in \prod_{v \in S_{1}}{\bf X}(k_{v})$
as in hypothesis (iii). This can be checked  by a finite amount of computations.
These computations  involve the pairings  $H^1(k_{v}, H) \times \pic H \to \Q/\Z.$

If one uses the Brauer pairing, one can also check hypothesis (i) by means of a finite
amount of computation.
For any  $\alpha \in \br X$ there exists a finite set $S_{\alpha}$ of places  of $k$
such that  $\alpha$ vanishes on ${\bf X}(O_{v})$ for any  $v \notin S_{\alpha}$.
Let $T$ be the union of  $S$ and the $S_{\alpha}$'s for a finite set   $E$ of $\alpha$'s spanning
the finite group $\br X/ \br k$. To decide if there exists 
$\{M_{v}\} $ as in hypothesis (i)
one only has to see whether there exists an element 
$\{M_{v}\} \in \prod_{v \in S}{\bf X}(k_{v}) \times \prod_{v \in T  \setminus S }{\bf X}(O_{v})$ which is orthogonal 
to each $\alpha$ in the finite set $E$,  the sum in the Brauer--Manin pairing being taken only
over the places in $T$.
 
 \begin{rem}
 In view of Remark \ref{transcendant}, the Brauer-Manin obstruction involved in this section
 may involve transcendental elements in the Brauer group of the homogeneous spaces $X=G/H$
 under consideration.
 \end{rem}

\section{The  Brauer--Manin obstruction for rational and integral points of homogeneous spaces with finite, commutative stabilizers} \label{sec.BMobsfinite}

Let $k$ be a number field and $\mu/k$ be a finite  commutative
$k$-group scheme.
The  image of the diagonal map $H^1(k,\mu) \to \prod_{v}H^1(k_{v},\mu)$ 
lies in the restricted direct product $ \prod'_{v}H^1(k_{v},\mu)$.
 For each place $v$, we have the cup-product  pairing
$$H^1(k_{v}, \mu) \times H^1(k_{v}, \hat{\mu})  \to \br k_{v} \subset \Q/\Z.$$ 
\begin{thm}(Poitou, Tate)\label{Poitou.Tate}
The above pairings induce a natural exact sequence of commutative groups
\begin{equation}
H^1(k,\mu)\to  \prod'_{v} H^1(k_{v}, \mu) \to \hom(H^1(k, \hat{\mu}), \Q/\Z).
 \end{equation}

Let $S$ be a finite set of places of $k$ containing all the archimedean places,
such that the finite \'etale $k$-group scheme $\mu$ extends to a finite
\'etale group scheme over $O_{S}$ of order invertible in $O_{S}$,
so that $\hat{\mu}$ also extends to a finite \'etale group scheme over $O_{S}$.
Then there is an exact sequence of finite abelian groups  
\begin{equation}
 H^1_{\etale}(O_{S},\mu) \to \prod_{v \in S} H^1(k_{v},\mu) \to  \hom(H^1_{\etale}(O_{S}, \hat{\mu}), \Q/\Z).
 \end{equation}
\end{thm}

See Milne's book \cite{Milne1986}, Chapter I, \S 4, Thm. 4.10 p. 70 (for both sequences)  and Chapter II, \S 4, Prop.  4.13 (c) p. 239 (for the second sequence).

\medskip

Let  $G$ be a connected linear algebraic group over $k$ and 
$\mu \subset G$   a finite commutative $k$-subgroup, not necessarily normal in $G$.
Let $X=G/\mu$.

We have the following natural commutative diagram
\begin{equation}
\begin{array}{ccccc}
G(k)                  &      \to          &          G({\mathcal A}_k )      &              & \\
\downarrow      &                   & \downarrow         &          &      \\
X(k)                       &  \to            &     X({\mathcal A}_k )       &  \to      & \hom(\brun  X/\br k, \Q/\Z) \\
\downarrow      &                   & \downarrow     &                                      &  \downarrow    \\
H^1(k,\mu)        &     \to               &  \prod'_{v} H^1(k_{v}, \mu) &      \to      &  \hom(H^1(k, \hat{\mu}), \Q/\Z)    \\
\downarrow      &                   & \downarrow&          &      \\
H^1(k,G) & \to &  \prod'_{v} H^1(k_{v}, G) &&
\end{array} \label{bigfini}
\end{equation}
In this diagram the two left vertical sequences are exact sequences of pointed sets
(\cite{SerreCG}, Chap. I, \S 5.4, Prop.  36).  
The map $$\hom(\brun  X/\br k, \Q/\Z) \to \hom(H^1(k, \hat{\mu}), \Q/\Z)$$
is induced by the map $\theta(G) :   H^1(k, \hat{\mu}) \to \brun X \subset \br X$ 
associated to the $\mu$-torsor $G \to G/\mu=X$, as defined in
  Proposition  \ref{prop.tautology} and the comments following that proposition.
The commutativity of the right hand side square follows from Proposition  \ref{prop.tautology}.

Let $S$ be a finite set of places of $k$ which contains all the archimedean places,
and is large enough so that  all the following properties hold: there exists a  
 a smooth, linear $O_{S}$-group ${\bf G}$ with connected
fibres, the group $\mu$ comes from a finite, \'etale $O_{S}$-group scheme $\mu \subset {\bf G}$,
the group $\hat{\mu}$ comes from a finite, \'etale $O_{S}$-group scheme $\hat{\mu}$,
the $k$-variety $X$ comes from a smooth $O_{S}$-scheme ${\bf X}$ and there is a finite \'etale map 
${\bf G} \to {\bf X}$  extending $G \to X=G/\mu$  and making ${\bf G}$ into a $\mu$-torsor over ${\bf X}$.
In the balance of this section, we shall simply say that  such an $S$ is ``big enough for $(G,\mu)$''.

One then has a commutative diagram
\begin{equation}
\begin{array}{ccccc}
{\bf G}(O_{S})                  &      \to          &          \prod_{v\in S} G(k_{v})     &              & \\
\downarrow      &                   & \downarrow         &          &      \\
{\bf X}(O_{S})                       &  \to            &      \prod_{v\in S} X(k_{v})      &  \to      & \hom(\brun X, \Q/\Z) \\
\downarrow      &                   & \downarrow     &                                      &  \downarrow    \\
H^1_{\etale}(O_{S},\mu)        &     \to               &  \prod_{v\in S} H^1(k_{v}, \mu) &      \to      &  \hom(H^1_{\etale}(O_{S}, \hat{\mu}), \Q/\Z)    \\
\downarrow      &                   & \downarrow&          &      \\
H^1_{\etale}(O_{S},{\bf G}) & \to &   \prod_{v\in S} H^1(k_{v}, G) &&
\end{array} \label{bigfinifini}
\end{equation}
The map $\hom(\brun X, \Q/\Z) \to \hom(H^1_{\etale}(O_{S}, \hat{\mu}), \Q/\Z)$
is induced by the composite map
$H^1_{\etale}(O_{S}, \hat{\mu}) \to H^1(k,\hat{\mu}) \to \brun X$.
By Theorem \ref{Poitou.Tate}, the sequence on the third line, which is a sequence of finite abelian groups,  is exact. The middle and left vertical maps are exact sequences of pointed sets.

\begin{thm}\label{brauer.kottwitz.fini}
 (i) With notation as above, the kernel of the  map
$$X({\mathcal A}_k) \to \hom(\brun X/\br k, \Q/\Z)$$ is included in the kernel of the composite map
$$X({\mathcal A}_k) \to   \prod'_{v} H^1(k_{v}, \mu)      \to        \hom(H^1(k,\hat{\mu}), \Q/\Z).$$ 

(ii) If $G$ is simply connected 
the kernels of these two maps coincide.

(iii) Let the finite set $S$ of places be big enough for $(G,\mu)$.
If   the point \break  $\{M_{v}\} \in \prod_{v \in S}X(k_{v}) \times
 \prod_{v \notin S}{\bf X}(O_{v})$ is in the kernel of the   composite map
 $$X({\mathcal A}_k) \to   \prod'_{v} H^1(k_{v}, \mu)      \to        \hom(H^1(k,\hat{\mu}), \Q/\Z)$$
 then its projection  $\{M_{v}\}_{v \in S}$ is in the kernel of the composite map
$$ \prod_{v \in S}X(k_{v} ) \to   \prod_{v\in S} H^1(k_{v}, \mu)      \to        \hom(H^1_{\etale}(O_{S},  \hat{\mu}),\Q/\Z).$$
  \end{thm}
\begin{proof} 
 Statement (i) follows from diagram (\ref{bigfini}).
 If $G$ is simply connected   then by Proposition  \ref{picbrHdisconnected} 
 the composite map $H^1(k,\hat{\mu})  \to \brun X \to \brun X/\br k$ is an isomorphism.
This proves (ii). 
 For $v \notin S$, 
the image of the composite map $$ {\bf X}(O_{v}) \to X(k_{v}) \to H^1(k_{v},\mu)$$ lies in
$H^1_{\etale}(O_{v},  \mu  )$. Since the
 the cup-product
$H^1_{\etale}(O_{v},\mu) \times H^1_{\etale}(O_{v}, \hat{\mu}  ) \to \br k_{v}$ vanishes,
this proves (iii).
\end{proof}

\medskip

Proceeding as in the previous section we get the first statement in each of the following results.
Here the (generally infinite) group $H^1(k, \hat{\mu})$ plays the r\^ole of
the (finite) group $\pic H$. The second statement in each of the following results
has the advantage of involving  only finitely many computations.

\begin{thm}   \label{thm.orbit.finite}
Let $G/k$ be a connected linear algebraic  group over a number field $k$.
Let $\mu  \subset G$ be a finite, commutative $k$-subgroup.
Assume ${\cyr X}^1(k,G)=0.$ 

(a) If $\{M_{v}\} \in X({\mathcal A}_k)$ is orthogonal to
the image of the   group $H^1(k, \hat{\mu})$ in $\br X$ with respect to the Brauer--Manin pairing, then there exist
$\{g_{v}\} \in G({\mathcal A}_k)$ and $M \in X(k)$ such that for each place $v$ of $k$
$$  g_{v}M=M_{v} \in X(k_{v}).$$

(b) Let $S$ be big enough for $(G,\mu)$. If
$\{M_{v}\} \in \prod_{v\in S}X(k_{v})$ is orthogonal to
the image of the   composite map $H^1_{\etale}(O_{S}, \hat{\mu}) \to H^1(k,\hat{\mu}) \to \br X$ with respect to the Brauer--Manin pairing (where the pairing is restricted to the places in $S$), then there exist
$\{g_{v}\} \in \prod_{v\in S}G(k_{v})$ and $M \in X(k)$ such that for each place $v \in S$
$$  g_{v}M=M_{v} \in X(k_{v}).$$ 
\end{thm}
\begin{proof} (a) Same as the proof of   Theorem  \ref{thm.orbit}, using diagram (\ref{bigfini}).

(b) Chasing through the  diagram (\ref{bigfinifini})
 one  first produces a class $\xi$ in $H^1_{\etale}(O_{S},\mu)$ whose image in $H^1_{\etale} (O_{S}, {\bf G})$
 has trivial image in each $H^1(k_{v},G)$ for $v \in S$. For $v \notin S$, the image of an element
 of $H^1_{\etale} (O_{S}, {\bf G})$ in $H^1(k_{v},G)$ is trivial by Lang's theorem and Hensel's lemma.
 Thus the image of $\xi$  in $H^1(k,G)$ has trivial image in each $H^1(k_{v},G)$, hence is trivial
 in $H^1(k,G)$. 
  This implies that
 the image of $\xi$ in $H^1(k,\mu)$ lies in the image of $X(k)$ and one concludes the argument 
 just as before.
\end{proof}

\begin{thm} \label{thm.sha.app.finite}
Let $G/k$ be a connected linear algebraic  group over a number field $k$.
Let $\mu  \subset G$ be a finite, commutative $k$-subgroup.
Assume ${\cyr X}^1(k,G)=0$ and assume that $G$ satisfies weak approximation.
 
 (a) Let $\{M_{v}\} \in X({\mathcal A}_k)$ be orthogonal to the image of  the   group $H^1(k, \hat{\mu})$
in $\br X$ with respect to the Brauer--Manin pairing. Then for each finite set $S$ of places
of $k$ and open sets $U_{v} \subset X(k_{v})$ with $M_{v } \in U_{v}$ there exists
$M \in X(k)$ such that $M \in U_{v}$ for $v \in S$.

(b) Let $S$ be big enough for $(G,\mu)$. If
$\{M_{v}\} \in \prod_{v\in S}X(k_{v})$ is orthogonal to
the image of the   composite map $H^1_{\etale}(O_{S}, \hat{\mu}) \to H^1(k,\hat{\mu}) \to \br X$ with respect to the Brauer--Manin pairing (where the pairing is restricted to the places in $S$), then for each family of
 open sets $U_{v} \subset X(k_{v})$ with $M_{v } \in U_{v}$ for $v\in S$ there exists
$M \in X(k)$ such that $M \in U_{v}$ for $v \in S$.
\end{thm}
\begin{proof}   This immediately follows from the previous theorem.
\end{proof}

On first reading the statement of the following theorem, the reader is invited to take $S$ to be just
 the set of archimedean places. 
 
\begin{thm} \label{thm.strongapp.finite}
 Let $k$ be a number field, $O$ its ring of integers, $S$ a finite set of places
of $k$ containing all archimedean places, $O_{S}$ the ring of $S$-integers.
Let $G/k$ be a semisimple,   simply connected group. 
Let $\mu \subset G$ be a finite, commutative $k$-subgroup.
Let ${\bf X} $ be a separated $O_{S}$-scheme of finite type such that $X={\bf X}\times_{O_{S}}k$
is $k$-isomorphic to $G/\mu$.

(a) For a point  $\{M_{v}\} \in    \prod_{v \in S} X(k_{v}) \times
\prod_{v \notin S}{\bf X}(O_{v}) \subset X({\mathcal A}_k) $   the following conditions
are equivalent :

(i) it is in the kernel of the map 
$X({\mathcal A}_k) \to \hom(\brun X, \Q/\Z)$;

(ii) it is in the kernel of the composite map
$$X({\mathcal A}_k) \to \prod'  H^1(k_{v},\mu) \to \hom(H^1(k,\hat{\mu}), \Q/Z).$$

 Let $S_{1}$ with   $S \subset S_{1}$ be  big enough for $(G,\mu)$ and ${\bf X}$, 
 more precisely assume that
$S_{1}$ contains $S$ and
there exists a semisimple $O_{S_{1}}$-group scheme ${\bf G}$ extending $G/k$, a
finite, commutative \'etale  subgroup  scheme $\mu \subset {\bf G}$ over $O_{S_{1}}$
extending $\mu$ and an  isomorphism of $O_{S_{1}} $-schemes 
${\bf X}\times_{O_{S}} O_{S_{1}} \simeq {\bf G}/  \mu$ 
 extending $X \simeq G/\mu$.  Then the above condtions imply :

 (iii)  the point $\{M_{v}\}_{v \in S_{1}}$ is in the kernel of the composite map
$$ \prod_{v \in S_{1}}X(k_{v} ) \to   \prod_{v\in S_{1}} H^1(k_{v}, \mu)      \to        \hom(H^1_{\etale}(O_{S_{1}},\hat{\mu}), \Q/\Z).$$

(b) Let $S_{0}  $ be a finite set of places of $k$ such that 
for each almost simple $k$-factor $G'$ of $G$ there exists a place $v \in S_{0} $ 
such that $G'(k_{v})$ is  not compact.  
Let
$S_{2}$ be a finite set of places of $k$.
 If 
 $\{M_{v}\} \in    \prod_{v \in S} X(k_{v}) \times
\prod_{v \notin S}{\bf X}(O_{v})$
  satisfies condition (i) or (ii)
 above, 
 then there exists 
  $M \in {\bf X}(O_{S_{0} \cup S})$ arbitrarily close to each $M_{v}$
for $v \in S_{2} \setminus S_{0}$. In particular ${\bf X}(O_{S \cup S_{0}}) \neq \emptyset$.

(c)  If the finite set $S_{1}$ of places is as above 
and contains $S_{0}$
and if   $\{M_{v}\} \in    \prod_{v \in S} X(k_{v}) \times
\prod_{v \notin S}{\bf X}(O_{v}) $ satisfies   condition (iii) above
then  there exists 
  $M \in {\bf X}(O_{S\cup S_{0}})$ arbitrarily close to each $M_{v}$
for $v \in S_{1}  \setminus S_{0}$. In particular ${\bf X}(O_{S\cup S_{0}}) \neq \emptyset$.
 \end{thm}

\begin{proof}

 (a)  This is   just a special case of Theorem  \ref{brauer.kottwitz.fini}.

(b)   We may assume that $S_{2}$ contains $S_{1}$. By Theorem \ref{thm.orbit.finite}   there exists $N \in X(k)$ and a family
$\{g_{v}\} \in G({\mathcal A}_k)$   such that for each place $v$ of $k$ we have
$  g_{v}N=M_{v} \in X(k_{v}).$ By the strong approximation theorem for $G$ with respect
to $S_{0}$, there exists $g \in {\bf G}(O_{S_{2} \cup S_{0}   })$ such that $g$ is very close to $g_{v} \in G(k_{v})$
for $v \in S_{2} \setminus S_{0}$.  We may thus arrange that the point $M=gN$ is 
very close to $M_{v}$ for $v \in S_{2} \setminus S_{0}$, and in particular lies in ${\bf X}(O_{v})$
for  each $v \in S_{2} \setminus (S \cup S_{0})$.
It also lies in ${\bf X}(O_{v})$ for  $v \notin S_{2} \cup S_{0}$. Hence it lies in ${\bf X}(O_{S \cup S_{0}})$.

(c) The proof here is more delicate than the proof of the   statement  (b) in Theorem \ref{obs.BM.strong}.
Let  $\{M_{v}\}  \in  \prod_{v \in S_{1}}X(k_{v})$ 
satisfy condition (iii).
By a theorem of Nisnevich \cite{Nisnevich}, 
the kernel of 
 the map
$$H^1_{\etale}(O_{S_{1}},{\bf G}) \to H^1(k,G)$$
is $H^1_{{\rm Zar}}(O_{S_{1}},{\bf G})$.
Under our assumptions on $G$,  we have ${\cyr X}^1(k,G)=0$  and
the group $G$ satisfies  strong approximation with respect to $S_{1}$:
the set $G(k).(\prod_{v \in S_{1}} G(k_{v}))$ is dense in $G({\mathcal A}_k)$.
By a theorem of Harder  (\cite{Harder}, Korollar 2.3.2 p.~179)
 this implies  $H^1_{{\rm Zar}}(O_{S_{1}},{\bf G})=0$. 

For $v \notin S_{1}$, we have $H^1_{\etale}(O_{v},{\bf G})=0$ (Hensel's lemma and
Lang's theorem).
Thus  the kernel of the map
$H^1_{\etale}(O_{S_{1}},{\bf G}) \to \prod_{v \in S_{1}}H^1(k_{v},G)$
is the same as the kernel of the map
$H^1_{\etale}(O_{S_{1}},{\bf G}) \to \prod_{v }H^1(k_{v},G)$
and under our assumptions, by the above argument, that kernel is trivial
 (note in passing that the only places where $H^1(k_{v},G)$ may not be trivial
are the real places of $k$).
Chasing through diagram (\ref{bigfinifini}) with $S$ replaced by $S_{1}$
one finds that there exist a point  
 $N \in {\bf X}(O_{S_{1}})$ and a family $\{g_{v}\} $ in $  \prod_{v \in S_{1}} G(k_{v})$ such that  $g_{v}.N=M_{v}$ for each $v \in S_{1}$.
By strong approximation with respect to  $S_{0}$ there exists $g \in {\bf G}(O_{S_{1}}) \subset G(k)$  
such that $g$ is very close to
$g_{v} \in G(k_{v})$ for $v \in S_{1} \setminus S_{0}$  and  $g \in  {\bf G}(O_{v})$ for  $v \notin S_{1}$.
Now the point $M=g N \in X(k)$
 belongs to ${\bf X}(O_{v})$ for   $v \notin S_{1}$
  and it is very close to $M_{v}\in {\bf X}(k_{v})$ for $v \in S_{1} \setminus S_{0}$,
  hence lies in ${\bf X}(O_{v})$ for $v \notin SÊ\cup S_{0}$.
 \end{proof}

\rem{}
 Under the assumption on $G$ made in (c), the proof of the above theorem shows
 -- in a very indirect fashion -- that  
 the existence of a point $\{M_{v}\}_{v \in \Omega_{k}}$  as in (i) or  (ii)  is equivalent to the existence of a point  $\{M_{v}\}_{v \in S_{1}}$ as in (iii). 
 \endrem

The group $H^1(k,\hat{\mu})$ is in general infinite, hence the conditions appearing  in (i) and (ii)
 do not lead to a finite decision process for the existence of $S$-integral points.
 However for any finite set $S \subset \Omega_{k}$ such that $\mu$ and its dual are finite \'etale over
 $O_{S}$, 
  the group $H^1_{\etale}(O_{S},\hat{\mu})$ is a finite group: this  is a consequence of Dirichlet's theorem on units and of the finiteness of the class number of number fields. Thus for $S=S_{1}$ as in (c),
  only finitely many computations are required to decide whether there exists a family
  $\{M_{v}\}_{v \in S_{1}}$ as in (iii), hence ultimately to decide if there exists an $S$-integral point on $\bf X$,
  which additionally may be chosen
  arbitrarily close to each $M_{v} \in X(k_{v})$ for $v \in S \setminus S_{0}$.

\section{Representation of a quadratic form by a  quadratic form over a   field} \label{sec.rep.quad.field}

\subsection{}
\label{5.1}
Let $k$ be a field of characteristic different from 2. 
Let $n\leq m$ be natural integers.
 A classical problem asks for the representation of a nondegenerate quadratic  form $g$ over $k$,
 of rank $n \geq 1 $, by  a nondegenerate quadratic form $f$ of rank $m \geq 2$,  over $k$, i.e. one looks for linear forms $l_{1}, \dots, l_{m}$ with coefficients in $k$ in the  variables $x_{1}, \dots, x_{n}$, such that $$g(x_{1},\dots, x_{n})=f(l_{1}(x_{1},\dots, x_{n}),\dots,l_{m}(x_{1},\dots, x_{n})).$$
This equation  in the coefficients of the forms $l_{i} $ defines an affine $k$-variety $X$. 

In an equivalent fashion, $X$ is the variety of  linear maps of $W=k^n$
into $V=k^m$ such that the quadratic form $f$ on $V$ induces the quadratic
form $g$ on $W$. The linear map is then necessarily an embedding.

Let $B_{f}(v_{1},v_{2})$ be the symmetric bilinear form on $V$ such that $B_{f}(v,v)=f(v)$
 for $v \in V$. Thus
 $B_{f}(v_{1},v_{2})=(1/2) (f(v_{1}+v_{2})-f(v_{1}) -f(v_{2})).$
In concrete terms, a $k$-point of $X$ is given by a set of $n$ vectors $v_{1},\dots,v_{n} \in V=k^m$,
such that the bilinear form $B_{f}$  satisfies: the matrix $$B_{f}(v_{i},v_{j})_{ i=1,\dots, n; j=1,\dots, n}$$ is the matrix of the bilinear form on $W=k^n$ attached to $g$.

The $k$-variety $X$ has a $k$-point if and only if there exists a nondegenerate quadratic form $h$ over $k$ in $m-n$ variables and an isomorphism of quadratic forms $f \simeq g \perp h$ over $k$. The quadratic form $h$ is then well defined up to (nonunique) isomorphism (Witt's cancellation theorem).

By another of Witt's theorems (\cite{O'Meara}, Thm. 42:17 p. 98),
over any field $K$ containing $k$
the set $X(K)$  is empty or a homogeneous space
of $O(f)(K)$, and the stabilizer of a point of $X(K)$ is isomorphic to the group $O(h_K)(K)$,
where $h_{K}$ is a nondegenerate quadratic form over $K$ such that $f_{K} \simeq  g_{K} \perp h_{K}$.
Thus the $k$-variety $X$ is a homogeneous space of the $k$-group $O(f)$.
If $X(k)\neq \emptyset$  the stabilizer of a 
$k$-point of $X$,   up to nonunique isomorphism, does not depend on the $k$-point, it is $k$-isomorphic to the $k$-group $O(h)$ for   a quadratic form $h$ over $k$ as above. 

For $n < m$, the group $SO(f)(k)$ acts transitively on $X(k)$. If $X(k) \neq \emptyset$, the stabilizer of
of a  $k$-point of $X$,  up to nonunique isomorphism, does not depend on the $k$-point, it is $k$-isomorphic to the $k$-group $SO(h)$ for   a quadratic form $h$ as above. For $n=m-1$, the $k$-variety $X$
is a principal homogeneous space of $SO(f)$.

For $n=m$, the $k$-variety $X$ is a principal homogeneous space of $O(f)$.
If it has a $k$-point, it breaks up into two connected components $X'$ and $X''$, each of which is
a principal homogeneous space of $SO(f)$.

Let us fix a $k$-point of $X$, that is an embedding $\lambda: (W,g)  \hookrightarrow (V,f)$ of quadratic spaces as above.
If $n=m$, the $k$-point determines a connected component of $X$,
say $X'$.
Such a $k$-point  $M \in X(k)$ defines a $k$-morphism $SO(f) \to X$
which sends $\sigma \in SO(f)$ to $\sigma \circ \lambda$. 
For $n=m$ this factorizes through a $k$-morphism $SO(f) \to X'$.
By Witt's results,
for $n <m$
over any field $K$ containing $k$, the induced map $SO(f)(K) \to X(K)$ is onto.
For $n=m-1$ it is a bijection.
For $n=m$ the map $SO(f)(K) \to X'(K)$  is a bijection.

\bigskip

In this paper we restrict attention to $m \geq 3$, which we henceforth assume.
We then have the exact sequence 
\begin{equation}\label{spinsequence}
1 \to \mu_{2} \to Spin(f) \to SO(f) \to 1.
\end{equation}
Let $G=Spin(f)$.
Assume $char(k)=0$.  Proposition  \ref{pic.br.simply.connected}  gives 
 $k^*=k[G]^*$, $\pic G=0$ and $\br k = \br G$ .
 
We thus  have
$k^*=k[SO(f)]^*$ and $\k^*=\k[SO(f)]^*.$
  Proposition \ref{Kraft}
gives natural isomorphisms
$$\nu(Spin(f))(h)  : \Z/2 \oi \pic SO(f) \hskip1cm   \nu(Spin(f)) :  \Z/2 \oi \pic SO(f)\times_{k}\k.$$ 

This induces an isomorphism
$ k^*/k^{*2} = H^1(k,\Z/2) \oi H^1(k,\pic SO(f)\times_{k}\k)$
which   combined with the inverse of the isomorphism
$\brun X/\br k \oi H^1(k,\pic SO(f)\times_{k}\k)$
yields an isomorphism
$$ k^*/k^{*2} \simeq \brun X/\br k.$$
This isomorphism is the one provided by Proposition \ref{picbrHdisconnected}.

Given a $k$-isogeny $1 \to \mu \to G \to G' \to 1$ with $G$ semisimple simply connected,
Proposition \ref{pic.br.simply.connected} and the Hochschild-Serre spectral sequence $$E_{2}^{pq}=H^p(\hat{\mu}(\k),H^q({\overline G},\G_{m})) \Longrightarrow H^n({\overline G}',\G_{m})$$ yield  an isomorphism $H^2(\hat{\mu}(\k),\k^*) \simeq \br {\overline G}'$. If the group $\hat{\mu}(\k)$ is cyclic,  as is the case here, this implies
$\br {\overline G}'=0$. 
Thus $$\brun SO(f)= \br SO(f).$$

\bigskip
\subsection{}
\label{5.2}

Applying Galois cohomology to   sequence (\ref{spinsequence})
and using Kummer's isomorphism
we get  a homomorphism of groups, the spinor norm map 
$$\theta: SO(f)(k) \to k^*/k^{*2}.$$

There are various ways to compute this map. One is well known:
any element $\sigma \in SO(V)(k)$ is a product of  of
reflections with respect to an even  number of anisotropic vectors $v_{1}, \dots, v_{n}$.
The product $\prod_{i} f(v_{i}) \in k $ is nonzero, its class in $k^*/k^{*2}$
is equal to  $\theta(\sigma)$.

The following result, due to  Zassenhaus, is quoted in \cite{O'Meara}, p.~137.
We thank P.~Gille for help with the proof.
\begin{prop}
For $\tau \in SO(V)(k) \subset GL(V)(k)$ such that $\det(1+\tau) \neq 0$,
we have $\theta(\tau)=\det((1+\tau)/2) \in k^*/k^{*2}$.
\end{prop}
\begin{proof}
Let $\sigma: End(V) \to End(V)$ denote the adjoint involution attached to the quadratic form $f$
(if one fixes a basis $V=k^m$ and $A$ is the matrix of $f$ in this basis, then for $M\in M_{m}(k)$ we have $\sigma(M)=A^{-1}.{}^tM.A$).

We have the following inclusions of $k$-varieties
$$ U \subset SO(f) \subset O(f) \subset GL(V) \subset End(V),$$
where $O(f) = \{a \in End(V), \sigma(a).a=1\}$, and $U = \{a \in O(f),  (1+a) \in GL(V)\}$.
The open set $U \subset O(f)$ is contained in the irreducible open set $SO(f) \subset O(f)$.
 
We also have the following inclusions of $k$-varieties
$$W \subset Alt(f) \subset End(V)$$
where $Alt(f) = \{b \in End(V), \sigma(b)+b=0\}$ and $$W =   \{b \in End(V), \sigma(b)+b=0, 1+b \in GL(V)\}.$$

The $k$-variety $Alt(f)$, which is the Lie algebra of $SO(f)$, is an affine space $\A^{m(m-1)/2}$.
One checks 
that the polynomial function $\det(1+b)$ on $End(V)$ induces a nonzero (geometrically) irreducible function $h_{W}$ on $Alt(f)\simeq \A^{m(m-1)/2}$. Thus the open set $W \subset Alt(f)$
satisfies $\pic(W)=0$ and $k[W]^*=k^*. h_{W}^{\Z}$.

The maps $b \mapsto  (1-b)(1+b)^{-1}$, resp.
 $a \mapsto (1+a)^{-1}(1-a)$, define $k$-morphisms $W \to U$, resp. $U \to W$,
which are inverse of each other (this is the well known Cayley parametrization
of the orthogonal group).
On $U \subset SO(f) $, the
invertible function $h_{U}$ which is the inverse image of $h_{W}$
  sends $\tau \in U \subset SO(f)$ to $ \det(2(1+\tau)^{-1})$. We have
  $\pic(U)=0$ and $k[U]^*=k^*. h_{U}^{\Z}$. From the Kummer sequence
  in \'etale cohomology we have $k[U]^*/k[U]^{*2} \simeq H^1_{\etale}(U,\mu_{2})$.
  Thus any \'etale $\mu_{2}$-cover of $U$ is given by an equation
  $c.h_{U}^r=z^2$, with $c \in k^*$ and $r=0$ or $1$. If the total space of the cover is
  geometrically irreducible, then $r=1$. 
   The restriction  of the \'etale $\mu_{2}$-cover $Spin(f) \to SO(f)$ on
$U$ is thus given by an equation $c.h_{U}=z^2$. The point $\tau=1$ belongs to $U$
and $h_{U}(1)=1$. There exists a $k$-point in $Spin(f)$ above $\tau=1 \in SO(f)$.
Thus the restriction of $c.h_{U}=z^2$ on $\tau=1$, which is simply $c=z^2$, has a $k$-point.
Hence $c$ is a square in $k$.

The spinor map $\theta: SO(f)(k) \to k^*/k^{*2}$ thus restricts to the map $U(k)  \to k^*/k^{*2}$
induced by $\tau \mapsto  \det(2(1+\tau)^{-1})$.
\end{proof}

\bigskip
\subsection{}
\label{m-n atleast 3}
If  $m-n \geq 3$,  the geometric stabilizers of the $ Spin(f)$ action on $X$ are 
of the shape $Spin(h)$ for $h$ a quadratic form over ${\overline k}$ of rank $m-n \geq 3$.
Assume $X(k) \neq \emptyset$. The stabilizer $H \subset G=Spin(f)$ of this $k$-point is then isomorphic
to a $k$-group $H=Spin(h)$ for $h$ a quadratic form over $k$ of rank $m-n \geq 3$.
Assume $char(k)=0$. Then $\hat{H}=0$, $\pic H=0$ and $\br k = \br H$ (Proposition  \ref{pic.br.simply.connected}). From Proposition \ref{prop.Sansuc}
we conclude
 $$k^*=k[X]^*, \hskip1cm \pic X =0, \hskip1cm \br k = \br X.$$

\bigskip
\subsection{}
\label{m-n=1}

 If $ m-n=1$,
  the stabilizers of the $SO(f)$-action on $X$ are trivial, $X$ is a principal homogeneous space of $SO(f)$, the geometric stabilizers of the $Spin(f)$ action on $X$ are 
of the shape $\mu_{2}$.  
Assume $X$ has a $k$-point, and fix such a $k$-point. This determines an isomorphism
of $k$-varieties  $  \phi: SO(f)  \buildrel \simeq \over \rightarrow
 X $ sending $1$ to the given $k$-point.
The composite map $Spin(f) \to SO(f) \to X$ makes $Spin(f)$ into a $\mu_{2}$-torsor over $X$.
As explained in subsection \ref{5.1}
 we then get  $$k^*=k[X]^* , \hskip5mm  \Z/2= \pic X$$ and 
$$  k^*/ k^{*2}=H^1(k,\Z/2) \simeq  \brun X/ \br k = \br X/ \br k.$$
 Let $\xi \in H^1_{\etale}(X,\mu_{2})$ be the class of the $\mu_{2}$-torsor $\Spin(f) \to X$.
 There is an associated map $\psi: X(k) \to H^1(k,\mu_{2}) = k^*/k^{*2}$.
The composite map $SO(f)(k) \to X(k)  \to k^*/k^{*2}$ is the spinor map  $\theta: SO(f)(k) \to   k^*/k^{*2}$.

For any $\alpha \in k^*/k^{*2}=H^1(k,\Z/2)$ we have the cup-product $\xi \cup \alpha \in H^2_{\etale}(X,\mu_{2})$
and we may consider the image  $A_{\alpha} \in {}_{2}\br X$ of this element  in the 2-torsion subgroup of $\br X$
under the natural map induced by $\mu_{2} \to \G_{m}$.
Using Proposition \ref{prop.tautology} 
and Proposition \ref{homologicalalgebra},  
one sees 
 that the map $\alpha \mapsto A_{\alpha}$ induces the  isomorphism 
 $k^*/k^{*2} \simeq \br X/ \br k$ described  above, and that for any $k$-point $N \in  SO(f)(k) =X(k) $, and any  $\alpha \in k^*/k^{*2}$,
 the evaluation of $A_{\alpha}$ at $\phi(N) \in X(k)$ is the class of the quaternion algebra 
 $(\psi(\phi(N)),\alpha)= (\theta(N),\alpha) \in \br k$.

\bigskip
\subsection{}
\label{m-n=0}

If $m=n$  and $X(k) \neq \emptyset$, then subsection 5.4  applies  to each of the two connected
components $X'$ and $X''$ of $X$.
 
\bigskip
\subsection{}
\label{m-n=2}

 If $m-n=2$, the geometric stabilizers, be they for the $SO(f)$ action or the
$Spin(f)$-action on $X$, are 1-dimensional tori. 
Assume $X$ has a $k$-point $M$. This fixes a morphism  $\phi: SO(f) \to X$.
This also corresponds to a decomposition
$f \simeq g \perp h$
for some 2-dimensional quadratic form $h$ over $k$.
The stabilizer of the $k$-point $M$ for the $SO(f)$-action is the $k$-torus $T_{1}=R^1_{K/k}\G_{m}$, where $K=k[t]/(t^2-d)$ and $\disc(f)=-\disc(g).d$. If $d$ is a square in $k$ then $T_{1} \simeq \G_{m,k}$.
The stabilizer of the $k$-point $M$ for the $Spin(f)$-action is a 
$k$-torus $T $ which fits into an exact sequence
$$1 \to \mu_{2} \to T  \to T_{1} \to 1.$$
This implies that the $k$-torus $T_{1}$ is $k$-isomorphic to the $k$-torus $T$.
For clarity, we keep the index $1$ for the torus $T_{1} $.
Thus from the $k$-point $M \in X(k)$ one builds a commutative diagram

\[
\begin{CD}
         1         @>>>   \mu_{2}                  @>>>    T              @>>>         T_{1}       @>>>   1\\
                   @.               @|                       @VVV                                       @VVV                        @.  \\
         1         @>>>  \mu_{2}                   @>>>   Spin(f)        @>>>       SO(f)      @>>>    1\\
        @.                          @.                           @VVV                                @VVV                    @.   \\
       @.                                             @.  X                            @=    X,         @.       @  .     
           \end{CD}
\]

\noindent where the   bottom vertical maps define torsors and the horizontal sequences are exact sequences of
algebraic groups over $k$.

By Proposition \ref{pic.br.simply.connected}, we have $k^*=k[Spin(f)]^*$, $\pic Spin(f)=0$, $\br k= \br Spin(f)$ and the analogous statements over $\k$.

This immediately implies $k^*=k[X]^*$ and $\k^*=\k[X]^*$.

Applying  Proposition \ref{prop.Sansuc} or Proposition \ref{Kraft} 
to the $T$-torsor $Spin(f) \to X$,
we get isomorphisms
$$\nu(Spin(f))(k) : {\hat T}(k) \oi \pic X \hskip1cm    \nu(Spin(f)) : { \hat T} \oi \pic {\overline X}.$$

The latter map induces an isomorphism
$ H^1(k,\hat{T}) \oi  H^1(k,\pic {\overline X}).$
If we compose this isomorphism with the inverse of the isomorphism
 $\brun X/\br k \oi H^1(k,\pic {\overline X})$ coming from Lemma \ref{Leray}, we get
 an isomorphism $ H^1(k,\hat{T}) \oi \brun X/\br k $ which is the one in Proposition \ref{picbrHdisconnected},
  i.e. is induced by  $\theta(Spin(f)) : H^1(k, \hat{T}) \to \brun X$.
 
Since $T$ is a torus, it is a connected group  and $\pic {\overline T}=0$. Proposition \ref{prop.connectedtautology},
shows that the map
  $\theta(Spin(f)) : H^1(k, \hat{T}) \to \brun X$ factorizes as
$$  H^1(k, \hat{T}) \oi  \pic T \to  \brun X.$$

If $d$ is a square, we get $$\Z \simeq \pic X, \hskip1cm \br X/ \br k=0.$$

 If $d$ is not a square in $k$,
then $$\pic X=0, \hskip1cm \Z/2 \simeq \brun X/ \br k \simeq \br X/ \br k.$$

Indeed, in the first case, $\hat T=\Z$ with the trivial Galois action, thus $\pic T \simeq H^1(k,\hat T)=0$.
In the second case, $\hat T=\Z[G]/\Z(1+\sigma)$ where $G={\rm Gal}(K/k)=\{1,\sigma\}=\Z/2$,
hence $\pic T\simeq  H^1(k,\hat T)=H^1(G, \hat T)=\Z/2$.

\bigskip

Taking Galois cohomology one gets the commutative diagram
\[
\begin{CD}
 \theta:  SO(f)(k)     @>>>    k^*/k^{*2}\\
 \phi_{k}   @VVV                 @VVV    \\
  \psi:   X(k)       @>>>  H^1(k,T) = k^*/N_{K/k}K^*\\
  @VVV                 @VVV    \\
  {\rm Hom}(\br X, \br k) @>>>  {\rm Hom}(H^1(k,{\hat T}),\br k),
  \end{CD}
\]
where the bottom right hand side vertical map is given by cup-product
and the bottom horizontal map is  induced by  $$\theta(Spin(f)) : H^1(k,{\hat T})  \oi  \pic T \to \brun X \subset \br X.$$

Let us check that this diagram is commutative. Given a point in $a \in SO(f)(k)$,
one lifts it to $b \in Spin(f)(k_s)$, where $k_s$ is a separable closure of $k$. The $1$-cocycle $\sigma \mapsto {}^{\sigma}b.b^{-1} \in \mu_2$
defines a class in $H^1(k,\mu_2)=k^*/k^{*2}$ which is exactly 
$\theta(a)$, i.e. the image of the spinor map.
On the other hand the image $c \in X(k_s)$ of $b$ under the map $Spin(f)(k_s) \to X(k_s)$
is precisely the same as the image of $a$ under the map $SO(f)(k)\to X(k)$.
Thus the image of $c$ in $H^1(k,T)$ under $\psi$ is given by the
class of the cocycle ${}^{\sigma}b.b^{-1}$ viewed in $T(k_s)$ rather than 
in $\mu_2 \subset T(k_s)$. That is, the top diagram is commutative.
The commutativity of the bottom square is a special case of Proposition 2.7.

The natural map
$ \psi: X(k) \to  k^*/NK^*$ associated to the torsor $Spin(f) \to X$ under the $k$-torus $T$
can thus be defined in a more concrete fashion. By Witt's theorem a point  in $X(k)$ may be lifted to some element $\sigma$
in $SO(f)(k)$. One may then send this element $\sigma$ to $k^*/k^{*2}=H^1(k,\mu_{2})$ using the middle horizontal exact sequence.   That is, one sends $\sigma$ to its spinor norm $\theta(\sigma)
\in k^*/k^{*2}$.
The top horizontal sequence defines a map $H^1(k,\mu_{2}) \to H^1(k,T )$,
which may be identified with  the obvious map $k^*/k^{*2} \to k^*/NK^*$.
Using this map, one gets an element in $H^1(k,T )$ which one immediately checks does not depend on the choice of the lift $\sigma$ in $SO(k)$. 
 
\bigskip

Let us assume that $d$ is not a square in $k$.  The torsor $Spin(f) \to X$ is associated to the choice of
a $k$-point $M$ of $X$. The above discussion yields a map $\Z/2=\pic T \to \br X$. The image of $1 \in \Z/2$
is the class of an element   $\alpha \in \br X$ which is trivial at $M$, vanishes when pulled back to
$Spin(f)$ and also vanishes when pulled back to $\br X_{K}$. There thus exists
a rational function $\rho \in k(X)^*$ whose divisor on $X$ is the norm of a divisor on $X_{K}$
and such that the image of $\alpha$ under the embedding $\br X \hookrightarrow \br k(X)$ is the
class of the quaternion algebra $(K/k,\rho)$. Let $U$ be the complement of the divisor of $\rho$.
On the subset $U(k) \subset X(k)$    the map $X(k) \to \br k$ defined by $\alpha$
is induced by the evaluation of the function $\rho$, which yields a map $U(k) \to k^*/N_{K/k}K^* \subset \br k$. In order to implement the results of the previous section it is thus useful to compute such a  function $\rho$. Here is a general way to do it. Let $F=k(X)$ be the function field of $X$. By Witt's theorem the map
$SO(f)(F) \to X(F)$ is onto. One may thus lift the generic point of $X$ to an $F$-point $\xi \in SO(f)(F)$,
which one may write as an even product of reflections $\tau_{v_{i}} $ 
with respect to anisotropic 
vectors $v_{i}$ with $F$-coordinates.
One  computes the image of $\xi \in SO(f)(F)$ in $H^1(F,\mu_{2})=F^*/F^{*2}$ under the boundary map, that is one computes the spinor norm of $\xi$.
The image of $\xi$ is thus the class of the product
$\prod_{i} f(v_{i}) \in F^*$. This product yields a desired function $\rho$.

For later use, it will be useful to give complete recipes for the computation of
the map $X(k) \to k^*/NK^*$. 

\subsection{}\label{spinnormcomputation}

We start with the general case $m=n+2$.
We fix a $k$-point $M \in X(k)$.  
As recalled above,
this is equivalent to giving $n$ vectors $v_{1},\dots,v_{n} \in V=k^m$ such that
the matrix $\{B_{f}(v_{i},v_{j})\}_{i=1,\dots,n; j=1,\dots, n}$ gives the coefficients of the
quadratic form $g(x_{1},\dots,x_{n})$ on $W=k^n$.
We may and shall assume $f(v_{1}) \neq 0$.
 Let us henceforth  write
$B(x,y)=\ B_{f}(x,y) $.

Let now $P$ be an arbitrary $k$-point of $X$, given by a linear map from $W=k^n$ to $V=k^m$
compatible with the bilinear forms. Let $w_{1},\dots,w_{n} \in k^m$  be the image of
the standard basis of $W$.
 There exists $\sigma \in SO(f)(k)$ such that $\sigma(M)=P$, i.e.
$\sigma(v_{i})=w_{i}$ for each $i=1,\dots,n$.
Let $\tau_y$ be the reflection along the vector $y\in V$ with
$f(y)\neq 0$ which is given by
$$ \tau_y (x)= x- 2\frac{B(x,y)}{f(y)}y . $$

Over a Zariski open set $U$ of $SO(f)$ such that all the following
related reflections are  defined, we define $\sigma_1=\sigma$
and $\sigma_2=\tau_{v_1}\tau_{\sigma_1 v_1+v_1}\sigma_1$ if $n$ is
odd and $\sigma_2=\tau_{\sigma_1 v_1-v_1}\sigma_1$ if $n$ is even.
Let $$\sigma_3=\tau_{\sigma_2 v_2-v_2}\sigma_2, \ \ \cdots ,
 \ \ \sigma_{n }=\tau_{\sigma_{n-1}
v_{n-1}-v_{n-1}}\sigma_{n-1}$$ inductively. Let us prove
\begin{equation}\label{vanishing}
B(v_i,
\sigma_j v_j-v_j ) =0 .
\end{equation}
for all $j>i$  with $1\leq i\leq n$.
Indeed, if $j=i+1>2$ or $n$ is even, then 
$$  
\aligned &
B(v_i, \ \sigma_{i+1}v_{i+1}-v_{i+1} ) =B(v_i, \ \tau_{\sigma_i
v_i-v_i} \sigma_i v_{i+1} ) - B( v_i, v_{i+1}) \cr = & 
B( \tau_{\sigma_i v_i- v_i} v_i, \ \sigma_i v_{i+1} )  - B(v_i, v_{i+1})
 = B( \sigma_i v_i,  \sigma_i v_{i+1} ) - B( v_i, v_{i+1})=0. 
\endaligned 
$$
If $j=i+1=2$ and $n$ is odd, then $$
\aligned &
B(v_i, \ \sigma_{i+1}
v_{i+1}-v_{i+1})=B(v_1, \tau_{v_1}\tau_{\sigma_1 v_1+v_1}\sigma_1 v_2
-v_2) \cr = & B(\tau_{\sigma_1 v_1+v_1}\tau_{v_1}v_1, \sigma_1
v_2)-B(v_1,v_2)=0.
\endaligned $$

Suppose (\ref{vanishing})  is true for the case which is less than $j$. Then
$$ \aligned  & B( v_i,  \ \sigma_j v_j-v_j )= B(
v_i, \  \tau_{\sigma_{j-1}v_{j-1}-v_{j-1}} \cdots \tau_{\sigma_i
v_i-v_i}\sigma_i v_j ) - B( v_i, v_{j} ) \cr = &
B( \tau_{\sigma_iv_i-v_i} \cdots
\tau_{\sigma_{j-1}v_{j-1}-v_{j-1}} v_i , \ \sigma_i v_j ) -
B( v_i, v_{j} ) .
\endaligned
$$
By the induction hypothesis, one has $$\tau_{\sigma_iv_i-v_i} \cdots
\tau_{\sigma_{j-1}v_{j-1}-v_{j-1}} v_i = \tau_{\sigma_i v_i-v_i}
v_i= \sigma_i v_i $$ and (\ref{vanishing})  follows.
From (\ref{vanishing})  we deduce $$\sigma v_i = \tau_{\sigma_1 v_1-v_1}
\tau_{\sigma_2 v_2-v_2} \cdots \tau_{\sigma_n v_n-v_n} v_i $$ for
all $1\leq i\leq n$ if $n$ is even and $$\sigma v_i =
\tau_{v_1}\tau_{\sigma v_1+v_1} \tau_{\sigma_2 v_2-v_2} \cdots
\tau_{\sigma_n v_n-v_n} v_i
$$ for all $1\leq i\leq n$ if $n$ is odd. 

For $n$ even, the even product of reflections 
  $\tau_{\sigma_1 v_1-v_1}
\tau_{\sigma_2 v_2-v_2} \cdots \tau_{\sigma_n v_n-v_n}$  is a lift of $P$ under
the map $SO(f) \to X$ associated to $M$. For $n$ odd, 
the even product of reflections 
 $\tau_{v_1}\tau_{\sigma_1 v_1+v_1} \tau_{\sigma_2 v_2-v_2}
\cdots \tau_{\sigma_n v_n-v_n}$   is a lift of $P$ under
the map $SO(f) \to X$ associated to $M$.

The spinor norm of this lift
 is thus the class of the product
$\prod_{i=1}^n f(\sigma_i v_i-v_i)\in k^*/k^{*2} $ if $n$ is even or
$f(v_1)f(\sigma v_1+v_1)\prod_{i=2}^n f(\sigma_i v_i-v_i)\in
k^*/k^{*2} $ if $n$ is odd. Hence the image of $P\in X(k)$ in
$k^*/NK^*$ is the class
$$\prod_{i=1}^n f(\sigma_i v_i-v_i)\in k^*/N_{K/k}K^*$$
if $n$ is even or $$f(v_1)f(\sigma v_1+v_1) \prod_{i=2}^n f(\sigma_i
v_i-v_i)\in k^*/N_{K/k}K^*$$ if $n$ is odd.
\bigskip

\subsection{} \label{subsec.affineconic}
We  now specialize to the case $m=3$, $n=1$. This is the classical problem of representing an
element $a \in k^*$ by a ternary quadratic form $f(x,y,z)$. The $k$-variety $X$
is the affine quadric given by the equation
$$f(x,y,z)=a.$$
Assume $d=-a.\disc(f)$ is not a square and $X(k)\neq \emptyset$. 
Let $K=k(\sqrt{d})$.
The general considerations above, or direct ones,
show that $\br X/\br k = \brun X/ \br k$ has order 2. 
Here is a more direct way to produce
a function $\rho$ with divisor a norm for the extension $K/k$, such that the
quaternion algebra $(\rho,d) \in \br k(X)$  comes from $\br X$ and yields a generator of
$\br X/\br k$.
Let $Y \subset \P^3_{k}$ be the smooth projective quadric given by the homogeneous equation
$$f(x,y,z)=at^2.$$
Suppose given a $k$-rational point $M$ of $Y$. Let $l_{1}(x,y,z,t)$ be a linear form with coefficients in 
$k$ defining the tangent plane to $Y$ at $M$. There then exist   linear forms
$l_{2}, l_{3}, l_{4}$, a constant $c \in k^*$ and an identity
$$f(x,y,z)- at^2  =l_{1}.l_{2} +c(l_{3}^2 -d l_{4}^2).$$ 
Such linear forms (and the constant $c$) are easy to determine.
The linear form $l_{i}$ are linearly independent. Conversely if we have such an identity
$l_{1}=0$ is an equation for the tangent plane at the $k$-point $l_{1}=l_{3}=l_{4}=0$.
Define $\rho=l_{1}(x,y,z,t)/t \in k(X)$.
Consider the quaternion algebra
$\alpha=(\rho,d)=(l_{1}(x,y,z,t)/t,d) \in \br k(X)$.
We have $(l_{1}(x,y,z,t)/t,d)=(l_{2}(x,y,z,t)/t,d)  \in \br k(X)$. 
Thus $\alpha$ is unramified on $X_{k}$ away from the plane at infinity $t=0$,
and the  finitely many closed points given by $l_{1}=l_{2}=0$. By the purity theorem for
the Brauer group of smooth varieties (\cite{Gro}, II, Thm. 2.1 and III, Thm. 6.1),
we see that this class is unramified on the affine quadric $X $,
i.e. belongs to $\br X \subset \br k(X)$.
The complement of $X_{k}$ in $Y$ is the smooth projective conic $C$ over $k$ given by $q(x,y,z)=0$.
An easy computation shows that the residue of $\alpha$ at the
generic point of this conic is the class of $d$ in $k^*/k^{*2}=H^1(k,\Z/2) \subset H^1(k(C),\Z/2)
\subset H^1(k(C),\Q/\Z)$
(note that $k$ is algebraically closed in $k(C)$).
Since $d$ is not a square in $k$, this class is not trivial. Thus $\alpha=(\rho,d) \in  \br X \subset \br  k(X)$ does not lie in the image of $\br k$. It is thus a generator of $\br X/ \br k$.
Note that at any $k$-point of $X$, either $l_{1}$ or $l_{2}$ is not zero. The map
$X(k) \to \br k$ associated to $\alpha$ can thus be computed by means of the map
$X(k) \to k^*/N_{K/k}K^*$ given by either the function $\rho=l_{1}(x,y,z,t)/t$ or the function
$\sigma=l_{2}(x,y,z,t)/t$.

For later use, let $(V,Q)$ denote the 3-dimensional quadratic space which in the given basis
$V=k^3$ is defined by $Q(u)=f(x,y,z)$ for $u=(x,y,z)$. Let 
$B(u,v) = (1/2)(Q(u+v)-Q(u)-Q(v))$ be the associated bilinear form.
Let $v_{0} \in V$ correspond to a point $M$ of the affine quadric $f(x,y,z)=a$.
Then the affine linear map $\rho: V \to k$ is given by $v \mapsto B(v_{0},v)-a$.
Thus on the open set $ B(v_{0},v)-a \neq 0$ of $X$ the restriction of $\alpha$
is given by the quaternion algebra $ (B(v_{0},v)-a,-a.\disc(f))$.

 \section{Representation of a quadratic form by a  quadratic form over  a ring  of integers} \label{sec.rep.quad.integers}

 Let $k$ be a number field, $O$ its ring of integers. Let $f$ and $g$ be quadratic forms over $O$.
 Assume $g_{k}$ and $f_{k}$ are nondegerate, of respective ranks $n\geq 1$ and $m \geq  n$.
 
 A classical problem raises the question of the representability of $g$ by $f$, i.e. the existence of
 linear forms $l_{1}, \dots, l_{m}$ with coefficients in $O$ in the  variables $x_{1}, \dots, x_{n}$, such that one has the identity $$g(x_{1},\dots, x_{n})=f(l_{1}(x_{1},\dots, x_{n}),\dots,l_{m}(x_{1},\dots, x_{n})).$$
 Such an identity corresponds to a point with $O$-coordinates of a certain $O$-scheme ${\bf X}$.
 
 There are variants of this question. For instance, when $g$ is of rank one, i.e. of the shape $ax^2$,
 in which case one simply  asks for the existence of an integral  point $y_{i}=b_{i}, i=1,\dots, n $ of
 the scheme
 $$a=f(y_{1},\dots,y_{n}),$$ 
  one sometimes demands that the ideal spanned by the $b_{i}$ be the whole ring $O$ (this is
  a so-called primitive solution of the equation).
  This simply corresponds to choosing a different $O$-scheme ${\bf X}$,
  but one with the same generic fibre $X={\bf X}\times_{O}k$. More precisely one takes the new 
  $O$-scheme to be  the complement of the closed set $y_{1}=\dots=y_{n}=0$ in the old ${\bf X}$.
  
  In the case $n=m$ and $X(k) \neq \emptyset$ we shall replace the natural $X$, which is disconnected, by
  one of its connected components  over $k$, and we shall consider $O$-schemes $\bf X$ with generic fibre this component.

 \medskip
 
 Quite generally the following problem may be considered.
 
 \medskip
 
 {\bf Problem} {\it Let    $k$ be a number field, $O$ its ring of integers. Let $f$ and $g$ be  nondegenerate quadratic forms over $k$, of respective ranks $m $ and $n \leq m$. Let ${\bf X} $ be a separated $O$-scheme of finite type
 equipped with an isomorphism of  $X={\bf X}\times_{O}k$ with the closed $k$-subvariety of $\A^{mn}_{k}$ which the identity
  $$g(x_{1},\dots, x_{n})=f(l_{1}(x_{1},\dots, x_{n}),\dots,l_{m}(x_{1},\dots, x_{n}))$$
  defines -- here the $l_{i}$ are linear forms. Assume $\prod_{v}{\bf X}(O_{v}) \neq \emptyset$. 
 Does this imply ${\bf X}(O) \neq \emptyset $ ?  }
    
\medskip

We have adopted the following convention:  for $v$ archimedean we   set ${\bf X}(O_{v})=X(k_{v})$.  One could also naturally address the question of existence and density of $S$-integral solutions
for an arbitrary finite set $S$ of places, as we did in sections 3 and 4. In the interest of simplicity,
in the balance of this paper, when discussing representation of quadratic forms by quadratic forms,  we restrict attention to integral representations, as opposed to
$S$-integral representations as considered in earlier chapters. Also, we concentrate on the {\it existence} of   integral points and do not systematically state the strongest approximation results.
The reader will have no difficulty to
apply the general theorems of earlier sections to get the most general results.

\medskip

According to a well known result of Hasse, 
the hypothesis $\prod_{v} X(k_{v}) \neq \emptyset$
implies $X(k) \neq \emptyset$.

 If $m\geq 3$ then as explained in \S 5, we may fix an isomorphism
$X \simeq Spin(f)/H$. Here   $H$ is a connected linear algebraic group if
$m-n \geq 2$, $H=\mu_{2}$ is $m-n\leq 1$ (as usual, in the case $n=m$,
we replace $X$ by one of its connected components).

We shall say that a finite set $S$ of places of $k$ is big enough for $\{f,g\}$
if $S$ contains all the archimedean places, all the dyadic places and  all the nonarchimedean places
such that $\disc(f)$ or $\disc(g)$ is not a unit.

 The following result is  well known (Kneser). It is most often stated
under the assumption that
 $v_{0}$ is an archimedean place, in which case the above integral representation problem has
 a positive answer.

 \begin{thm}\label{atleast3} Let $f, g,$ and  ${\bf X}/O$ be as above, with
$m-n \geq 3$.  Let  $v_{0}$ be a place of $k$ such that $f_{k_{v_{0}}}$
 is isotropic. If $\prod_{v}{\bf X}(O_{v}) \neq \emptyset$ then ${\bf X}(O_{\{v_{0}\}}) \neq \emptyset$:
 there is a point which is integral away from $v_{0}$.
 Moreover ${\bf X}(O_{\{v_{0}\}})$ is dense in the topological product $\prod_{v \neq v_{0}}  {\bf X}(O_{v})$.
 \end{thm}
 \begin{proof} 
 In this case $X \simeq Spin(f)/H$ with $\pic H=0$ and $\br X/\br k =0$ (subsection 5.3).
 The hypothesis $f_{k_{v_{0}}}$
 isotropic is equivalent to the hypothesis that $SO(f)(k_{v_{0}})$ or equivalently $Spin(f)(k_{v_{0}})$ 
 is not compact. The theorem is thus a special case of Theorem  \ref{obs.BM.strong}.
   \end{proof}
   
   \begin{rem}
One may prove the above theorem without ever mentioning the Brauer group.
One   uses the left hand side of diagram (\ref{big3}),
strong approximation for $G=Spin(f)$
and the Hasse principle: for $G$ semisimple and simply connected,
the map $H^1(k,G) \to \prod'_{v\in \Omega_{k}} H^1(k_{v},G)$
reduces to a bijection $H^1(k,G) \to \prod_{v 
\in S_{{\infty}}}H^1(k_{v},G)$,
where $S_{{\infty}}$ denotes the set of archimedean places of $k$.
Surjectivity of the map is used for $H=Spin(h)$, injectivity for 
$G=Spin(f)$.
\end{rem}

\bigskip

When $m-n \leq 2$, examples in the literature, some of which will be mentioned in   later sections,
show that the existence of local integral solutions is not a sufficient condition for the existence of an integral solution.

 \begin{thm} \label{thm.obs.codim2.indefinite}
Let $f, g, {\bf X}/O$ be as above, with
 $m-n = 2$. Let $d=-\disc(f).\disc(g)\in k^*$.   
 Let $K=k[t]/(t^2-d)$.
Let $T$ denote the $k$-torus $R^1_{K/k}G_{m}$.
Assume $\prod_{v}{\bf X}(O_{v}) \neq \emptyset$. Then $X(k) \neq \emptyset$.
Fix $M \in X(k)$.  
The choice of  $M$ defines a $k$-morphism $SO(f) \to X$.
Let $\xi  \in H^1_{\etale}(X,\mu_{2})$ be the class of the $T$-torsor
defined by the composite map $Spin(f) \to SO(f) \to X$.
For any field $F$ containing $k$ we have the map $\psi_{F}: X(F) \to H^1(F,T)=F^*/N(FK)^*$.
The quotient $\br X/\br k$ is of order 1 if $d$ is a square, of order 2 if $d$ is not a square.
In the latter case it is spanned by the class of an element $\alpha \in \br X$ of order 2, well defined up to addition of
an element of $\br k$.

For a point $  \{M_{v}\}  \in 
\prod_{v \in \Omega_{k}}{\bf X}(O_{v}) $
the following conditions are equivalent:

 \begin{enumerate}[\rm (i)]
 \item 
$  \{M_{v}\} $
 is orthogonal to $\br X$ for the Brauer--Manin pairing.
 \item	
 $  \sum_{v\in \Omega_{k} }{\rm inv}_{v}(\alpha(M_{v}))=0$.
 \item
 
   $  \{M_{v}\} $ is in the kernel of the composite map
$$X({\mathcal A}_k)  \to \oplus_{v \in \Omega_{k} }k_{v}^*/NK_{v}^* \to \Z/2,$$
where the first map is defined by the various $\psi_{k_{v}}$
and the second map is the sum of the local Artin maps $k_{v}^*/NK_{v}^* \to {\rm Gal}(K_{v}/k_{v})Ê\subset {\rm Gal}(K/k)=\Z/2.$
\end{enumerate}
\item

Let $S$ be a finite set of places of $k$,  big enough for $\{f, g \}$,
 and such that there exists an isomorphism ${\bf X} \times_{O}O_{S} \simeq {\bf Spin}(f)/{\bf T}$ over
 $O_{S}$. Here $\bf T$ is an $O_{S}$-torus such that ${\bf T} \times_{O_{S}}k = T$.

Then the  above conditions on   $  \{M_{v}\}  \in 
\prod_{v \in \Omega_{k}}{\bf X}(O_{v}) $   are equivalent to:

(iv) The projection $\{ M_{v} \}_{v \in S}$ is in the kernel of the composite  map
$$\prod_{v \in S} X(k_{v}) \to \oplus_{v \in S}k_{v}^*/NK_{v}^* \to \Z/2.$$

Let $v_{0}$  be a place of $k$ such that $f_{k_{v_{0}}}$
 is isotropic. Under any of the above conditions 
 the element $ \{M_{v}\}\ \in 
  \prod_{v  \in \Omega_{k} \setminus v_{0}}  {\bf X}(O_{v})$ can be approximated arbitrarily closely by an
  element of ${\bf X}(O_{\{v_{0}\}})$.  In particular ${\bf X}(O_{\{v_{0}\}}) \neq \emptyset$.
\end{thm}

 \begin{proof} Just combine Subsection 5.6 with Theorem  \ref{obs.BM.strong}.
 \end{proof}

 {\it Computational recipes}

(i)   
To be in a position to apply the above theorem, one must first exhibit a $k$-rational point
 of $X$. Starting from such a $k$-point, one determines a finite set $S$ of places as
 in the theorem.
To decide if  an $\{M_{v}\}_{v \in S}$ satisfies  (iv), or even if there is
 such an $\{M_{v}\}_{v \in S}$  is then  the matter of finitely many computations.
 Indeed one only needs to give a concrete description of the maps
 $\psi_{k_{v}}: X(k_{v}) \to k_{v}^*/NK_{v}^* $ for each $v\in S$.
 This has already been given in subsection 5.6, with complements in subsection 5.2
 and 5.7 for the computation of the spinor norm map.

 Given any point  $M_{v} \in X(k_{v})$ there exists an element $\sigma_{v} \in SO(f)(k_{v})$
such that $\sigma_{v}(M)= M_{v}$. 
To $\sigma_{v} \in SO(k_{v})$ one associates its spinor norm $\theta(\sigma_{v}) \in k_{v}^*/k_{v}^{*2}$.
Then $\psi_{k_{v}} (M_{v})$ is the image of this element under
projection $k_{v}^*/k_{v}^{*2} \to k_{v}^*/NK_{v}^*$.

(ii)
 In the case $m=3, n=1$, that is when $X$ is given by an equation $f(x,y,z)=a$,
the discussion in subsection 5.8
leads to an alternative, possibly more efficient,  recipe. Compare the comments after
Theorem \ref{obs.BM.strong}.

\begin{thm}\label{differenceatmostone} Let $f, g$ and $ {\bf X}/O$ be as above, with
$m \geq 3$ and $1 \geq m-n \geq 0$.
Assume $\prod_{v}{\bf X}(O_{v}) \neq \emptyset$.
Then $X(k) \neq \emptyset$. The choice of a $k$-point $M \in X(k)$
defines a   $k$-isomorphism $ SO(f) \simeq X$. 
Let $\xi  \in H^1_{\etale}(X,\mu_{2})$ be the class of the $\mu_{2}$-torsor
defined by $Spin(f) \to SO(f) \simeq X$.
For any field $F$ containing $k$ this torsor defines a map $\psi_{F}: X(F) \to H^1(F,\mu_{2})=F^*/F^{*2}$. The composite map $SO(f)(F) \simeq X(F) \to F^*/F^{*2}$ is the spinor norm map.

(a) For a point  $\{M_{v}\} \in \prod_{v \in \Omega_{k}}{\bf X}(O_{v})$ 
  the following conditions are equivalent:

(i) $\{M_{v}\} $ is  the kernel of the map  $X({\mathcal A}_k) \to \hom(\br X, \Q/\Z)$.

(ii) $\{M_{v}\} $  is in the kernel of the composite map
$$X({\mathcal A}_k) \to \prod'  k_{v}^*/k_{v}^{*2}  \to \hom(k^*/k^{*2}, \Z/2),$$
where the last map is given by the sum over all $v$'s of Hilbert symbols.

Assume that the finite set of places $S$ is big enough for $(G,\mu_{2})$ and
that there is an isomorphism $  {\bf SO}(f) \simeq {\bf X}\times_{O} O_{S} $ 
 extending $ SO(f) \simeq X$. Conditions (i) and (ii) on    $\{M_{v}\} \in \prod_{v \in \Omega_{k}}{\bf X}(O_{v})$ 
 imply

 (iii)  The point $\{M_{v}\}_{v \in S} \in  \prod_{v \in S}{\bf X}(O_{v})$ is in the kernel of the map
$$ \prod_{v \in S}X(k_{v} ) \to   \prod_{v\in S} k_{v}^*/k_{v}^{*2}      \to        \hom(H^1_{\etale}(O_{S},\mu_{2}), \Z/2).$$

Let $v_0$ be a place of $k$ such that $f_{k_{v_0}}$ is isotropic.

(b)  If $\{M_{v}\} \in \prod_{v \in \Omega_{k}}{\bf X}(O_{v})$ satisfies  condition (i) or (ii) 
 and $S_{1}$ is a finite set of places containing $v_{0}$, then there exists 
  $M \in {\bf X}(O_{\{v_{0}\}})$ arbitrarily close to each $M_{v}$
for $v \in S_{1} \setminus S_{0}$. In particular ${\bf X}(O_{\{v_{0}\}}) \neq \emptyset$.

(c) If the finite set $S$ of places is as above and contains $v_{0}$ 
and if   $\{M_{v}\}_{v \in S} \in  \prod_{v \in S}{\bf X}(O_{v})$ is as in condition (iii)
then  there exists 
  $M \in {\bf X}(O_{\{v_{0}\}})$ arbitrarily close to each $M_{v}$
for $v \in S  \setminus S_{0}$. In particular ${\bf X}(O_{\{v_{0}\}}) \neq \emptyset$.
\end{thm}

\begin{proof}  Just combine Subsections 5.4 and 5.5  
 with Theorem \ref{thm.strongapp.finite}. 
\end{proof}

\bigskip

 {\it Computational recipe}

One first exhibits some point $M \in X(k)$. Using this point one determines  $S$ as in the theorem. 
One enlarges $S$ so that the 2-torsion of the class group of $O_{S}$ vanishes. One then has
$O_{S}^*/O_{S}^{*2} \simeq H^1_{\etale}(O_{S},\mu_{2})$. The group $O_{S}^*/O_{S}^{*2}$ is finite.

{\it First method}
To each element $\eta \in O_{S}^*/O_{S}^{*2}$
one associates the cup-product $\xi \cup \eta \in H^2(X,\mu_{2})$ which one then pushes into 
$\br X$. This produces finitely many elements $\{\beta_{j}\}_{{j \in J}}$  of order 2 in $\br X$,
which actually are classes of quaternion Azumaya algebras over ${\bf X}\times_{O}O_{S}$.
For a given  $j$ one considers the  map
$$  \prod_{v \in S}{\bf X}(O_{v})  \to \Z/2$$
given by $\{M_{v}\}_{v \in S} \mapsto \sum_{v \in S}{\rm inv}_{v}(\beta_{j}(M_{v})) \in \Z/2.$
One then checks whether there exists a point   $\{M_{v}\}  \in \prod_{v \in S}{\bf X}(O_{v})$  which simultaneously 
lies in the 
kernel of these finitely many maps.

 {\it Second method} 
 One considers the map $X(k) \to H^1(k,\mu_{2})=k^*/k^{*2}$
 associated to $\xi$. For $S$ as above, the image of ${\bf X}(O_{S})$ lies in the finite group
 $C= H^1_{\etale}(O_{S},\mu_{2})= O_{S}^*/O_{S}^{*2}$. 
 For each element $\rho \in C$, one considers
 the $\mu_{2}$-torsor $Y^{\rho}$ over $X$ obtained by twisting
 $Y$ by a representant of $\rho^{-1}  \in  O_{S}^*/O_{S}^{*2} \subset k^*/k^{*2}$.
Then the kernel in (iii) is not empty  if and only if there exists at least one $\rho \in C$
 and a family $\{M_{v}\}  \in \prod_{v \in S} {\bf X}(O_{v})$ such that
 there exists a family $\{N_{v}\} \in \prod_{v\in S} Y^{\rho}(k_{v})$ which maps to 
  $\{M_{v}\} \in 
   \prod_{v \in S} X(k_{v})$ under the 
  structural map $Y^{\rho} \to X$.

\section{Genera and spinor  genera}\label{genera}

A necessary condition for an integral quadratic form $g$ to be represented
by an integral quadratic form $f$ (of rank at least 3) over the integers is that it be represented
by $f$ over each completion of the integers and at the infinite places.
If that is the case, $g$ is said to be represented by the genus of $f$.
A further, classical necessary condition, considered by Eichler  \cite{Eichler}
and Kneser  \cite{Kn1956},
is that $g$ be  represented by
the spinor genus of $f$ (see  \cite{O'Meara}). 
In this section we first recall the classical language of lattices.
We then show how the spinor genus condition boils down 
 to an integral Brauer--Manin condition of the type considered in the previous section.
 Finally, we 
compare the results in terms of the Brauer--Manin obstruction with some
 results obtained in \cite{Kn1961}, \cite{HSX1998}, \cite{CX2004}, \cite{X2005} and \cite{SPX2004}.
 With hindsight, we see that some version of the Brauer--Manin condition had already been encountered
 in these papers.

\subsection{Classical parlance}

\medskip

Let $k$ be a number field, $O$ its ring of integers. Let $V$ be a finite-dimensional vector space over $k$
 equipped with a nondegenerate quadratic form $f$ with associated bilinear form $B_{f}$.
A quadratic lattice 
$L \subset V$ is a finitely generated, hence   projective, $O$-module  
such that $f(L) \subset O$ and such that the restriction of the quadratic form $f$ on
$L_{k}=L\otimes_{O}k \subset V$ is nondegenerate.
Given any element $\sigma\in O(f)(k)$ the $O$-module
$\sigma.L$ is a quadratic lattice.  A quadratic lattice  $L$ is called full
if its rank is maximal, i.e. $L_{k}=V$.

Two full quadratic lattices $L_{1}$ and $L_{2}$
are in the same class, resp. the same proper class, if there exists  $\sigma \in O(f)(k)$,
resp. $\sigma \in SO(f)(k)$, 
such that $L_{1}=\sigma.L_{2}$.

Given a quadratic lattice  $N \subset V$ of rank $n$  and  a full quadratic lattice $M \subset V$ of rank $m={\rm dim}_{k}V$, one asks whether
there exists $\sigma \in O(f)(k)$, resp.  $\sigma \in SO(f)(k)$, 
 such that $N \subset \sigma.M$. If  the rank of $N$ is strictly less than the rank of $M$,
 i.e. if $N$ is not full, the two statements are equivalent.

 If that is the case, one
says that the quadratic lattice $N$ is represented by the class, resp. the proper class of the quadratic lattice $M$. One sometimes writes
$N  \rightarrow cls(M)$, resp. $N  \rightarrow cls^+(M)$.

From now on we assume $m={\rm dim}V \geq 3$. (The case $m=2$ is very interesting but
requires other techniques.)

There is an action of  the group of ad\`eles $O(f)({\mathcal A}_k)$ on the set of quadratic lattices in $V$.
 Indeed, given an ad\`ele $\{\sigma_{v}\} \in O(f)({\mathcal A}_k)$ and a full quadratic lattice 
$L \subset V$, one shows (\cite{O'Meara}, 81:14)  
that there exists  a unique full quadratic lattice
$L_{1} \subset V$   such that $L_{1}\otimes_{O}O_{v}=\sigma_{v}(L\otimes_{O}O_{v})
\subset V \otimes_{k}k_{v}$ for each finite place $v$.

Two full quadratic lattices in $(V,f)$ in the same orbit of $O(f)({\mathcal A}_k)$ 
are said to be in the same genus. They automatically lie in the same orbit of $SO(f)({\mathcal A}_k) \subset O(f)({\mathcal A}_k)$ (\cite{O'Meara}, \S 102 A).

On says that a  quadratic lattice $N \subset V$ is  represented   by
the genus  of the full quadratic lattice $M \subset V$ if there exists at least one quadratic lattice $M_{1} \subset V$
in the genus of $M$ such that $N \subset M_{1} \subset V$.
One sometimes writes
$N  \rightarrow gen(M)$.

We have the natural isogeny $\varphi: Spin(f) \to SO(f)$, with kernel $\mu_{2}$.
The group $Spin(f)({\mathcal A}_k)$ acts on the set of maximal quadratic  lattices in $V$  through $\varphi$.
The group $\varphi
(Spin(f)( {\mathcal A}_k) $ is a normal subgroup in $SO(f)({\mathcal A}_k)$.
One therefore has an action of the group $$ O(f)(k)\cdot \varphi (Spin(f)({\mathcal A}_k)) = \varphi (Spin(f)({\mathcal A}_k)) \cdot  O(f)(k)$$
on the set of such lattices. One says that two full quadratic lattices 
are in the same spinor genus, resp. in the same proper spinor genus,  if they are in the same orbit of $O(f)(k)\cdot \varphi (Spin(f)({\mathcal A}_k)) $, resp. of  $SO(f)(k)\cdot \varphi (Spin(f)({\mathcal A}_k)) $

One says that a quadratic lattice $N \subset V$ is represented by the spinor genus, resp. the proper spinor genus
of the full quadratic lattice $M$ if there exists at least one quadratic lattice $M_{1} \subset V $ in the spinor genus, resp. in the proper spinor genus of $M$,
such that $N \subset M_{1}$. One  sometimes writes
$N  \rightarrow spn(M)$, resp. $N  \rightarrow spn^+(M).$
 
\bigskip

Let $N \subset V$ and $M \subset V$ be quadratic lattices in $(V,f)$, with $M$ a full lattice.
There is an induced quadratic form $f$ on $M$ and an induced quadratic form $g$ on $N$.
We may then consider $(N,g)$ and $(M,f)$ as abstract quadratic spaces over $O$
(with associated bilinear form nondegenerate over $k$, but not necessarily over $O$).
We let $N_{k}=N\otimes_{O}k$ and $M_{k}=M\otimes_{O}k=V$.

Let $Hom_{O}(N,M)$ be the scheme of linear maps from $N$ into $M$.
Let ${\bf X}/O$ be the closed subscheme defined by the linear maps
compatible with the quadratic forms on $N$ and $M$.
 Let $X={\bf X}\times_{O}k$.
As explained in \S 1, for the purposes of this paper we may if we wish replace
${\bf X}$, which need not be flat over $O$ (dimensions of fibres may jump), by the schematic closure of
$X$ in ${\bf X}$, which is integral and flat over $O$. This does not change  the generic fibre $X$, and it does not change the sets ${\bf X}(O)$ and ${\bf X}(O_{v})$.

Since we are given
quadratic lattices $N \subset V$ and $M \subset V$ in the same quadratic space $(V,f)$ over $k$, 
we are actually given a  $k$-point $\rho \in X(k)$, that is a $k$-linear map $\rho:N_{k} \to M_{k}$
which is compatible with the quadratic forms $f$ and $g$. Conversely such a map defines
a point of $X(k)$.
If $n <m$ then the $k$-variety $X$ is connected and is a homogeneous space of $SO(f)$. If $n=m$, we shall henceforth replace $X$ by the connected component to which the given $k$-point belongs
and ${\bf X}$ by the schematic closure of that connected component;
the new $X$ is a (principal) homogeneous space of $SO(f)$.
In all cases, we shall view the $k$-variety $X$ as a homogeneous space of 
the $k$-group $Spin(f)$.

\bigskip
\subsection{Classical parlance versus integral Brauer--Manin obstruction}

\begin{prop}\label{classic.class}

With notation as   in subsection 7.1, the following conditions are equivalent:

(i) The quadratic lattice  $N$ is  represented by the proper class of the quadratic lattice $M$.

(ii)   We have ${\bf X}(O) \neq \emptyset$.

\end{prop}
\begin{proof}
Assume (i). Thus there exists $\sigma \in SO(f)(k)$
such that $\sigma(N) \subset M \subset V$. The linear map $\sigma(\rho): N_{k} \to M_{k}$ sends 
$N$ to $M$ and is compatible with the quadratic forms. It is thus a point of ${\bf X}(O)$.

Assume (ii). There exists an $O$-linear map $\lambda: N \to M$ which is compatible with
the quadratic forms $f$ and $g$. 
We also have the given $k$-point  $\rho \in X(k)$.
By a theorem of Witt and the definition of ${\bf X}$ in the case $n=m$ there exists $\sigma \in SO(f)(k)$ such that $\sigma(\rho)=\lambda_{k}$ over $k$.
Thus $\sigma(N) \subset M \subset V$. 
\end{proof}

\begin{prop}\label{classic.genus}

With notation as in subsection 7.1,  the following conditions are  equivalent:

(i) The quadratic lattice $N$ is   represented by the   genus of  the quadratic lattice $M$.

(ii) $\prod_{v \in \Omega_{k}}{\bf X} (O_{v}) \neq \emptyset$.
\end{prop}
\begin{proof}
For any  place $v$  of $k$ let $N_{v}=N\otimes_{O}O_{v}$ and $M_{v}=M\otimes_{O}O_{v}$.

Assume (i). Let
$\{\sigma_{v}\} \in SO(f)({\mathcal A}_k)$ be such that $\sigma_{v}(N_v) \subset M_v \subset V\otimes_{k}k_{v}$.
For each finite place $v$ the linear map $\sigma_{v}(\rho): N\otimes_{k}{k_{v}} \to M\otimes_{k}{k_{v}}$ sends 
$N_{v}$ to $M_{v}$ and is compatible with the quadratic forms. It is thus a point of ${\bf X}(O_{v})$.
By assumption $\rho \in X(k)$. Thus for $v$ archimedean ${\bf X}(O_{v})=X(k_{v}) \neq \emptyset$.   

Assume (ii). The argument given in the proof of the previous proposition shows that for each place
$v \in \Omega_{k}$ there exists $\tau_{v} \in SO(f)(k_{v})$ such that  $N_{v} \subset \tau_{v} (M_{v})$.
For  all places $v$ of $k$ not in a finite set $S \subset \Omega_{k}$, the discriminant of $g$ and the discriminant of $f$ are units in $O_{v}$, $N_{v} $ is an orthogonal factor of the unimodular $O_{v} $-lattice $M_{v}$ and it is also an orthogonal factor in the unimodular $O_{v} $-lattice $\tau_{v}(M_{v})$. 
For each place $v \notin S$ there thus exists $\varsigma_{v} \in SO(f)(k_{v})$ which sends
isomorphically $\tau_{v}(M_{v})$ to $M_{v} \subset V\otimes_{k}k_{v}$ and induces the identity map on $N_{v} \subset M_{v} \subset V\otimes_{k}k_{v}$ (\cite{O'Meara} Thm. 92:3).
 Therefore $\varsigma_v
\tau_v M_v=M_v$ for  all $v \notin S$. Let $\varsigma_v=1$ for $v \in S$.
 Then $\{\varsigma_v \tau_v\}\in
SO(f)({\mathcal A}_k)$ and $N_v \subset  \varsigma_v \tau_v M_v$ for all $v$.
Therefore $N$ is   represented by the genus of $M$.
\end{proof}

 \begin{prop}\label{classic.spinorgenus}
 With notation as   in subsection 7.1, the following conditions are equivalent.

 (i) The quadratic lattice $N$ is represented by the proper spinor genus of   the quadratic lattice $M$.
 
  (ii)  $(\prod_{v \in \Omega_{k}}  {\bf X}(O_{v}))^{\br X} \neq \emptyset$.
 \end{prop}
\begin{proof}
Assume (i). 
Let $\sigma
\in SO(f)(k)$ and $\{\tau_v\}\in \varphi (Spin({\mathcal A}_k))$ be such that
$N_{v}\subset  \{\tau_v\}  \sigma M_{v}$ for each $v \in \Omega_{k}$. 
The   map $\sigma(\rho): N_{k} \to M_{k}$ defines a
$k$-point $p \in X(k)$, which defines a point in $\{p_{v}\} \in X({\mathcal A}_k)$.
One then applies the element $\{\tau_v\}\in \varphi (Spin({\mathcal A}_k))$
to get the point $\{x_{v}\} = \{\tau_v.p\} \in X({\mathcal A}_k)$. 
By hypothesis,
the element $\{x_{v}\}$ lies in $\prod_{v \in \Omega_{k}}  {\bf X}(O_{v})$.
Since $\{p_{v}\} \in X({\mathcal A}_k)$ is the diagonal image of an element of
$X(k)$, it is orthogonal to $\br X$. 
Consider diagram (3.1) 
after Theorem 
\ref{Kottwitzrecipro}
 if $m-n\geq 2$, 
resp. diagram (4.3) after Theorem 
\ref{Poitou.Tate}
if $0Ê\leq m-n \leq 1$. 
The commutativity of those diagrams
implies  that the image of $\{x_{v}\}=\{\tau_v.p\}$ in ${\rm Hom}(\pic H,\Q/\Z)$,
resp. ${\rm Hom}(H^1(k,\hat{\mu} ),\Q/\Z)$ is zero.
If $m-n\geq 2$  we know from
subsections 5.3 and 5.6   that
the map ${\rm Hom}(\br X/\br k, \Q/\Z) \to {\rm Hom}(\pic H,\Q/\Z)$ is 
an isomorphism.
If $0Ê\leq m-n \leq 1$ we know from subsections 5.4 and 5.5   that
the map
${\rm Hom}(\br X/\br k , \Q/\Z) \to {\rm Hom}(H^1(k,\hat{\mu} ),\Q/\Z)$
 is 
an isomorphism. Thus in all cases we find that $\{x_{v}\}$ lies in
$(\prod_{v \in \Omega_{k}}  {\bf X}(O_{v}))^{\br X}$.

Assume (ii).  Let $\{x_{v}\}$ belong to $(\prod_{v \in \Omega_{k}}  {\bf X}(O_{v}))^{\br X} $.
Each $x_{v}$ corresponds to an $O_{v}$-linear map $N_{v} \to M_{v}$
which respects the quadratic forms $f$ and $g$.
We have $X=G/H$ with $G=Spin(f)$ and $H$ semisimple simply connected
if $m-n \geq 3$, $H$ a 1-dimensional $k$-torus $T$ if $m-n=2$ and
$H=\mu_{2}$ if $0Ê\leq m-n \leq 1$. Applying Theorem 
\ref{thm.orbit}
when $H$ is
connected and Theorem 
\ref{thm.orbit.finite}
 when $H=\mu_{2}$, we see that there
exists  a rational point $p \in
X(k)$, with associated  linear map $N_{k} \to  M_{k}$
 and $\{\tau_v\}\in \varphi (Spin({\mathcal A}_k))$ such that $\tau_v p=x_{v}\in
{\bf X}(O_v)$ for all $v$.  By Witt's theorem there exists  $\sigma \in SO(f)(k)$ be such that
$\sigma(\rho): N_{k} \to M_{k}$ is given by the point $p$.
Then
$\tau_v\sigma (N_v)\subset  M_v$ for all $v$. 
Thus $N$ is represented by the proper spinor genus of $M$.
\end{proof}

\bigskip

\rem{} Let $N,M$ be quadratic lattices in the quadratic space $(V,f)$, with $M$ a full lattice.
Let us assume there exists an archimedean place $v_{0}$ of $k$ such that
$f$ is isotropic over $k_{v_{0}}$, that is $Spin(f)(k_{v_{0}})$ is not compact.

Using the above propositions, we recover the classical result:
For such an $f$ and $m-n \geq 3$, Theorem \ref{atleast3} 
implies that any quadratic lattice $N$  represented by the   genus of $M$
 is represented by the proper class of $M$. 
 
 In the cases $  m-n \leq 2$, Theorems \ref{thm.obs.codim2.indefinite} 
 and \ref{differenceatmostone} show that the representation of a given quadratic lattice $N$
 by the proper class of $M$ may be decided after a finite amount of computation.
   \endrem{}

\subsection{Relation with some earlier literature}

We keep notation as in \S 7.1.
We thus have a finite dimensional vector space $V$ over  $k$ of rank $m \geq 3$,
equipped with a nondegenerate quadratic form $f$.
We are given two quadratic lattices $N \subset V$ and $M \subset V$, with $M$
a full lattice. 

We let ${\bf X}/O$ be the closed subscheme of 
${\rm Hom}_{O}(N,M)$ consisting of maps which respect the quadratic form $f$
on $M$ and the form $g$ it induces on $N$. We let $X={\bf X}\times_{O}k$.
The natural inclusion $N_{k} \subset M_{k}=V$  determines a $k$-point $\rho \in X(k)$.
When $n=m$ we replace $X$ by the connected component of $\rho$ and
${\bf X}$ by the schematic closure of that component in ${\bf X}$.
If we wish, when $n <m$,  we may perform the same replacement.

 In this subsection we work under the standing assumption $\prod_v {\bf X}(O_v) \neq \emptyset$.

We have the natural homogeneous map
$$ \phi: SO(f) \to X$$
sending $1$ to the point $\rho$.
For any field $K$ containing $k$, the induced map $SO(f)(K) \to X(K)$ is surjective (Witt).

In earlier studies of the representation of $N$ by $M$ (\cite{CX2004}, p. 287 and \cite{X2005}, p. 38), 
 the following sets  played an important r\^ole.
For any place $v$ of $k$, one lets
$$X(M_v/N_v) = \{ \sigma \in
SO(f)(k_v):   N_v \subset \sigma(M_v) \}.$$
For almost all places $v$,  the form $f$ is nondegenerate over $O_{v}$
and  there is an inclusion $N_{v} \subset M_{v}$ over $O_{v}$
which over $k_{v}$ yields $\rho\otimes_{k}k_{v}$. For almost all $v$ we therefore have
$$ SO(f)(O_{v}) \subset X(M_v/N_v).$$

The  set $X(M_v/N_v) $ is not empty if and only if ${\bf X}(O_v)\neq \emptyset$.
As a matter of fact,
$$X(M_v/N_v) = \phi^{-1}({\bf X}(O_{v})) \subset SO(f)(k_v).$$
The spinor maps induce  maps
$$\theta_{v }: X(M_v/N_v)  \to k_{v}^*/k_{v}^{*2}.$$

 With notation as in Theorem \ref{thm.obs.codim2.indefinite}
we  have:
  
\begin{thm}\label{classic.codim2}
 Assume $m-n=2$. Let $d=-\disc(f).\disc(g)$ and $K=k[t]/(t^2-d)$.
The following conditions are equivalent:

(i) The quadratic lattice $N$ is represented by the  proper spinor genus of   the quadratic lattice $M$.

(ii)  $ (\prod_{v \in \Omega_{k}}{\bf X}(O_{v}))^{\br X} \neq \emptyset$.

(iii) There exists a point in the kernel of the composite map
$$\prod_{v \in \Omega_{k}}'  X(M_v/N_v)   \to 
 \prod_{v \in \Omega_{k}}' k_{v}^*/k_{v}^{*2} \to 
 \oplus_{v \in \Omega_{k}} k_{v}^*/N_{K/k}K_{v}^* \to \Z/2.$$
 The restricted product on the left hand side is taken with respect to the subsets
 $SO(f)(O_{v}) \subset X(M_v/N_v)$ at places of good reduction.
 
 Let $S$ be a finite set of places containing all archmimedean places, all dyadic places,
 all finite places $v$ at wich either $\rm{disc}(f)$ or  $\rm{disc}(g)$ is not a unit. 
 Assume moreover that at each finite nondyadic place $v \notin S$ the natural injection $N_{k_{v}} \subset M_{k_{v}}$,
 which is given by $\rho_{k_{v}}$,  comes from an injection $N_{v} \subset M_{v}$.
Then the above conditions are equivalent to:
 
 (iv) There exists a point in the kernel of the composite map
$$\prod_{v \in S}  X(M_v/N_v)   \to 
  \prod_{v \in S} k_{v}^*/k_{v}^{*2} \to 
 \oplus_{v \in S} k_{v}^*/N_{K/k}K_{v}^* \to \Z/2.$$

 If the form $f$ is isotropic at an archimedean place of $k$, these conditions are equivalent to
 
 (v) The quadratic lattice $N$
 is represented by the proper class of $M$.
 \end{thm}
 \begin{proof} Combine Proposition \ref{classic.spinorgenus}
  and Theorem 
 \ref{thm.obs.codim2.indefinite}. Note that the assumption on $S$ 
 ensures that $\rho \in X(k)$ actually belongs to ${\bf X}(O_{S})$
 and defines  an $O_{S}$-isomorphism ${\bf Spin}(f)/{\bf T} \simeq
 {\bf X}\times_{O}O_{S}$, for $\bf T$ as in Theorem \ref{thm.obs.codim2.indefinite}.
\end{proof}

With notation as in Theorem \ref{differenceatmostone}
we have:

\begin{thm}\label{classic.codim1} 
Assume $0 \leq m-n \leq 1$.
With notation as above, the following three conditions are equivalent:

(i) The quadratic lattice $N$ is represented by the proper spinor genus of   the quadratic lattice $M$.

(ii)  $(\prod_{v \in \Omega_{k}}   {\bf X}(O_{v}))^{\br X} \neq \emptyset$.

(iii)  There exists a point in the kernel of the composite map
$$\prod_{v \in \Omega_{k}}' X(M_v/N_v)   \to 
 \prod_{v \in \Omega_{k}}' k_{v}^*/k_{v}^{*2}  \to  \hom(k^*/k^{*2},\Z/2).$$
 
  Let $S$ be a finite set of places containing all archmimedean places, all dyadic places,
 all finite places $v$ at wich   $\rm{disc}(f)$   or or  $\rm{disc}(g)$ is not a unit. 
 Suppose the form  $f$ is isotropic at an archimedean place of $k$.
  Then these conditions are equivalent to 
  
   (iv) There exists a point in the kernel of the composite map
$$\prod_{v \in S}  X(M_v/N_v)   \to 
  \prod_{v \in S} k_{v}^*/k_{v}^{*2}  \to        \hom(H^1_{\etale}(O_{S},\mu_{2}), \Z/2).$$
 
 (v) The quadratic lattice $N$
 is represented by the proper class of $M$.
\end{thm}
 \begin{proof} Combine Proposition \ref{classic.spinorgenus}
 and Theorem \ref{differenceatmostone}. Note that the assumption on $S$ 
 ensures that $\rho \in X(k)$ actually belongs to ${\bf X}(O_{S})$.
   Note that the assumption on $S$ 
 ensures that $\rho \in X(k)$ actually belongs to ${\bf X}(O_{S})$
 and defines  an $O_{S}$-isomorphism ${\bf SO}(f) \simeq
 {\bf X}\times_{O}O_{S}$.
 \end{proof}

 \rem{}
 The equivalence of (i) and (iii) in each of the last two theorems appears in various guises in 
 the literature.  Let us here quote  Eichler \cite{Eichler},  Kneser \cite{Kn1961} (see Satz 2 p.~93), Jones and Watson \cite{JW1956}, Schulze-Pillot \cite{SP1980} (see Satz 1, Satz 2), \cite{SP2000}, \cite{SP2004}, and most particularly
  Hsia, Shao, Xu \cite{HSX1998} (Thm. 4.1). See  also
 \cite{X2000}, \cite{CX2004} (Thm. 3.6 p.~292),   \cite{SPX2004} (Prop.  7.1), 
\cite{X2005}((5.4) and Cor. 5.5 p.~50).

One may
rephrase Theorems  \ref{classic.codim2} and  \ref{classic.codim1} in terms of   the  {\it spinor class fields}
defined in \cite{HSX1998} p.~131.  
 The construction of such spinor class fields is based on the
following fact, which is proved by {\it ad hoc} computations
(Thm.~2.1 of \cite{HSX1998}):

\begin{fact*} {\it Let $v$ be a  finite place of $k$. If $N_v\subseteq M_v$,
then the set $\theta (X(M_v/N_v))$ is a subgroup of
$k_v^*/{k_v^*}^2$.}
\end{fact*}

Assume  $\prod_{v\in \Omega_k} \mathbf
X(O_v) \neq \emptyset$. There is $\sigma \in SO(f)(k_{v})$ such that
$\sigma N_{v} \subseteq M_{v}$. Then $\theta (X(\sigma^{-1}M_{v}/
N_{v}))$ is a subgroup of $k_{v}^*/{k_{v}^*}^2$ by the above fact.
Let $\tau\in SO(f)(k_v)$ such that $\tau N_v\subseteq M_v$. Then
$$X(\sigma^{-1} M_v/ N_v)\sigma^{-1}= X(\tau^{-1}M_v/ N_v)\tau^{-1} \subset SO(f)(k_{v}).$$ 
This implies that  the group $\theta (X(\sigma^{-1}M_{v}/
N_{v})) \subset k_{v}^*/{k_{v}^*}^2$ is independent of the choice of
$\sigma$. 
For  all $v$  such that $2 det(N) det(M)$ is a unit
this subgroup 
contains 
$O_{v}^*/{O_{v}^*}^2$.
One lets  $\theta (M_v,N_v)Ê\subset k_{v}^*$  denote its inverse image
under the map  $k_{v}^* \to k_{v}^*/{k_{v}^*}^2.$

The finite Kummer 2-extension $\Sigma_{M/N}$ of $k$ corresponding to
$$k^*\prod'_{v\in \Omega_k} \theta (M_v,N_v)$$ is called the
{\it spinor class field} of $M$ and $N$. 
The notation $\prod'$ here means the trace of $\prod'_{v\in \Omega_k} \theta (M_v,N_v) \subset \prod_v k_v^*$ on the group of id\`eles of $k$.

In the situation of Theorem   \ref{classic.codim2} ($m-n=2$),  one has $k \subset \Sigma_{M/N} \subset K=k(\sqrt{d})$.
In the situation of Theorem  \ref{classic.codim1} ($m-n \leq 1$), the field extension $\Sigma_{M/N}/k$
is a subfield of the   maximal 2-Kummer extension of $k$
which is ramified only  at
primes with $v|2det(M)$ and at the archimedean primes.

The detailed comparison    with the results of \cite{HSX1998} is left to the reader.

There are several articles
devoted to explicit computations of the group $\theta (M_v,N_v)$ in terms of the local
Jordan splitting of $M_v$ and $N_v$.  
In the case $rank (M_v)=3$,
$rank(N_v)=1$, the group
 $\theta (M_v,N_v)$ is computed in
\cite{SP1980} for $v$
 non-dyadic or 2-adic and it is computed  in 
   \cite{X2000} for the general dyadic case.
 In \cite{HSX1998}, the group $\theta (M_v,N_v)$ is computed for general
$M_v$ and $N_v$ with non-dyadic $v$.
\endrem

\subsection{Spinor exceptions}

In this subsection we assume $m=n+2$.

\begin{Def} 
Suppose the quadratic lattice $N$ is represented by the genus of the quadratic lattice $M$. The lattice $N$ is called
a spinor exception for the genus of $ M$ if there is a proper spinor genus
in $gen(M)$ such that no lattice in that proper spinor genus
represents $N$.
\end{Def}

That is to say, there exists a lattice $M'$ in $gen(M)$ such that
$N$ is represented by $gen(M')$ but no lattice in the proper spinor
genus of $M'$ represents $N$.

We let $\bf X$ and $X$ be associated to the pair $N,M$ as in the begining of \S 7.3.
If $N$ is a spinor exception for $gen(M)$, then  
 Proposition  \ref{classic.spinorgenus} and  \S \ref{m-n=2}
imply $$d=-det(M)\cdot det(N)\not\in {k^*}^2 $$ and $\Br X/ \Br k \cong \Z/2$.

\begin{prop} 
Suppose the lattice $N$ is represented by $gen(M)$.
Suppose
$d=-det(M)\cdot det(N)\not\in {k^*}^2 $. Let
  $K=k(\sqrt{d})$. Let $A \in \Br X $ generate the group 
$\Br X/\Br k \cong \mathbb Z/2$.
The following conditions are equivalent:

$(i)$ $N$ is a spinor exception for $gen(M)$.

$(ii)$ For each $v\in
\Omega_k$, 
$ A$ assumes only a single value on ${\bf X}(O_v)$.
\end{prop}

\begin{proof} 
The condition in (ii) does not depend on the representant $A$.
Since $X(k) \neq \emptyset$, there exists an element
$A \in \Br X$ which is of exponent 2 in $\Br X$ and which generates
$\Br X/\Br k \cong \mathbb Z/2$. We fix such an element $A \in \Br X$.

Assume that for some place $v$, $A$ takes two distinct values on ${\bf X}(O_v)$.
This implies that $A$ has a nontrivial image in the group $\Br X_{k_{v}}/\Br k_{v}$,
which is of order at most 2.
The natural map  $\Z/2=\Br X/\Br k  \to \Br X_{k_{v}}/\Br k_{v} $ is thus an isomorphism.
Let $M'$ be a lattice in the genus of $M$. Let
  ${\bf X}'$ be the $O$-scheme attached to the pair $N,
M'$. The hypothesis that $M'$ is in the genus of $M$ implies that there exists  
   an   isomorphism
of $k_{v}$-schemes $X'_{k_{v}}= {\bf X}'\times_{O}k_v \cong {\bf X}\times_{O} k_v = X_{k_{v}}$.
The latter map induces an isomorphism
$\Br X_{k_{v}}/\Br k_{v} \cong    \Br X'_{k_{v}}/\Br k_{v}$.
The inverse image $B \in \Br X'_{k_{v}}$ of $A \in \Br X$ under the composite map
$X'_{k_{v}}  \cong  X_{k_{v}}  \to X$
takes two distinct values on ${\bf X}'(O_{v})$.
We also have a natural map
$ \Br X'/\Br k \to  \Br X'_{k_{v}}/\Br k_{v}.$ 
Since $M'$ is in the genus of $M$, there 
  exists an isomorphism of $k$-varieties $X \cong X'$.
The map $ \Br X'/\Br k \to  \Br X'_{k_{v}}/\Br k_{v}$ is therefore an isomorphism,
and both groups are isomorphic to $\Z/2$.
There thus exists an element $A' $ of order 2 in $\Br X'$ whose image  in $  \Br X'_{k_{v}}$
differs from $B$ by an element in $\Br k_{v}$.
Thus $A'$ takes two distinct values on ${\bf X}'(O_{v})$.
This implies 
$(\prod_{v\in \Omega_k}{\bf
X}'(O_v))^{\Br X'} \neq \emptyset$. By Proposition 7.3 this shows that
$N$ is represented by the proper spinor genus of the quadratic lattice $M'$.
Since $M'$ is an arbitrary lattice in the genus of $M$, this shows that (i) implies (ii).

Assume (ii). If the sum of the values of $A$ on each ${\bf X}(O_v)$ is nonzero, then by Prop.
\ref{classic.spinorgenus}
$N$ is not represented by the proper spinor genus of $M$
so $N$ is a spinor exception.
Let us assume otherwise.
Thus  the value $inv_v(A({\bf X}(O_v))) \in \Z/2$ is well defined and
we have: $$\sum_{v\in \Omega_k}inv_v(A({\bf X}(O_v)))=0 \in \mathbb
Z/2.$$
For any smooth compactification $X_{c}$ of $X$, we have
$\Br k = \Br X_{c}$. This is easy to show in the case $n=1,m=3$. In the general case,
this follows from \cite{CTK} (see Prop.
\ref{picbrHconnected} (iii) above)  together with the easy computation that for
the $k$-torus $T=R^1_{K/k}\G_{m}$  any class in $H^1(\g,\hat{T})$ whose restriction to
procyclic subgroups of $\g$ vanishes must itself vanish.
Thus the class $A \in \Br X$ does not extend 
to a class
on a smooth compactification of $X$.
By a result of Harari (Corollaire 2.6.1 in \cite{Harari}), this implies that there exist infinitely many  primes $v_0\in \Omega_k$
such that $A$ takes at least two distinct values over $X(k_{v_0})$. 
We choose such a prime $v_{0}$,   nonarchimedean and such
 that $\rho_{v_{0}}$ sends $N_{v_{0}}$ into $M_{v_{0}}$, i.e. $\rho_{v_{0}} \in {\bf X}(O_{v_{0}})$.

Let  $ P \in X(k_{v_0})$ be such that $$inv_{v_0}(A(P))\neq
inv_{v_0}(A({\bf X}(O_{v_0}))) \in \Z/2. $$ By  Witt's theorem, there exists
$\sigma \in SO(V\otimes_{k} k_{v_{0}})$ sending
the point $\rho \in X(k)$ 
 to $P \in X(k_{v_{0}})$.

 Let the quadratic lattice $M' \subset V$ be defined by the conditions
$M'_{v}=M_v$ over $O_{v}$ for each $v\neq v_0$ and $M'_{v_0}=\sigma M_{v_0}$ over $O_{v_{0}}$.
The lattice $M'$ is in the genus of $M$.

Let ${\bf X}'$ be the   $O$-scheme attached to the pair of lattices $N,M'$.
We have equalities ${\bf X}\times_{O}k=X$ and ${\bf X}'\times_{O}k=X$.
For each $v \neq v_{0}$ we have an equality ${\bf X}\times_{O}O_{v}={\bf X}'\times_{O}O_{v}$.
For $v=v_{0}$, the $k_{v_{0}}$-isomorphism $\sigma : X\times_{k}k_{v_{0}} \simeq X\times_{k}k_{v_{0}}$
induces a bijection between ${\bf X}(O_{v_{0}}) $ and ${\bf X}'(O_{v_{0}}) $.
Let $Q \in  {\bf X}'(O_{v_{0}}) $ be the image of $\rho$ under this bijection.
The image of $Q$ under the natural embedding $ {\bf X}'(O_{v_{0}}) \subset X(k_{v_{0}})$ is the point $P$.

For $v \neq v_{0}$, for trivial reasons, the element $A$ takes on $ {\bf X}'(O_{v})$ the same value as
$A$ on  $ {\bf X}(O_{v})$. For $v=v_{0}$, the values taken by $A$ on  ${\bf X}'(O_{v_{0}}) $
are those taken by $\sigma^*(A)$ on ${\bf X}(O_{v_{0}}) $. Since $\sigma$ is an automorphism of
the $k_{v_{0}}$-scheme $X\times_{k}k_{v_{0}} $, the element $\sigma^*(A) \in \br(X\times_{k}k_{v_{0}})$
is of order 2 and its class generates $\br(X\times_{k}k_{v_{0}})/\br(k_{v_{0}})=\Z/2$.
Thus $\sigma^*(A)$ differs from $A$ by an element in $\br(k_{v_{0}})$.
In particular it takes  a single value on ${\bf X}(O_{v_{0}})$, thus $A$ takes  a single value on ${\bf X}'(O_{v_{0}}) $. That value is the one taken on $P$.

Thus for any $\{M_{v}\} \in \prod_{v\in \Omega_k}{\bf X}'(O_v)$ we have
$$\sum_{v \in \Omega_{k}} {\rm inv}_{v}(A'(M_{v})) \neq 0 \in \Z/2,$$
hence
$$(\prod_{v\in \Omega_k}{\bf X}'(O_v))^{\Br X'}=\emptyset.$$
By Proposition \ref{classic.spinorgenus} this implies that $N$ is not represented by the proper
spinor genus of $M'$. Therefore $N$ is a spinor
exception for $gen(M)$.
\end{proof}

With notation as in \S 7.3 (see especially Remark 7.7),
the above result and Theorem 4.1 
of  \cite{HSX1998} give:

\begin{cor}  
With notation as  in this and in the previous section
 the following conditions are equivalent:

$(i)$ $N$ is a spinor exception for the genus of $ M$.

$(ii)$ $\Sigma_{M/N}=K$.

$(iii)$ $\theta (M_v,N_v) = N_{K_w/k_v}(K_w^*)$ for all $v\in
\Omega_k$ and $w|v$.

$(iv)$ For any $A\in \Br X$ which generates $\Br X/\Br k$, $inv_v(A)$
takes a single value over ${\bf X}(O_v)$ for all $v\in \Omega_k$.

\end{cor}

\rem
If these conditions are fulfilled, then it is known that $N$ is represented by exactly half the spinor genera in
$gen(M)$.
\endrem

It is a purely local problem to determine whether $N$ is a spinor
exception for $gen(M)$.
Moreover the   finiteness of the set of
$\Sigma_{M/N}$'s for a given $M$ implies that
the $det(N)$'s of
spinor exception lattices $N$ for $gen(M)$ belong to finitely many
square classes of $k^*/{k^*}^2$.

 In particular, spinor exception
integers for a given ternary genus belong to finitely many square
classes, a fact which has been known for a long time \cite{Kn1961}. For a
proof in terms of the Brauer group, see the next-but-one paragraph.

Suppose $N$ is a spinor exception for $gen(M)$. Then
$N$ is represented by $spn(M)$ if and only if the number of places
$v$ of $k$ satisfying
$$\theta_v(X(N_v/M_v)) \neq \theta(M_v,N_v)$$ is even: this follows from Theorem \ref{classic.codim2}
$(iii)$
and statement  $(iii)$ in the above corollary.
 This is the exact statement of 
Prop.~7.1 in \cite{SPX2004}.

\bigskip
Let us discuss the finiteness of square classes associated to spinor exceptions in the case $n=1, m=3$.
Let  $O$ be the ring of integers in a number field $k$, let $f(x,y,z)$ be a quadratic form in three
variables defined over $O$, nondegenerate over $k$. Let $a \in O, a\neq 0$.
Let $v$ be a finite nondyadic place such that $v(\disc(f))$ is even   and $v(a)$ is odd.
Let ${\bf  X}_{a}$ be the $O$-scheme defined by $f(x,y,z)=a$.
Let $X_{a}/k$ be the affine quadric with equation $q(x,y,z)=a$.
Over $O_{v}$ the quadratic form $f$ is isomorphic to the  form $uv-\det(f)w^2$.
In these last coordinates a generator of $\Br X_{k_{v}}/ \Br k_{v}$ is given by  $\alpha=(u,-a.\disc(f))$.
There are integral $O_{v}$-points on ${\bf  X}_{a}$ with $w=0$ and  $u \in O_{v}^*$
either a square or a nonsquare in the residue field. Thus $\alpha$ takes two distinct values
on ${\bf  X}_{a}(O_{v})$. This implies that there is no Brauer--Manin obstruction.

We thus see that if $a \in O$ is such that $f(x,y,z)=a$ has solutions in all $O_{v}$,
then there may exist a Brauer--Manin obstruction to the existence of an integral point 
 only if  for each nonarchimedean nondyadic  
$v$ with $v(\disc(f)) $ even -- and these are almost all places -- we have $v(a)$ even. This implies that  for given $f$,
 such an $a $ belongs to a finite number of classes in  $k^*/k^{*2}$.

\section{Representation of an integral quadratic form   by another integral quadratic form:  some examples from the literature} \label{sec.rep.quad.ex}
\subsection{Some numerical examples from the literature}

 \subsubsection{}
 
In Cassels' book \cite{Cassels}, p.~168, we find Example~23.
Let $m \equiv \pm 3 \mod 8$. Then $m^2$ is represented {\it primitively } by  the indefinite form $x^2-2y^2+64z^2$
over every $\Z_{p}$ but not over  $\Z$. The equation for ${\bf X}/\Z$ is the complement of $x=y=z=0$ in the affine $\Z$-scheme with equation
$$ x^2-2y^2=(m+8z)(m-8z).$$
One considers the algebra $\alpha=(m+8z,2)=(m-8z,2)$ over $X={\bf X}\times_{\Z}{\Q}$. Over  (primitive) solutions
in ${\bf X}(\Z_{p})$ for $p$ odd or infinity one checks that $\alpha$ vanishes whereas it never vanishes
on points in $X(\Z_{2})$.
For this, one uses the obvious equalities
$ (m+8z) + (m-8z)=2m$ and $ (m+8z) - (m-8z)=16z$.
(Note that for primes $p$ which do not divide $m$, any $\Z_{p}$-solution is primitive.)

\bigskip
 
\subsubsection{}
In the same book \cite{Cassels},  in Example 7 p.~252, we find two examples of positive definite forms and elements which are primitively represented locally but no globally. Cassels refers to papers by G. L. Watson; the hint he gives can certainly be reinterpreted in terms of the law of quadratic reciprocity.

Here is one of these examples. If $m$ is odd and positive and $m \equiv 1 \mod 3$ then $4m^2$ is not represented primitively
by $x^2+xy+y^2+9z^2$ over $\Z$, although it is primitively represented over each $\Z_{p}$.
We can write the equation of $X/\Q$ as
$$x^2+xy+y^2=(2m+3z)(2m-3z).$$
The $\Z$-scheme ${\bf X}$ under consideration here is   the complement
of $x=y=z=0$ in  the ${\Z}$-scheme given by the same equation.
For any $\Z$-algebra $A$, the points of ${\bf X}(A)$ are the primitive solutions
of the above equation, with coordinates in $A$.
We consider the algebra $\alpha=(2m+3z,-3)=(2m-3z,-3)$ over $X$.
Using $(2m+3z)+(2m-3z)=4m$ and $(2m+3z)-(2m-3z)=6z$, one checks that
$\alpha$ vanishes on   points of ${\bf X}(\Z_{p})$ for $p \neq 2,3, \infty$.
It also vanishes on   points of ${\bf X}(\Z_{2})$. Indeed, $\Q(\sqrt{-3})/\Q$
is unramified at $2$. For a   point of ${\bf X}(\Z_{2})$, one checks that
$z \in \Z_{2}^*$. Thus   $2m+3z \in \Z_{2}^*$  and
$2m+3z$ is a  local norm at $2$ for the unramified extension $\Q_{2}(\sqrt{-3})/\Q_{2}$.
Over $\R$, either $2m+3z$ or $2m-3z$ is positive, but since their product is positive
both must be positive. Hence $\alpha$ vanishes on $X(\R)$.

The assumption $m \equiv 1 \mod 3$ implies $2m+3z \equiv 2 \mod 3$.
But the units in $\Z_{3}$ which are norms for the ramified extension $\Q_{3}(\sqrt{-3})/\Q_{3}$
are precisely those which are congruent to $1   \mod 3$. Thus
$\alpha$ never vanishes on ${\bf X}(\Z_{3})$.

\bigskip
 
\subsubsection{}
One may also give such examples  with a positive definite form, excluding the
existence of integral solutions -- not only primitive integral solutions.
Over  $\Q(\sqrt{35})$, Schulze-Pillot in \cite{SP2004} gives an example (Example 5.3).
Let us show that this example can be accounted for by the Brauer--Manin 
condition.

\begin{prop}
 Let $k=\mathbb Q(\sqrt {35})$. Then $7p^2$ where $p$ is a
prime with $(\frac{p}{7})=1$ is not a sum of three  integral squares over the ring of integers
$O=\mathbb Z [\sqrt {35}]$ but is a sum of three integral squares over $O_v$ for all primes $v$
of $F$.
\end{prop}

\begin{proof}
The tangent plane through the rational point
$(\frac{7}{\sqrt{35}}p, \frac{14}{\sqrt{35}}p, 0)$ of
$$x^2+y^2+z^2=7p^2$$ is given by $x+2y-\sqrt{35} p =0 .$

By subsection 5.7, 
one can consider the following quaternion $(x+2y-\sqrt{35}
p, -7)$ over the integral points of $x^2+y^2+z^2=7p^2$ at each
local completion.

If $v$ is a finite prime away from 2,5,7 and $p$, we claim that
$$ord_v(x+2y-\sqrt{35} p)\equiv 0 \mod 2 \ \ \ \text{or} \ \ \ -7\in
(k_v^*)^2 .$$ This implies that $inv_v((x+2y-\sqrt{35}p,-7))=0$.
Indeed, one can write $$x+2y-\sqrt{35} p=u_v \pi_v^{n_v}$$ with
$u_v\in O_v^*$ and only needs to consider that $n_v>0$ and
$-7\not\in (k_v^*)^2 $.

Suppose $(2y-\sqrt{35}p)\in O_v^*$. Then
$$z^2+7(2p-\frac{\sqrt{35}}{7}y)^2 =
2(2y-\sqrt{35}p)u_v\pi_v^{n_v}-u_v^2\pi_v^{2n_v}$$ by plugging $x$
to the above quadric. Since the left hand side is a norm of the unramified
extension $k_v(\sqrt{-7})/k_v$, one concludes that $n_v$ is even.

Otherwise $2y\equiv \sqrt{35} p \mod \pi_v$. Then 
$ z^2\equiv -7 (\frac{p}{2})^2 \mod
\pi_v .$  By Hensel's lemma,  $-7\in (k_v^*)^2 $ which is a
contradiction. The claim follows.

If $v\mid p$, one has
$$(\frac{-7}{p})=(-1)^{\frac{p-1}{2}}(\frac{7}{p})
=(-1)^{\frac{p-1}{2}}(-1)^{\frac{p-1}{2}\frac{7-1}{2}}(\frac{p}{7})=(\frac{p}{7})=1$$
and $-7\in (\mathbb Q_p^*)^2 \subset  (k_v^*)^2$. Then
$inv_v((x+2y-\sqrt{35}p,-7))=0$.

If $v\mid 2$, then $-7\in (\mathbb Q_2^*)^2 \subset  (k_v^*)^2$ and
$inv_v((x+2y-\sqrt{35}p,-7))=0$.

If $v\mid 5$ and $x+2y-\sqrt{35}p\equiv 0 \mod \pi_v$, then
$$7p^2=x^2+y^2+z^2 \equiv 2^2 y^2 + y^2+z^2\equiv z^2 \mod \pi_v .$$
Since $5$ is ramified in $k/\mathbb Q$, the above equation implies that
$7$ is a square mdoulo 5
which is a
contradiction. Therefore $x+2y-\sqrt{35}p$ is a unit and
$inv_v((x+2y-\sqrt{35}p,-7))=0$.

Since
$-7=(\frac{\sqrt{35}}{5})^2(-5)$ and $
(\frac{-5}{7})=1,$ one has $-7\in (k_v^*)^2$ and
$inv_v((x+2y-\sqrt{35}p,-7))=0$ if $v\mid 7$.

 The quaternion is $(x+2y-\sqrt{35}p,-7)$ at one real place $\infty_1$ and
 $(x+2y+\sqrt{35}p,-7)$ at the other real place $\infty_2$. Since
 $x+2y-\sqrt{35} \leq 0$ and $ x+2y+\sqrt{35}\geq
 0$ for $x^2+y^2+z^2=7p$  over $\mathbb R$, one has
 $$ inv_{\infty_1}((x+2y-\sqrt{35}p,-7))=\frac{1}{2} \ \ \ \text{and} \ \ \
 inv_{\infty_2}((x+2y+\sqrt{35}p,-7))=0 . $$

 Therefore $$(\prod_v \mathbf{X}(O_v))^{\Br X}=\emptyset$$ and the
 result follows from \S 6.   
\end{proof}

\bigskip
 \subsubsection{}
We leave it to the reader to handle the following example with $(n,m)=(1,3)$ 
(\cite{BR1995},    \cite{B2001}):
$$-9x^2+2xy+7y^2+2z^2=1.$$
More generally, one may give a criterion for an integer   to be
represented by the indefinite  form $-9x^2+2xy+7y^2+2z^2$
(\cite{X2005} 6.4).
As an exercise, the reader should recover the results of \cite{X2005} from the present point of
view and give a criterion for primitive representation of integers by the above   form.

\bigskip

\subsubsection{}
Here is an example  with $(n,m)=(2,4)$
which goes back to Siegel. This is Example 5.7
in \cite{X2005},   p.~50. Over $\Z$, the form $g(x,y)=x^2+32y^2$ is not represented by
the form $f(x,y,z,t)=x^2+128y^2+128yz+544z^2-64t^2$, even though it is represented
over each $\Z_{p}$ and $\R$. 

Let us explain how this example can be explained from
the present point of view.  
Let $   B(u, v ) $ denote the bilinear form with coefficients in $\Z$ such that $B(u,u)=f(u)$.
Let ${\bf X}/\Z$ be the closed $\Z$-scheme of $\A^8_{\Z} $ given by the identity
$$g(x,y)=f(l_{1}(x,y),\dots,l_{4}(x,y)). $$ Let $X={\bf X}\times_{\Z}\Q$. 
For any commutative ring $A$ a point of ${\bf X}(A)$ is given by a pair of vectors $u_{1}, u_{2} \in A^4$
with $$B( u_{1}, u_{1} ) =1, B( u_{1}, u_{2} ) = 0 ,B( u_{2}, u_{2} ) =32.$$
The standard basis for $A^4$ will be denoted $e_{1},e_{2}, e_{3},e_{4}$.
The discriminant of $g$ is $2^5$, the discriminant  of $f$
is $-2^{22}$. In the notation of previous section,
 we may take $d=2$, $K=\Q(\sqrt{2})$.
On $X$ we find the $\Q$-point  $M$ given by the pair $$v_{1}=(1,0,0,0)=e_{1}; \hskip5mm  v_{2}=(0,1/5,1/5,0)=(1/5)(e_{2}+e_{3}).$$ We also have the  $\Q$-point 
given by the pair $$(1,0,0,0); \hskip5mm (0,1/2,0,0).$$
This ensures ${\bf X}(\Z_{p}) \neq \emptyset $ for all $p$.

The $\Q$-point $M$ gives rise to a morphism $SO(f) \to X$ over $\Q$
and to maps $X(F) \to F^*/N(FK)^*$, 
hence
for each prime $p$ to a map
$$\theta_{p} : {\bf X}(\Z_{p}) \to \Q_{p}^*/NK_{p}^*.$$
Each of these maps is computed in the following fashion:
given a point of ${\bf X}(\Z_{p})$ represented by 
a pair of vectors $w_{1},w_{2} \in  (\Z_{p})^4$, one picks up
 $\sigma  \in SO(f)(\Q_{p})$
such that simultaneously $\sigma v_{1}=w_{1}$ and $\sigma v_{2} = w_{2}$.
One then computes the spinor class of $\sigma$, which is an element in $\Q_{p}^*/\Q_{p}^{*2}$
and one takes its image in $ \Q_{p}^*/NK_{p}^*$.

\medskip

{\bf Assertion} {\it For each prime $p \neq 5$ the image of $\theta_{p}$ is reduced to $1 \in  \Q_{p}^*/NK_{p}^*$. For $p=5$ the image of $\theta_{5}$ is reduced to the nontrivial class $5 \in  \Q_{5}^*/NK_{5}^*$. }

\medskip

The reciprocity law then implies that there is no point in ${\bf X}(\Z)$: the form $f$ does not represent $g$ over $\Z$.

\medskip

Let us prove the assertion.
The system ${\bf X}, M$  has good reduction away from $S=\{2,5,\infty\}$.
We also have $\R^*/N_{K/\Q}K_{\infty}^*=1$. To prove the assertion we could restrict
to considering the primes $p=2$ and $p=5$ but as we shall see
the recipe we apply easily yields the triviality of all maps $\theta_{p}$ for $p\neq 2,5$.

\medskip

In subsection \ref{spinnormcomputation}  we gave a recipe for writing the rotation $\sigma$ as an even product of  reflections, so as to be
able to compute the spinor norm of $\sigma$.
This recipe works provided the pair $\{w_{1},w_{2}\} $
lies in a certain Zariski open set. 
Typically one wants some $f(x-\sigma x)$ to be nonzero in order to use the reflection with respect to
$x-\sigma x$. There are however many ways to  write a rotation as a product  of reflections.
We shall use the basic equality
$$f(x+\sigma x) + f(x-\sigma x) = 4 f(x).$$
This equality ensures that if $f(x)\neq 0$ then one of $f(x+\sigma x)$ or $f(x-\sigma x)$
 is nonzero. For $x \in \Z_{p}^4$ with $p \neq 2$ and $f(x) \in \Z_{p}^*$ it ensures that at least one
 of $f(x+\sigma x)$ or $f(x-\sigma x)$ is in $\Z_{p}^*$.
 
 With notation as in subsection \ref{spinnormcomputation} whenever the appropriate reflections are defined we have
 for each of $i=1,2$ the equalities
 $$ \tau_{\sigma v_1-v_1}
\tau_{[
\tau_{\sigma v_{1}-v_{1}}
 \sigma v_2-v_2]}  v_i = \sigma  v_i  $$ 
  $$
                 \tau_{\sigma v_1+v_1}   \tau_{v_1} \tau_{[\tau_{v_1}\tau_{\sigma v_1+v_1}\sigma
v_2-v_2]} v_i=  \sigma  v_i  $$ 
$$       \tau_{\sigma
v_1-v_1}   \tau_{[\tau_{\sigma v_1-v_1}\sigma v_2+v_2]}   \tau_{v_2} v_i =  \sigma  v_i $$ 
$$           \tau_{\sigma v_1+v_1}  \tau_{v_1}  \tau_{[\tau_{v_1}\tau_{\sigma v_1+v_1} \sigma
v_2+v_2]}     \tau_{v_2} v_i = \sigma v_i  .$$
(To check these formulas, use the property $\tau_{x-y}(y)=x$ is $f(x)=f(y)$.)

 Note that the form $f$ represents 1. Thus for any point $P \in X(F)$ with lift $\sigma \in SO(f)(F)$
 the image of $P$ in $F^*/N(FK)^*$ is the class of any nonzero element among
 $$h_{1}=f(\sigma v_{1}-v_{1}) f( \tau_{\sigma v_{1}-v_{1}} \sigma v_{2}-v_{2})$$
 $$h_{2}=
f(\sigma v_1+v_1) f(v_1) f(\tau_{v_1}\tau_{\sigma
v_1+v_1}\sigma v_2-v_2)$$
$$h_{3}=f(\sigma v_1-v_1)  f(\tau_{\sigma v_1-v_1}\sigma v_2+v_2)  f(v_2)$$
$$ h_{4}=
 f(\sigma v_1+ v_1) f(v_1)  f(\tau_{v_1}\tau_{\sigma v_1+v_1} \sigma
v_2+v_2) f(v_2).$$

 The extension $K=\Q(\sqrt 2)/\Q$ is ramified only at  $2$.

 Suppose $p\neq 2$ and $p\neq 5$. Then $f(v_1)$ is in $\Bbb Z_p^*$. This implies that at least one of $f(\sigma v_1-v_1)$ or $f(\sigma v_1+v_1)$ is in $\Bbb Z_p^*$. 
  Since $f(v_2)$ is in $\Bbb Z_p^*$ for $p\neq 2$ and 5, one has that at least one of $f(\tau_{\sigma v_1-v_1}\sigma v_2 -v_2)$ or $f(\tau_{\sigma v_1-v_1}\sigma v_2 +v_2)$ is in $\Bbb Z_p^*$, and one of $f(\tau_{v_1}\tau_{\sigma v_1+v_1}\sigma v_2 -v_2)$ or $f(\tau_{v_1}\tau_{\sigma v_1+v_1}\sigma v_2 + v_2)$ is in $\Bbb Z_p^*$. This combination implies that at least one of $h_i$ for $1\leq i\leq 4$ is in $\Bbb Z_p^*$, 
  hence has trivial image in $\Q_{p}^*/NK_{p}^*$.
  This proves the assertion for such $p$'s.
  
  \medskip
  
  Consider the case $p=5$. Let $\{w_{1}, w_{2} \} \in {\bf X}(\Z_{5})$.
Thus $w_{1}, w_{2} $ are in $ \Z_5 e_1+ \Z_5 e_2+ \Z_5 e_3 + \Z_5 e_4 $
and there exists $\sigma \in SO(f)(\Q_{5})$ such that $\sigma v_{1}=w_{1}$ and $\sigma v_{2} = w_{2}$. Since $f(v_{1})=1$ from the basic equality we deduce that at least one of
$ f(\sigma v_1-v_1)$  and $ f(\sigma v_1+v_1)$ belongs to $\Z_5^*$.
Let $$ \varrho =\begin{cases} \tau_{\sigma
v_1-v_1}
\sigma   & \mbox{if } f(\sigma v_1-v_1) \in \Z_5^* \\
\tau_{v_1}\tau_{\sigma v_1+v_1} \sigma & \mbox{otherwise}
\end{cases}
$$
Then $\varrho v_1= v_1$ . As for
$\varrho v_2$, it is integral and orthogonal to $\varrho v_1$
hence it belongs to
 $ \Z_5 e_2+\Z_5 e_3 +\Z_5
e_4 .
$ 

There exists  $\varepsilon \in (\Z_5)^4$ with  $\Z_5 e_2 + \Z_5 e_3 = \Z_5
(e_2+e_3) + \Z_5 \varepsilon$ such that $f(\varepsilon)=0 $ and $
B( \varepsilon, e_2+e_3 ) =1 . $ Write $\varrho v_2 = a(e_2+e_3)+
b \varepsilon + c e_4$ with $a$, $b$ and $c \in \Z_5$. Then $a, b\in \Z_5^*$. Otherwise one would have $32= f(\varrho v_2)\equiv -64 c^2
\mod 5$ but $2$ is not a square modulo $5$. 
Immediate computation now yields $B( \varrho
v_2, v_2 ) \in 5^{-1} \Z_5^* .$ This implies 
$h_1$ or $ h_2 \in 5^{-1} \Z_5^*. $
Since $\Q_{5}(\sqrt 2)/\Q_{5}$
is an unramified quadratic field extension, this proves the assertion for $p=5$.

\medskip

Let $p=2$ and $\{w_{1}, w_{2} \} \in {\bf X}(\Z_{2})$.
Thus $w_{1}, w_{2} $ belong to $ \Z_2 e_1+ \Z_2 e_2+ \Z_2 e_3 + \Z_2 e_4 $
and there exists $\sigma \in SO(f)(\Q_{2})$ such that $\sigma v_{1}=w_{1}$ and $\sigma v_{2} = w_{2}$.

Write  $\sigma v_1= \alpha v_1 + w $ with $\alpha
\in \Z_2 $ and $ w\in \Z_2 e_2+\Z_2 e_3
+ \Z_2 e_4 .$ Then  $1-\alpha^2= f(w) \in 2^5 \Z_2. $
Therefore $ \min \{ ord(1-\alpha), \ ord(1+\alpha) \} =1 $ and
$ord(1-\alpha)+ ord(1+\alpha) \geq 5 .$
We have 
$$f(\sigma v_1+v_1)=f((1+\alpha)v_{1}+w)=(1+\alpha)^2 + f(w)= (1+\alpha)^2
+(1-\alpha^2)=2(1+\alpha)$$ 
$$f(\sigma v_1-v_1)=f((\alpha-1)v_{1}+w)=(\alpha-1)^2 + f(w)= (1-\alpha)^2
+(1-\alpha^2)=2(1-\alpha).$$

We have  $1\pm \alpha =2-(1\mp \alpha)=2(1-2^{-1}(1\mp \alpha)).$
If $ord(1+\alpha) \geq 4$ (first case) then  $1-\alpha \in
2(\Z_2^*)^2 $ hence $f(\sigma v_1-v_1) \in 4(\Z_2^*)^2 $.
If $ord(1-\alpha) \geq 4$ (second case) then  $1+\alpha \in
2(\Z_2^*)^2 $ hence $f(\sigma v_1+v_1) \in 4(\Z_2^*)^2. $ 

In the first case   set  $\varrho = \tau_{\sigma v_1-v_1}
\sigma $. We have $\varrho v_{1}=v_{1}$.
The element ${\sigma v_1-v_1}$ belongs to $2  \Z_{2}e_{1} +  \Z_2 e_2+\Z_2 e_3 + \Z_2 e_4 .$ 
This implies that $\varrho v_2$ belongs to $ \Z_2 e_1+ \Z_2 e_2+ \Z_2 e_3 + \Z_2 e_4 $. It actually lies in $ \Z_2 e_2+ \Z_2 e_3 + \Z_2 e_4 $, because
  $\varrho v_2$ is orthogonal to $\varrho v_{1}=v_{1}$.

In the second case  set $ \varrho =\tau_{v_1}\tau_{\sigma v_1+v_1} \sigma$. We have $\varrho v_{1}=v_{1}$.
The element ${\sigma v_1+v_1}$ belongs to $2  \Z_{2}e_{1} +  \Z_2 e_2+\Z_2 e_3 + \Z_2 e_4 .$  This implies that $\varrho v_2$ belongs to $ \Z_2 e_1+ \Z_2 e_2+ \Z_2 e_3 + \Z_2 e_4 $. It actually lies in $ \Z_2 e_2+ \Z_2 e_3 + \Z_2 e_4 $, because
  $\varrho v_2$ is orthogonal to $\varrho v_{1}=v_{1}$.

In the first case we have $h_{1}=f(\sigma v_{1} -v_{1}) f(\varrho v_{2} -v_{2})$
and $h_{3}= f(\sigma v_{1} -v_{1}) f(\varrho v_{2} +v_{2})f(v_{2})$.
In the second case we have $h_{2}=f(\sigma v_{1} +v_{1})f(\varrho v_{2} -v_{2})$
and $h_{4}=f(\sigma v_{1} +v_{1})f(\varrho v_{2} +v_{2})f(v_{2})$

There exists $\varepsilon$ such that  $\Z_2 e_2 + \Z_2 e_3 = \Z_2 v_2
\perp \Z_2 \varepsilon $ with $f(\varepsilon)=2^{11}$.
Write
$\varrho v_2 = a v_2+ b \varepsilon + c e_4$ with $a$, $b$ and $c \in
\Z_2$.  We have $2^5=f(v_{2})=f(\varrho v_2)= 2^5 a^2+ 2^{11} b^2 -2^6 c^2$.
From this we deduce $a \in \Z_2^*$ and $ord(c)\geq 1$. Set $c=2d$ with $d \in \Z_2$. From $a \in \Z_2^*$ we deduce
 $\min \{ord(1+a), ord(1-a)\} =1 $ and $ord(1+a)+ord(1-a)\geq 3 . $
 We have
$f(\varrho v_{2} -v_{2})=2^5(a-1)^2+2^{11}b^2-2^8d^2$ and
$f(\varrho v_{2} +v_{2})=2^5(a+1)^2+2^{11}b^2-2^8d^2$.

Suppose $ord(1-a)=1$. Then $f(\varrho v_{2} -v_{2}) \neq 0 $ is a norm for the extension
$\Q_{2}(\sqrt{2})/\Q_{2}$. If we are in the first case we find that $h_{1}=f(\sigma v_{1} -v_{1}) f(\varrho v_{2} -v_{2})$ is   a norm. If we are in the second case we find that
$h_{2}=f(\sigma v_{1} +v_{1})f(\varrho v_{2} -v_{2})$ is a norm.

Suppose $ord(1+a)=1$. Then $f(\varrho v_{2}+v_{2}) \neq 0 $ is a norm for the extension
$\Q_{2}(\sqrt{2})/\Q_{2}$. 
If we are in the first case we find that $h_{3}= f(\sigma v_{1} -v_{1}) f(\varrho v_{2} +v_{2})f(v_{2})$ is   a norm (recall $f(v_{2})=2^5$).
If we are in the second case we find that
$h_{4}=f(\sigma v_{1} +v_{1})f(\varrho v_{2} +v_{2})$ is a norm.

This completes the proof of the assertion.

\subsubsection{}
Starting from the previous example one immediately gets an example with $(n,m)=(2,3)$.
Indeed the form $x^2+32y^2$ is represented by the form $x^2+128y^2+128yz+544z^2$
over each $\Z_{p}$ and it certainly is  not represented by this form over $\Z$
since it is not represented by the form $x^2+128y^2+128yz+544z^2-64t^2$.

We leave it to the reader to analyze Example 2.9 of \cite{CX2004}: the form
$5x^2+16y^2$ is represented by $4x^2+45y^2-10yz+45z^2$ over each 
$\Z_{p}$ but not over $\Z$.

\bigskip

\subsection{Representation of an integer by a 3-dimensional form: a $2$-parameter family}

In \cite{SPX2004} we find  the following  result (op. cit. p. 324, example 1.2).

\begin{prop} \label{prop.SPX}
Let $n,m,k \geq 1$ be positive integers. The diophantine equation
$$m^2x^2+n^{2k}y^2-nz^2=1$$
is solvable over each $\Z_{p}$ and $\R$ except if $(n,m) \neq 1$.
It is solvable over $\Z$ except in the following cases:
\begin{enumerate}[\rm (i)]
\item
 $(n,m) \neq 1$;
\item
$n \equiv 5  \mod 8 $ and $2$ divides $m$ ;
\item
$n \equiv 3 \mod  8$ and 
4 divides $m$.
\end{enumerate}
\end{prop}

Let us prove this result with the method of the present paper. 
\begin{proof} Let us denote by
${\bf X}$ the affine scheme over $\Z$ defined by the equation $m^2x^2+n^{2k}y^2-nz^2=1$
and by $X$ the $\Q$-scheme ${\bf X}\times_{\Z}\Q$.

Let us first discuss the existence of local solutions. 
If $(n,m) \neq 1$ then there is a prime $p$ such that ${\bf X}(\Z_{p})=\emptyset$.
We now assume $(n,m)=1$.  Over $\Q$, we have the point
$(x,y,z)=(1/m,0,0)$ and the point $(x,y,z)=(0,1/n^k,0)$. For each prime $p$,
at least one of these two  points lies in ${\bf X}(\Z_{p})$. Both lie in $X(\R)$.

One has the equation
$$(1+n^ky) (1-n^ky)=m^2x^2-nz^2. \hskip2mm ({\rm E1})$$
We have the second equation
$$(1+n^ky) +(1-n^ky) = 2. \hskip2mm ({\rm E2})$$
One introduces $\alpha=(1+n^ky,n)$ (note that $1+n^ky=0$ is the tangent plane
to the quadric $X$ at the obvious rational point $(0,-1/n^k,0)$).
As explained above, $\alpha$ belongs to
$\brun X$ and induces the nontrivial element in $\brun X /\br(\Q)$.
Let us restrict attention to the open set $U \subset X$ defined by
$$(1+n^ky)(1-n^ky)=m^2x^2-nz^2 \neq 0. \hskip2mm ({\rm E3})$$

First note that over any field $F$ containing $\Q$ and any point $(x,y,z) \in U(F)$,
the equation (E3) implies $$(1+n^ky,n)=(1-n^ky,n) \in \br(F). \hskip2mm ({\rm E4})$$

Claim: for any prime $p \neq 2$, $\alpha$ vanishes on ${\bf X}(\Z_{p})$.
Let  $(x,y,z) \in {\bf X}(\Z_{p})  \cap U(\Q_{p})$. If $p$ divides $n$, then
$1+n^ky$ is a square in $\Z_{p}$, hence $\alpha$ vanishes.
Suppose   that $p$ does not divide $n$. If $n$ is a square mod $p$
there is nothing to prove. Suppose $n$ is not a square mod $p$.
If $v_{p}(1+n^ky)=0$ then each entry in $(1+n^ky,n)_{p} \in \br(\Q_{p})$
is a unit, hence $(1+n^ky,n)_{p}=0$.
Suppose $v_{p}(1+n^ky)>0.$
Then from (E2) we get $v_{p}(1-n^ky)=0$  hence $(1-n^ky,n)_{p}=0$,
which using (E4) shows  $(1+n^ky,n)_{p}=0$.

From $n>0$ we see that $\alpha$ vanishes on $X(\R)$. 

It remains to discuss the value of $\alpha$ on ${\bf X}(\Z_{2})\cap U(\Q_{2})$.

The theorem now says: there is an integral solution, i.e. a point in ${\bf X}(\Z)$, if and only if
there exists a point of ${\bf X}(\Z_{2})$ on which $\alpha$ vanishes.

If $m$ is odd, then the point $M$ with coordinates $(x,y,z)=(1/m,0,0)$ belongs to ${\bf X}(\Z_{2})\cap U(\Q_{2})$, and $\alpha(M)=(1,n)_{2} =0 \in \br(\Q_{2})$.

Assume now $m$ even, hence $n$ odd.

If $n \equiv 1  \mod 8$ then $n$ is a square in $\Z_{2}$ hence $(1+n^ky,n)_{2}=0$ for any
point of $X(\Z_{2})\cap U(\Q_{2})$.

If $n \equiv -1   \mod 8$ then there exists a point $M \in {\bf X}(\Z_{2})\cap U(\Q_{2})$ with  coordinates $(x,y)=(0,0)$, hence $\alpha(M)=(1,n)_{2}=0$.

Let us consider the remaining cases, i.e. $m$ even and  $n \equiv \pm 3 \mod 8$.

Let us recall the following values of the Hilbert symbol at the prime 2.
We have  $(r,5)_{2}=0$ if $r$ is an odd integer and $(2,5)_{2}=1 \in \Z/2$.
We have $(3,3)_{2}=(3,7)_{2}=1 \in \Z/2$ and $(2,3)_{2}=1 \in \Z/2$.

Assume that $n \equiv 3  \mod 8$ and
$m=2m_{0}$ with $m_{0}$ odd.
The equation $1-4m_{0}^2=-nz^2$ has a solution with $z \in \Z_{2}$. The point
$M$ with coordinates $(x,y,z)=(1,0,z)$ belongs to ${\bf X}(\Z_{2})\cap U(\Q_{2})$,
and $\alpha(M)=(1,n)_{2}=0$.

Assume that $n \equiv 3 \mod 8$ and $4$ divides $m$. Let $M=(x,y,z)$ be a point of ${\bf X}(\Z_{2})\cap U(\Q_{2})$. We have $1-y^2 \equiv -3z^2  \mod 8$.
 If $v_{2}(y)>0$  then $z \in \Z_{2}^*$   and the last equality
implies $1-y^2 \equiv -3  \mod 8$ hence $y=2y_{0}$ with $y_{0} \in \Z_{2}^*$. Thus $1+n^ky \equiv  
3 \  \text{or} \ 7  \mod  8$.
This implies $(1+n^ky,n)_{2}=(3,3)_{2}  \  \text{or} \  (7,3)_{2}=1 \in \Z/2$.
Assume $y \in \Z_{2}^*$. Then (E)
 implies $0\equiv -3z^2  \mod 8$. Hence $4$ divides $z$. Hence
$1-n^{2k}y^2 \equiv 0  \mod 16$. This implies $n^ky \equiv \pm 1  \mod 8$. 
Thus either $1+n^ky \equiv 2   \mod 8$
which implies $(1+n^ky,n)_{2}=1 \in \Z/2$
or $1-n^ky \equiv 2  \mod 8$ which implies $(1-n^ky,n)_{2}=1 \in \Z/2$ hence using (E4) $(1+n^ky,n)_{2}=1 \in \Z/2$.
That is, $\alpha$ is never zero on ${\bf X}(\Z_{2})\cap U(\Q_{2})$.

Assume that   $n \equiv 5  \mod 8$ and $2$ divides   $m$. Let $M=(x,y,z)$ be a point of ${\bf X}(\Z_{2})\cap U(\Q_{2})$. We have $1-y^2 \equiv  3z^2  \mod 4$. This implies $y \in \Z_{2}^*$  and $v_{2}(z)>0$ even. Each of
$1+n^ky$ and $1-n^ky$ has positive 2-adic valuation. Since their sum is $2$, one of them is of the shape $2r$ with $r \in \Z_{2}^*$.
Now $( r,5)_{2}=0$ for $r \in \Z_{2}^*$, so $(2r,5)_{2}=(2,5)_{2}=1 \in \Z/2$. Thus at least one of $(1+n^ky,n)_{2}$ or
$(1-n^ky,n)_{2}$ is nonzero, hence both are nonzero.
\end{proof}

\subsection{Quadratic diophantine equations}

In this section we illustrate how our insistance on arbitrary integral models, as
opposed to the classical ones, immediately leads to results
which in the classical literature would have required 
some work .

\begin{thm} Let $k$ be a number field, $O$ its ring of integers,
$f(x_{1},\dots,x_{n})$ a   polynomial of total  degree 2 and $l(x_{1},\dots,x_{n})$
a   polynomial of total degree 1 which does not divide $f$.

Let ${\bf X}/O$ be the affine closed $O$-subscheme of $\A^n_{O}$  defined by
$$f(x_{1},\dots,x_{n})=0, \hskip1mm l(x_{1},\dots,x_{n})=0.$$

Assume that  $X={\bf X} \times_{O}k$ 
is smooth. Let $v_{0}$ be a place of $k$ such that $X(k_{v_{0}} )$ is noncompact.
Let $O_{\{v_{0}\}}$ be the ring of integers away from $v_{0}$.

(i) If $n \geq 5$, i.e. the dimension of $X$ is at least 3, then ${\bf X}(O_{\{v_{0}\}})$
is dense in $\prod_{v \neq v_{0}} {\bf X}(O_{v}) $.

(ii) If $n=4$, i.e. the dimension of $X$ is 2,   
and if $\{M_{v}\} \in \prod_{v} {\bf X}(O_{v})$
is orthogonal to $\Br X/\Br k \subset \Z/2$, then 
$\{M_{v}\} \in \prod_{v \neq v_{0}} {\bf X}(O_{v})$
may be approximated arbitrarily closely by an element of 
${\bf X}(O_{\{v_{0}\}})$.
\end{thm}

\begin{proof}
The hypothesis on $X$ guarantees that $X$ is $k$-isomorphic to
a smooth affine quadric of dimension $n-2$.  Such a quadric is of the shape
$G/H$ for $G$ a spinor group attached to a quadratic form of rank $n-1$
and $H \subset G$ a spinor group if $n\geq 5$, a torus if $n=4$
(see subsections 5.3 and 5.6).
The result    is   a special case of Theorem
\ref{obs.BM.strong}
(see also Theorem \ref{atleast3} if $n \geq 5$ and 
Theorem \ref{thm.obs.codim2.indefinite}
if $n=4$).
 \end{proof}

\rem
One may write a more general statement, where one gives oneself a finite
set of places containing $v_{0}$ and one
approximates elements in $\prod_{v \in S \setminus v_{0}} X(k_{v}) \times \prod_{v \notin S}{\bf X}(O_{v})$  by points in ${\bf X}(O_{S})$.
\endrem

\rem
Watson proved a result (\cite{Watson1961}, Thm. 2) closely related to the case $n \geq 5$ of the above theorem. It would be interesting to revisit his paper \cite{Watson1967}.
\endrem


\section{Sums of three squares in an imaginary quadratic field} \label{sec.sum.three.squares}

Let   $k=\Q(\sqrt{d})$ be an imaginary  quadratic field. We may and will assume
that $d$ is a negative square-free integer. Let $O$ denote the ring of integers of $k$.
In this section we give a proof based on Theorem \ref{thm.obs.codim2.indefinite} of the
following theorem due to Chungang Ji, Yuanhua Wang and Fei Xu \cite{JWX}:

\begin{thm} \label{thm.three.squares}
Suppose $a\in O$ can be expressed as a sum of three integral squares at each
local completion $O_{v}$.
 If  $a$ is not  a sum of three integral squares in $O$, then 
all the following conditions are fulfilled:
\begin{enumerate}[\rm (i)]
\item
$d \not\equiv 1$ mod $8$;  
\item 
there is a  square-free positive integer $d_0$ such that $a=d_0\alpha^2$ for some $\alpha\in k$;
such a $d_{0}$ is then uniquely determined;
\item
 $d=d_0.d_{1}$ with $d_{1} \in \Z$;
\item
$d_0\equiv 7$ mod $8$;
\item For any odd prime $p$ which divides $ N_{F/\Q}( a)$, at least one of 
 $(\frac{-d_0}{p}), (\frac{-d_1}{p})$ is equal to 1.
 \end{enumerate}
If conditions {\rm (i)} to {\rm (v)}  are fulfilled, then $a$ is not a sum of three squares in $O$.
\end{thm}

\begin{rem}
(1) For $v$ nondyadic,  $-1$ is a sum of two squares in $O_{v}$, hence any element in $O_{v}$ is a sum
of three squares (use the formula $x=(\frac{x+1}{2})^2 - (\frac{x-1}{2})^2$). The local condition on $a$ in the theorem only has to be checked for the dyadic valuations. 

(2) Let us explain the comment on uniqueness of the positive square-free $d_{0}$ in (ii).
 Let $a=d'_{0}.\beta^2$ be another representation.
Then $d_{0}/d'_{0}=(n+m\sqrt{d})^2$ with $n,m \in \Q$, which implies $d_{0}/d'_{0}=n^2$ or $d_{0}/d'_{0}=m^2.d$.
From $d<0$ we conclude that we are in the first case and then $d_{0}=d'_{0}$.

(3)
An easy application of Hilbert's theorem 90 implies that  (ii) holds if and only if  $N_{k/\\Q}(a)=r^2$ for some integer $r \in \Z$.

(4) There is a big difference with the family of examples discussed in Proposition
\ref{prop.SPX}. Given an integer $a \in k$ which is a good candidate, there is in general
 no obvious rational point on $a=x^2+y^2+z^2$ -- unless $d$ is such that $-1$
 is an explicit sum of 2 squares in $k$. 
 
(5) Earlier results on the representation of an integer in a quadratic imaginary field
as a sum of three squares are due to Estes and Hsia \cite{Estes.Hsia}.

 \end{rem}

 Before we begin the proof let us fix some notation and recall facts from subsection~Ê\ref{subsec.affineconic}.
 
 Let $k$ be a field of characteristic not  2, let $(V,Q)$ be a   3-dimensional quadratic space 
over $k$ which in a given basis $V \simeq k^3$ associates $Q(v)=f(x,y,z)$ to $v=(x,y,z)$.
Let $B(v,w)=(1/2)(Q(v+w)-Q(v)-Q(w))$ be the associated linear bilinear form.
 Let $a\in k^*$. 
 We let $X \subset \A^3_{k}$ be the smooth affine quadric defined  by
 the equation $Q(v)-a=0$.  We let $Y \subset \P^3_{k}$ be the smooth projective quadric given by
 the homogeneous equation $Q(v)-at^2=0$.
Suppose $-a$ is not a square in $k$. Then according to subsection  \ref{subsec.affineconic}  we
have $\Br X /\Br k = \Z/2$. 
Let $M    \in X(k)$, which we may view as an element $v_{0} \in V$.
The trace on $X$ of the
tangent plane to $Y$ at $M$ is given by
 $B(v_{0},v) -a=0.$ 
 Let $U_{M} \subset X \subset V$
be the complement in $X$ of that  plane.

By subsection~Ê\ref{subsec.affineconic}  the class of the quaternion algebra
$$A=(B(v_{0},v)-a,-a.\disc(f)) \in \Br U_{M}$$
is the restriction to the open set $U_{M}$ of an element $\alpha$ of
$\Br X $ which generates $\Br X /\Br k =\Z/2$.

 For $ X(k) \neq \emptyset$ that very statement   implies  that
for $F$ any field extension of $k$ such that $-a \notin F^{*2}$, 
the restriction map $\Z/2 =\Br X/ \Br k \to
\Br X_{F}/ \Br F = \Z/2$ is an isomorphism.

If we start from a point $M \in Y(k) \setminus X(k)$,  which may be given
by an element $v_{0} \in V \setminus 0$ with $Q(v_{0})=0$,
the same construction yields the algebra $$A=(B(v_{0},v),-a.\disc(f)) \in \Br U_{M}.$$

\begin{proof}

We thus assume that $a \in O$ is a sum of three squares in each $O_{v}$.

If $-a$ is a square in $k$ then according to  Theorem \ref{thm.obs.codim2.indefinite} $a$ is a sum of
three squares in $O$.
If $-a$ is a square in $k$ and (ii) holds then $-d_{0}$ is a square in $k=\Q(\sqrt{d})$
hence $-d_{0}=d$. Then (i) and (iv) may not simultaneously hold. 

To prove the theorem it is thus enough to restrict to the case where $-a$ is not a square in $k$.
In that case $\Br X/\Br k= \Z/2$. Let $A \in \Br X$ a 2-torsion element which spans
$\Br X/\Br k$.
By Theorem \ref{thm.obs.codim2.indefinite} we know that
${\bf X}(O) \neq \emptyset$ if and only if there exists
a family $\{M_{v}\} \in \prod {\bf X}(O_{v})$ such that
$$\sum_{v} inv_{v}(A(M_{v}))=0 \in \Z/2.$$

For this to happen, it suffices that for some place $v$
the map ${\bf X}(O_{v}) \to \Z/2$ given by evaluation of $A$
is onto.

\medskip

In the next three lemmas we discuss purely local situations.

If $O \subset k$ is a discrete valuation ring with field of fractions $k$, we shall write $L=L_{O} \subset V$
for the trace of $O^3 \subset k^3 \simeq V$.
 
\begin{lem} \label{prop.good.reduction} 
Let $k$ be a nonarchimedean, nondyadic local field, 
$O$ its ring of integers, $a \in O^*$.
Let  ${\bf X} \subset \A^3_{O}$ be the $O$-scheme with affine equation $x^2+y^2+z^2=a$
and let $X={\bf X}\times_{O}k$. 
Then ${\bf X}(O) \neq \emptyset$. For any element $A \in \Br X$, the image of
the map ${\bf X}(O) \to\Q/\Z$ given by $P \mapsto inv(A(P)) $ is reduced to one element.
\end{lem} 
\begin{proof} 
The quadratic form $x^2+y^2+z^2$ is $O$-isomorphic to the quadratic form $2uv -w^2$, and ${\bf X} \subset \A^3_{O}$
 is given by the equation $2uv-w^2=a$. In particular ${\bf X}(O) \neq \emptyset$. Its natural compactification is the smooth $O$-quadric
 ${\bf Y} \subset \P^3_{O}$ given by the homogeneous equation $2uv-w^2=at^2$.
 
If $-a$ is a square, then $\Br X/\Br k=0$ and the result is obvious.
Assume $-a$ is not a square. The point $(u,v,w,t) =(0,1,0,0)$ is a point  of $Y(k)$.
Its tangent plane is given by $u=0$. Thus there exists an element of order 2
in $\Br X$ whose restriction to the open set $u \neq 0 $ of $X \subset \A^3_{k}$
is given by  the quaternion algebra $(u,-a)$, and which spans $\Br X/\Br k$.
Given any point $(\alpha,\beta,\gamma) \in {\bf X}(O)$, from $2\alpha\beta-\gamma^2=a$
we deduce that $\alpha$ and $\beta$ are in $O^*$. Thus $\alpha$
is a norm for the unramified extension $k(\sqrt{-a})/k$ and $inv(\alpha,-a)=0 \in \Z/2$.
 \end{proof}

\begin{lem} \label{prop.nondyadic}
Let $k$ be a nonarchimedean, nondyadic local field, 
$O$ its ring of integers, $a \in O$.
Let  ${\bf X} \subset \A^3_{O}$ be the $O$-scheme with affine equation $x^2+y^2+z^2=a$
and let $X={\bf X}\times_{O}k$. 
Then ${\bf X}(O) \neq \emptyset$. Assume $-a$ is not a square in $k$
and $v(a)>0$.
Then $\Br X/\Br k = \Z/2$ and there exists an element $A$ of order 2 in $\Br X$ which spans
$\Br X/\Br k$. For any such element the image of the map
${\bf X}(O) \to \Z/2$ given by $P \mapsto inv(A(P)) \in \Z/2$ is the whole group
$\Z/2$. 
\end{lem}
\begin{proof}
Since the local field $k$ is not dyadic, the quadratic form $x^2+y^2+z^2$ is $O$-isomorphic to the quadratic form $2xy -z^2$, and ${\bf X} \subset \A^3_{O}$ is given by the equation $2xy-z^2=a$ which we now only consider. We clearly have ${\bf X}(O) \neq \emptyset$.

The assumption  $-a$ is not a square yields $\Br X/\Br k = \Z/2$
as recalled above, that group being generated by the class of an algebra
$A$ whose restriction to a suitable open set is given by a 
 a quaternion algebra $A$ computed from  the equation of the tangent
plane at a $k$-point.

Let  $\nu$ denote the valuation of $k$.
If $\nu(a)$ is odd, let us set $v_{0} =(\frac{1}{2} a,1,0)
\in {\bf X}(O) \subset X(k).$
For given $\epsilon \in O^*$   set  
$v=v(\epsilon)=(\epsilon, \frac{1}{2}a\epsilon^{-1},0)  \in {\bf X}(O) . $ Since $B(
v_0, v) -a=\epsilon+\frac{1}{4}\epsilon^{-1}a^2 -a=\epsilon
\eta^2 $ for some $\eta\in O^*$ by Hensel's Lemma, one
has
$$inv( B( v_0, v) -a,
-a)=inv(\epsilon, -a) \in \Z/2$$
This is equal to $0$ if $\epsilon$ is a square  and to $1$ otherwise.

Fix $\pi$ a uniformizing parameter for $O$.
If $\nu(a)>1$ is even, and  $-a \notin k^{*2}$, set
$v_{0}=(\frac{\pi }{2} a,  \pi^{-1} , 0)  \in X(k).$
Let $v_1=(1,\frac{a}{2},0)  \in {\bf X}(O)$ and 
$v_2=(\pi , \frac{a}{2\pi},0) \in {\bf X}(O).$ 
Then
$B( v_{0}, v_1 )=\pi^{-1} +\frac{a^2}{4} \pi$
and $B( v_{0},
v_2)=1+\frac{a^2}{4}$. Thus
$$inv( B( v_{0}, v_1) -a,
-a )=inv( \pi^{-1}, -a )=1 \in \Z/2.$$
and  $$inv( B( v_{0},
v_2) -a, -a )=inv( 1, -a )=0 \in \Z/2 .$$  
\end{proof}

\begin{lem} \label{prop.dyadic}
Let $k$ be a finite extension of $\Q_{2}$ and $O$ be its
ring of integers. Let $a \in O$. Let  ${\bf X} \subset \A^3_{O}$ be
the $O$-scheme with affine equation $x^2+y^2+z^2=a$ and let $X={\bf
X}\times_{O}k$. Assume ${\bf X}(O) \neq \emptyset$, i.e. $a$ is a
sum of 3 squares in $O$. Assume $-a$ is not a square in $k$. Then
$\Br X/\Br k = \Z/2$ and there exists an element $A$ of order 2 in
$\Br X$ which spans $\Br X/\Br k $. For any such element $A$ the
image of the map ${\bf X}(O) \to \Z/2$ given by $P \mapsto inv(A(P))
\in \Z/2$ is the whole group $\Z/2$.
\end{lem}

\begin{proof} We let $\nu$ denote the valuation on $k$
and $\pi$ be a uniformizing parameter for $O$. We shall use the
following facts from the theory of local fields. Let $K/k$ be a
quadratic field extension of local fields. The subgroup of norms $
N_{K/k} K^* \subset k^*$ is of index 2. This subgroup coincides with
the group of elements of $k$ of even valuation if and only if  $K/k$
is the unramified quadratic extension of $k$. If $K=k(\sqrt{a})$ and
$b \in k^*$, then   $b$ is a norm from $K$ if and only if  the
Hilbert symbol $(a,b)=0 \in \Z/2$.

Suppose $x^2+y^2+z^2$ is $O$-isomorphic to $2xy-z^2$ and ${\bf
X}\subset \A^3_O$ is given by the equation $2xy-z^2=a$ which we now
first consider.

Since ${\bf X}(O)\neq \emptyset$, there is $(x_0,y_0,z_0) \in {\bf
X}(O)$. Let $v_0=(x_0y_0, 1, z_0) \in {\bf X}(O)$ and
$v=(\epsilon^{-1}, \epsilon x_0 y_0,z_0)\in {\bf X}(O)$ for any
$\epsilon \in O^*$. Then
 $$inv(B( v_0, v) -a,-a)=(\epsilon^{-1}(\epsilon x_0y_0-1)^2,-a)=(\epsilon,-a)$$
takes both values $0$ and $1$ in $\Z/2$ if $k(\sqrt{-a})/k$ is
ramified.

If $k(\sqrt{-a})/k$ is unramified, then $\nu(a)$ is even and there
are infinitely many $\xi$ and $\eta$ in $O^*$ such that
$-a\pi^{-\nu(a)}=\xi^2+4\eta$ by 63:3 of \cite{O'Meara}.

 Let $v_0=(-2\eta \pi^{\nu(a)},1,\xi\pi^{\frac{\nu(a)}{2}})$ be such that
$$\pi^{-1} B((\pi,1,0),v_0)=\pi^{-1}(\pi-2\eta\pi^{v(a)})=1-2\pi^{v(a)-1}\eta \in O$$
is non-zero.

 Since the above $v$ only produces the value $0$ in this
case, one needs
$$v'=v_0-\pi^{-1}B((\pi,1,0),v_0)(\pi, 1,0)\in {\bf
X}(O)$$ and $$inv(B( v_0, v') -a,-a)
=(-\pi^{-1}(B((\pi,1,0),v_0))^2,-a)=(-\pi,-a)=1 \in \mathbb Z/2 .$$

\medskip

Next we assume that $x^2+y^2+z^2$ is not isomorphic to $2xy-z^2$
over $O$. Since $x^2+y^2+z^2$ is isomorphic to $2x^2+2xy+2y^2+3z^2$
over $O$ by $x\mapsto x-z$, $y\mapsto y-z$ and $z\mapsto x+y+z$, we
consider that ${\bf X}\subset \A^3_O$ is given by the equation
$2x^2+2xy+2y^2+3z^2=a$. Since ${\bf X}(O) \neq \emptyset$, one can
fix $v_0=(x_0,y_0,z_0)\in {\bf X}(O)$ such that at least one of
$x_0$ or $y_0$ is non-zero by Hensel's Lemma. For any $\epsilon\in
O^*$, there are infinitely many $\eta\in O^*$ such that
$\epsilon\equiv \eta^2 \mod \pi$. By Hensel's Lemma, there is
$\xi\in O$ such that $\xi^2+\xi\eta+\eta^2=\epsilon$ for each
$\eta$. Since there are at most two $\eta$'s satisfying
$B(v_0,(\xi,\eta,0))=0$ for the given $v_0$, one can choose $\eta$
such that $B(v_0,(\xi,\eta,0)) \neq 0$. 
Let
$$v=v_0-\epsilon^{-1}B(v_0,(\xi,\eta,0)) (\xi,\eta,0)\in {\bf X}(O).$$
Then $$inv (B(v_0,v)-a,-a)=(-\epsilon^{-1}
B(v_0,(\xi,\eta,0))^2,-a)=(-\epsilon,-a) $$ takes both values $0$ and
$1$ in $\Z/2$ if $k(\sqrt{-a})/k$ is ramified.

If $ k(\sqrt{-a})/k$ is unramified, then $\nu(a)$ is even. We claim
$\nu(2x^2+2xy+2y^2)$ is odd for all $x,y\in O$. First we show that
$\nu(2)$ is odd. Suppose not, there is $\alpha \in O^* $ such
that $2+3\alpha^2\pi^{\nu(2)}\equiv 0 \mod 2\pi $ by the perfectness
of the residue field. Then
$$2x^2+2xy+2y^2+3z^2 \sim
(2+3\alpha^2\pi^{\nu(2)})x^2+2xy+2y^2+(3+6\alpha^2\pi^{\nu(2)})z^2$$
over $O$ by $$x\mapsto x-2\alpha\pi^{\frac{\nu(2)}{2}}z, \ \ \
y\mapsto y+\alpha \pi^{\frac{\nu(2)}{2}}z, \ \ \ z\mapsto
z+\alpha\pi^{\frac{\nu(2)}{2}}x . $$ By Hensel's Lemma, one has
$$(2+3\alpha^2\pi^{\nu(2)})x^2+2xy+2y^2 \sim 2xy $$ over $O$ (see 93:11 of \cite{O'Meara}). This
contradicts our assumption. Suppose the claim is not true. Then
there are $\alpha,\beta \in O^*$ such that
$\alpha^2+\alpha\beta+\beta^2\equiv 0 \mod \pi$. By Hensel's Lemma
and 93:11 of \cite{O'Meara}, one has $2x^2+2xy+2y^2 \sim 2xy$ over $O$
which contradicts to our assumption. The claim is proved.

By the claim, one obtains that $z_0\neq 0$. Since the above $v$ only
produces the value $0$ in this case, one needs $v'=(x_0,y_0,-z_0)\in
{\bf X}(O)$ and $$inv(B( v_0, v') -a,-a)=(-2z_0^2,-a)=(-2,-a)=1\in
\mathbb Z/2. $$ The proof is complete.

\end{proof}

\begin{lem} \label{prop.step2}
Let $d  < 0$  be a squarefree negative integer.  Let $k={\Q}(\sqrt{d})$.
Let $a$ be a nonzero element in the ring $O$ of integers of $k$. 
Assume that for each place $v$, $a$ is a sum of 3 squares in   $O_{v}$.
Then the set of  conditions
\begin{enumerate}[\rm (a)]
\item
 For each non-dyadic valuation $v$ with $v(a)>0$,  $-a$ is a square in $k_{v}$.
\item
For each dyadic valuation $v$, $-a$ is a square in $k_{v}$.
\end{enumerate}
is equivalent to the set of  conditions
\begin{enumerate}[\rm (i)]
\item
$d \not\equiv 1$ mod $8$;  
\item 
there is a 
square-free  integer $d_0 \in \Z$ such that $a=d_0\alpha^2$ for some $\alpha\in k$;
\item $d=d_0.d_{1}$ with $d_{1} \in \Z$;
\item$d_0\equiv 7$ mod $8$;
\item For any odd prime $p$   which divides $ N_{k/\Q}( a)$ one at least of
$(\frac{-d_0}{p}),  (\frac{-d_1}{p})$ is equal to 1.
\end{enumerate}
\end{lem}

Note that the only difference between the second list of conditions and  the list in Theorem \ref{thm.three.squares} is that we do not demand $d_{0 }>0$.
 
\begin{proof}
 
From (b) we deduce that$-1$ is a sum of 3 squares in each dyadic field $k_{v}$.
 Since $-1$ is not a sum of 3 squares in $\Q_{2}$, this implies that the prime $2$ is not split in the extension $k/\Q$,
 i.e. $d \not\equiv 1$ mod $8$, which proves (i).

Hypotheses (a) and (b) imply

\hskip3mm (c) for any (nonarchimedean) valuation $v$ of $k$, $v(a)$ is even.   

 Thus for any prime $p$, the $p$-adic valuation of the positive integer $N_{k/\Q}(a)$
 is even. Thus   $N_{k/\mathbb Q}(a) \in \N$ is a square.
 An application of Hilbert's theorem 90 shows that  there exist an
 
 integer $r \in \N$
 and an element $\xi \in k$ such that   $a=r.\xi^2.$ We may and will take $r$
 squarefree. From (a)  and (b) we conclude that $r$ consists only of primes ramified in 
 the extension $k/\Q$, i.e. $r$ divides the discriminant $D$ of $k/\Q$. 
  In particular $r$ divides $4d$. Since $r$ is squarefree, $r $
  divides $2d$.

 \medskip

 As we have seen, there is just one valuation $v$ of $k$ above the prime 2.
 In the dyadic field $k_{v}=\Q_{2}({\sqrt d})$, $-a$ is a square, hence so is
 $-r$.  This implies that
 either $-r$ or $-d/r$ is a square in $\Q_{2}$.

 If $-r \in \Z$, which is squarefree,  is  a square in $\Q_{2}$, then $-r$ is odd
and $-r  \equiv 1$ mod $8$.
 We set $d_{0}=r>0$.  The integer $d_{0} $ is squarefree  divides $2d$ and is odd hence divides $d$. It satisfies $d_{0} \equiv 7$ mod $8$.
 Let $\alpha=\xi$. Then $a=d_{0}.\alpha^2$.

Assume that  $-r$ is not a  square in $\Q_{2}$. 
Then  $-d/r $   is a square in $\Q_{2}$. 
Since $r$ divides $2d$ in $\Z$ and  $d$
is squarefree, the $2$-adic valuation of $-d/r$ is $-1$, $0$ or $1$.
It must therefore be $0$, and $-d/r$ is a positive squarefree integer congruent to $1$ mod $8$.
We set $d_{0}=d/r  \in \Z$, $d_{0} < 0$.
 The integer $d_{0} $ is squarefree,  divides $d$
  and satisfies $d_{0} \equiv 7$ mod $8$.
Let $\alpha=r.\xi$. Then 
$a=r.\xi^2= (d/r) .\alpha^2=d_{0}.\alpha^2$.

Since $d$ is squarefree, we may write $d=d_{0}d_{1}$ with $d_{1} \in \Z$
and $d_{0},d_{1}$ coprime and squarefree.

 Let $p$ be an odd prime which divides $N_{F/\Q}( a) \in \N$.
There exists a place $v$
 of $k$ above $p$ such that  $v(a)>0$. By hypothesis (b), $-a$ is a square in $k_{v}$.
 Thus $-d_{0}$ is a square in $k_{v}$.
  If $p$ splits in $k=\Q(\sqrt{d})$,  then $k_{v} \simeq  \Q_{p}$, 
 the squarefree integer $-d_{0}$  is a square in $\Q_{p}$ hence  
 is prime to $p$ and satisfies $(\frac{-d_0}{p})=1$. 
  If $p$ is inert or ramified in $k$, then 
$-d_{0}$ is a square in the quadratic extension $\Q_{p} (\sqrt d)/\Q_{p}$. Thus one of the squarefree integers
$-d_{0}$ or $-d_{1}=-d/d_{0}$ is a square in $\Q_{p}$ hence is the square of a unit in $\Z_{p}$.

Thus the second set of conditions is implied by the first one. 

Suppose (i) to (v) hold.
From (ii) and (iv), where  we get  (b).
From (ii) and (iii) we see that we may write $-a=-d_{0}\alpha^2$ 
and $a=-d_{1}\beta^2$ with $\alpha,\beta \in k$. 
If $v$ be a place of $k$ above an odd prime $p$ and
 $v(a)>0$ then $p$ divides $N_{k/\Q}(a)$. From (v) we then get
 that either $-d_{0}$ or $-d_{1}$ is a square in $\Q_{p}$, hence
 $-a$ is a square in $k_{v}$.
\end{proof}

\bigskip

Let us go back to the  global situation. 
Thus $a$ lies in the ring of integers of $k=\Q(\sqrt{d})$
and 
$a \in O$ is a sum of three squares in $O_{v}$ for each place $v$. Moreover,
$-a$ is not a square in $k$.
Let $A \in \Br X$ be a 2-torsion element  which spans $\Br X/\Br k=\Z/2$.
Let us consider the two conditions 
\begin{enumerate}
\item
 For each non-dyadic valuation $v$ with $v(a)>0$,  $-a$ is a square in $k_{v}$.
\item
For each dyadic valuation $v$, $-a$ is a square in $k_{v}$.
\end{enumerate}
From Lemmas \ref{prop.good.reduction}, \ref{prop.nondyadic} and \ref{prop.dyadic} 
we deduce:

{\it If one of these two conditions does not hold, then there exists
a place $v$ such that the image of ${\bf X}(O_{v}) \to \Z/2$ given by evaluation of $A$
is the whole group $\Z/2$, hence ${\bf X}(O)\neq \emptyset$.}

For any given $A \in \Br X$ as above, we thus have
that for each place $v$ of $k$, the image of the evaluation map
${\bf X}(O_{v}) \to \Z/2$ given by $M_{v}   \mapsto inv_{v}(A(M_{v}))$ is reduced to one element,
say $\alpha_{v} \in \Z/2$, and ${\bf X}(O)=\emptyset$ if and only if
$\sum_{v } \alpha_{v} =1 \in \Z/2$.

For given $a$ satisfying the two conditions above, let us produce a convenient $A$.
Under our assumptions, $-1=a/(-a)$ is a sum of 3 squares, hence of 2 squares, in each dyadic field $k_{v}$. It is  a sum of 2 squares
in any other completion $k_{v}$. By Hasse's  principle
it is a sum of 2 squares in $k$. There thus exists $\rho,\sigma,\tau \in O$, which we may take all nonzero, such that
$\rho^2+\sigma^2+\tau^2=0.$ This defines a point (at infinity) on $Y(k)$. Starting from this point,
the technique 
recalled at the beginning of the proof shows that
the quaternion algebra $ (\rho x + \sigma y + \tau z, -a)$
is   the restriction to the open set $\rho x + \sigma y + \tau z\neq 0$
of a 2-torsion element $A$ of $\Br  X$
which spans $\Br X/\Br k$.

Let $v$ be a place of $k$ and $\pi$ a uniformizing parameter.
 If $-a$ is a square in $k_{v}$, then $A=0$. That is thus the case
 for $v$ dyadic and for $v$ nondyadic such that $v(a)>0$.

Assume that $-a$ is not a square in $k_{v}$.
For $v$ nondyadic with $v(a)=0$,
for $b \in k_{v}^*$,   we have $inv(b,-a)=0 \in \Z/2$ if and only if $v(b)$ is even.
Let $n= inf (v(\rho),v(\sigma),v(\tau))$.
Let $\rho=\rho_{v}Ê\pi^n$, $\sigma=\sigma_{v}Ê\pi^n$, $\tau=\tau_{v} \pi^n$.
Assume $v(\rho_{v})=0$. Let   $$M_{v}= (\rho_{v},\sigma_{v},\tau_{v}) + (a/2\rho_{v}^2) (\rho_{v},-\tau_{v},\sigma_{v}) \in O_{v}^3.$$ We have $M_{v} \in {\bf X}(O_{v})$.
Thus $\alpha_{v}=inv_{v}(A(M_{v}))=inv_{v}( \frac{1}{2}\pi^na,-a) \in \Z/2$ is $0$ or $1$ depending on whether
$n$ is even or odd.  By symmetry in $\rho,\sigma,\tau$, the result holds whichever is the
smallest of $v(\rho),v(\sigma),v(\tau)$.
 
We thus conclude:

{\it Assume that
for each non-dyadic valuation $v$ with $v(a)>0$,  $-a$ is a square in $k_{v}$
and that 
for each dyadic valuation $v$, $-a$ is a square in $k_{v}$. 
There exist $\rho,\sigma,\tau \in $O$,$ none of them zero, such that
$\rho^2+\sigma^2+\tau^2=0.$
Fix such a triple. 
Then
the set ${\bf X}(O)$ is not empty if and only if  the number of  places
 $v$ such that 
 \begin{enumerate}
 \item
 $-a $ is not a square in $ k_{v}$
 \item
  $inf (v(\rho),v(\sigma),v(\tau))$ is odd 
 \end{enumerate}
 is
 even.}

Let us now look for values of $\rho,\sigma,\tau \in O$.
We know that $-1$ is a sum of 2 squares in $k$.
Since it is not a sum of 2 squares in $\Q_{2}$, this implies that the prime $2$ is not split in the extension $k/\Q$, i.e. the squarefree integer $d$ satisfies $d \not\equiv 1$ mod $8$, hence the square free integer $-d$ satisfies$-d  \not\equiv 7$ mod $8$. Thus there exist
$\alpha, \beta, \gamma, \delta$ in $\Z$, not all zero, such that
$\alpha^2+\beta^2 +\gamma^2+d\delta^2=0.$
We may choose them so that none of $\alpha,\beta,\gamma,\delta, \alpha^2+\beta^2 $ is zero.
Then
$$(\alpha^2+\beta^2)^2+(\alpha\gamma+\beta\delta\sqrt{d})^2
+(\beta\gamma-\alpha\delta\sqrt{d})^2=0.$$ 
We may thus take 
$$(\rho,\sigma,\tau)=(\alpha^2+\beta^2, \alpha\gamma+\beta\delta\sqrt{d}, \beta\gamma-\alpha\delta\sqrt{d}).$$

\bigskip

To produce convenient $\alpha,\beta,\gamma,\delta $,
we shall use Hecke's results on primes represented by a binary quadratic forms.
 
\begin{prop}\label{Hecke} 
Let $d<0$ be a squarefree integer, $d  \not\equiv 1$ mod 8.
There exists a prime $l \equiv 1$   mod $4 $ which does not divide $d$
and which is
 represented  over $\Z$  as
  \begin{enumerate} [\rm (a)]
  \item
$l=  -2i^2+2ij-\frac{d+1}{2}j^2 $ if $d\equiv 5$  mod $8$,
\item
$l=
-i^2-dj^2$ if $d\equiv 2$  mod $4$, 
\item
$
l=-i^2+ij-\frac{d+1}{4}j^2$ if $d\equiv 3$  mod $4$.
\end{enumerate}
\end{prop}
\begin{proof}
Let us denote by $q(x,y)$ the quadratic form in the right hand side.
This is a form of discriminant $-4d$, hence
over the reals it breaks up as a  product of two linear forms.
In each of the above cases one checks that there exists
$i_{0}, j_{0} \in \Z$ such that  $q(i_{0},j_{0}) \equiv 1$ mod  $4 $.
Let $\Delta \subset \R \times \R$ be a convex cone with vertex at the origin in
the open set defined by $q(x,y) > 0$.  In each of the above cases the
quadratic form $q(x,y)$ is primitive: there is no prime which divides all its
coefficients.  A direct application of  Hecke's result as made explicit in
 Theorem 2.4 p. 162 of \cite{CTCS1980} 
shows that there exist $i,j \in \Z$ and $l$ a prime number such that $(i,j) \in \Delta$ hence
$q(i,j)>0$,  with $q(i,,j)=\pm l$ hence $q(i,j)=l$ such that
moreover $(i,j) \equiv (i_{0}, j_{0})$ mod $4$, hence
$l \equiv 1$ mod $4$.
\end{proof}

 \begin{rem}
 Earlier papers on the subject \cites{Estes.Hsia, JWX}  already use special representations as provided by the above proposition.
 For their purposes,   Dirichlet's theorem  (for number fields) was enough. 
  \end{rem}
 Fix $l$ as above.
Fix $\alpha,\beta \in \Z$ such that
\begin{enumerate} [\rm (a)]
  \item
$\alpha^2+\beta^2=  2l $
if $d\equiv 5 $ mod $8$,
\item
$ \alpha^2+\beta^2=l $
if $d\equiv 2$  mod $4$,
\item
$\alpha^2+\beta^2=
4l $ if $d\equiv 3$ mod $4$ .
\end{enumerate}

For $i,j$ as above set
\begin{enumerate} [\rm (a)]
  \item
$\gamma = 2i-j , \delta=j $ if $d\equiv 5 $ mod $8$,
\item
$\gamma =i , \delta=j $  if $d\equiv 2$ mod $4$ ,
\item
$\gamma =2i-j, \delta=j $ if $d\equiv 3 $ mod $4$.
\end{enumerate}

 Then $\alpha^2+\beta^2+\gamma^2+d\delta ^2=0$ and $(l,\alpha\beta\gamma\delta)=1$.

 The prime $l$ splits in $k/\Q$.  Let $v_1$ and $v_2$ be the two places of $k$ above $l$. We have
$N_{k/\Q}(\beta\gamma-\alpha\delta\sqrt{d})=
  2l(\alpha^2+\gamma^2)$
if  $d\equiv 5$ mod $8$, 
$N_{k/\Q}(\beta\gamma-\alpha\delta\sqrt{d})=
  l(\alpha^2+\gamma^2)$ if $d\equiv 2$ mod $4$,
  $N_{k/\Q}(\beta\gamma-\alpha\delta\sqrt{d}) 
=4l(\alpha^2+\gamma^2)$  if $d\equiv 3 $ mod $4$. 
We have
$ N_{k/\Q}(\alpha\gamma+\beta\delta\sqrt{d})=  2l(\beta^2+\gamma^2)$  if $d\equiv 5 $ mod $8$,
 $ N_{k/\Q}(\alpha\gamma+\beta\delta\sqrt{d})=l(\beta^2+\gamma^2)$ 
 if $d\equiv 2$ mod $4$, and
 $ N_{k/\Q}(\alpha\gamma+\beta\delta\sqrt{d})= 4l(\beta^2+\gamma^2)$ if $d\equiv 3 $ mod $4$.

Thus
$$ord_{v_1}(\beta\gamma-\alpha\delta\sqrt{d})+ord_{v_2}(\beta\gamma-\alpha\delta\sqrt{d})\geq
1$$
and
 $$ord_{v_1}(\alpha\gamma+\beta\delta\sqrt{d})+ord_{v_2}(\alpha\gamma+\beta\delta\sqrt{d})\geq
1 . $$

 Since
$\beta(\beta\gamma-\alpha\delta\sqrt{d})+\alpha
(\alpha\gamma+\beta\delta\sqrt{d})$
 is equal to 
$ 2l\gamma$ if $d\equiv 5$ mod $8$, is equal to
$l\gamma $ if $d\equiv 2$ mod $4$ and is equal to
$4l \gamma$ if $d\equiv 3$ mod $4$,
one has $$ord_{v_1}(\beta\gamma-\alpha\delta\sqrt{d})\geq 1 \ \
\Leftrightarrow \ \
ord_{v_1}(\alpha\gamma+\beta\delta\sqrt{d})\geq 1$$ and
$$ord_{v_2}(\beta\gamma-\alpha\delta\sqrt{d})\geq 1 \ \
\Leftrightarrow \ \
ord_{v_2}(\alpha\gamma+\beta\delta\sqrt{d})\geq 1 .$$ Since
$$ N_{k/\Q}(\alpha\delta\sqrt{d}-\beta\gamma)+ N_{k/\Q}(\alpha\gamma+\beta\delta\sqrt{d})$$
is equal to
$ 4l(l+\gamma^2)   $ if $d\equiv 5 $ mod $8$,
to $l(l+2\gamma^2) $ if $d\equiv 2$ mod $4$ and to
$8l(2l+\gamma^2)$ if $d\equiv 3$ mod $4$, 
 one has
$$ord_{v_1}(\beta\gamma-\alpha\delta\sqrt{d})+ord_{v_2}(\beta\gamma-\alpha\delta\sqrt{d})=
1$$ or
$$ord_{v_1}(\alpha\gamma+\beta\delta\sqrt{d})+ord_{v_2}(\alpha\gamma+\beta\delta\sqrt{d})=
1 . $$ 
Without loss of generality, one can assume that
$$ord_{v_1}(\beta\gamma-\alpha\delta\sqrt{d})=0 \hskip3mm  and 
\hskip3mm   ord_{v_2}(\beta\gamma-\alpha\delta\sqrt{d})= 1. $$
 Then
$$ord_{v_1}(\alpha\gamma+\beta\delta\sqrt{d})=0 \hskip3mm   and 
\hskip3mm   ord_{v_2}(\alpha\gamma+\beta\delta\sqrt{d})\geq 1.$$

For $$(\rho,\sigma,\tau)=(\alpha^2+\beta^2, \alpha\gamma+\beta\delta\sqrt{d}, \beta\gamma-\alpha\delta\sqrt{d})$$
we thus have 
$$ inf (v_{1}(\rho), v_{1}(\sigma), v_{1}(\tau))=0$$
and
$$ inf (v_{2}(\rho), v_{2}(\sigma), v_{2}(\tau))=1$$
and
$$ inf (v(\rho), v (\sigma), v (\tau))=0$$
for any other nondyadic prime $v$.

 Assume  ${\bf X}(O) = \emptyset$. From Lemmas  \ref{prop.good.reduction}, \ref{prop.nondyadic} and \ref{prop.dyadic} 
we know that  hypotheses (a) and (b)  in Lemma  \ref{prop.step2} hold.
In particular 
$a=d_{0}\zeta^2$ for some $\zeta \in k$. 
Also, the squarefree integer $d$ is not congruent to $1$ mod 8.
We may thus produce  $(\rho,\sigma,\tau)$ as above.
For $A$ associated to $(\rho,\sigma,\tau)$, we have
$\alpha_{v}=0$ for any $v \neq v_{2}$ and $\alpha_{v_{2}}=1$ if
and only if $-a$ is not a square in $k_{v_{2}}$, i.e. if and only if $-d_{0}$ is not a square in $\Q_{l}$,
i.e. if and only if $-1=(\frac{-d_0}{l})$.

From Lemma \ref{prop.step2} we have $d \neq 0, 1, 4$ mod $8$.
Thus
$$-1=(\frac{-d_0}{l})=\prod_{p\mid d_0}(\frac{l}{p})= \prod_{p\mid d_0}(\frac{-1}{p}) $$
 if $d \equiv 2$  or
$3 $ mod $4$
where the second equality follows from  the quadratic reciprocity law
and the third equality follows from Proposition \ref{Hecke}. 
 This  implies that there is an odd number of primes congruent to 3 mod 4
 which divide the squarefree integer $d_{0}$.
 Hence $d_{0}$ is congruent to its sign times 3 modulo 4.
 From Lemma \ref{prop.step2} 
 we have $d_{0}$ congruent to $7$ mod $8$, hence to $3$ mod $4$.
 We now conclude $d_{0}>0$.
 
 Similarly
$$-1=(\frac{-d_0}{l})=\prod_{p\mid d_0}(\frac{l}{p})=  
\prod_{p\mid d_0}(\frac{-2}{p})$$ if $d \equiv 5 $ mod $8$,
and we deduce 
$d_{0}>0$.

Together with 
Lemma \ref{prop.step2}, this completes the proof of  one half of Theorem \ref{thm.three.squares}:
if ${\bf X}(O)$ is empty, then  conditions (i) to (v) in that theorem are fulfilled.

Assume that hypotheses (i) to (v) in Theorem \ref{thm.three.squares} hold.
Then from Lemma \ref{prop.step2} we deduce: for each non-dyadic valuation $v$ with $v(a)>0$,  $-a$ is a square in $k_{v}$ and   for each dyadic valuation $v$, $-a$ is a square in $k_{v}$. 
Proceeding as above one finds suitable $\alpha,\beta,\gamma,\delta $ with associated
$\rho,\sigma,\tau$,  and a prime $l$ congruent to 1 mod 4 such that
${\bf X}(O) =\emptyset$ if and only if $-d_{0}$ is not a square in $\Q_{l}$.
The same computation as above now shows that $d_{0}>0$ implies
that $-d_{0}$ is not a square in $\Q_{l}$.
\end{proof}
 
\section{Appendix: Sum of three squares in a cyclotomic field, \break by Dasheng Wei and Fei Xu} \label{appendix}

In this appendix, we will show that the local-global principle holds
for the sum of three squares over the ring of integers of cyclotomic
fields. First we need the following   lemma which  
appears in \cite{Raj}, Exercises 2 and 3 p.~70. For completeness, we provide
the proof.
\begin{lem} \label{applem} Suppose $R$ is a commutative ring with identity $1_R$.
If $-1_R$ can be written as a sum of two squares over $R$, then any
element which can be written as a sum of squares over $R$ is a sum
of three squares over $R$.
\end{lem}
\begin{proof} Suppose $\alpha\in R$ can be written as a sum of
squares. Since $-1_R$ can be written as a sum of squares, one has
that $-\alpha$ can be written as a sum of squares as well. Let
$-\alpha=\sum_{i=1}^s x_i^2$. Then 
$$\alpha=(\prod_{1\leq i<j\leq
s}x_i x_j + \sum_{i=1}^s x_i +1)^2-[(\prod_{1\leq i<j\leq s}x_i x_j
+ \sum_{i=1}^s x_i )^2+(1+\sum_{i=1}^s x_i)^2].$$ Since $-1_R$ is a
sum of two squares and
$$(a^2+b^2)(c^2+d^2)=(ac+bd)^2+(ad-bc)^2, $$ one concludes that
$\alpha$ is a sum of three squares.
\end{proof}

 Let $k=\mathbb Q(\zeta_n)$ be a cyclotomic field, where
$\zeta_n$ is a primitive $n$-th root of unity. Let $O$ be the ring
of integers of $k$.

\begin{thm} An integer $x\in O$ is a sum of three squares over $O$
if and only if $x$ is a sum of three squares over all local
completions $O_v$.
\end{thm}
\begin{proof} If $2$ is ramified in $k/\mathbb Q$, then $4\mid n$ and
$-1$ is a square in $O$. By Theorem \ref{atleast3}, one has $x$
is a sum of four squares. By Lemma \ref{applem}, one concludes that  $x$ is a
sum of three squares as well.

Now one only needs to consider the case that $2$ is unramified.
Therefore one can assume that $n$ is odd. Let $f$ be the order of
the Frobenius of $2$ in ${\rm Gal}(k/\mathbb Q)= \mathbb Z/(n)^* $.

If $f$ is even, there is a prime $p\mid n$ such that the order of
$2$ in $\mathbb Z/(p)^* $ is even by the Chinese Remainder Theorem,
which is denoted by $2t$. Then $2^t\equiv -1 \mod p$. Let
$\zeta_p\in O$ be a $p$-th primitive root of unity. One has
$$\prod_{i=1}^{t}(1+\zeta_p^{2^i})=\frac{1-(\zeta_p^{2^{t}})^2}{1-\zeta_p^2}
=\frac{1-(\zeta_p^{-1})^2}{1-\zeta_p^2}=-\zeta_p^{-2}.$$ This
implies that $-1$ can be written as a sum of two squares over $O$.
By the same argument as above, $x$ is a sum of three squares.

Otherwise $f=[\mathbb Q_2(\zeta_n):\mathbb Q_2]$ is odd. Suppose $x$ is
not a sum of three squares over $O$. By Lemma 
\ref{prop.dyadic} 
and Theorem \ref{thm.obs.codim2.indefinite}, one obtains that $-x$ is a square in $\mathbb
Q_2(\zeta_n)$. This implies $-1$ is a sum of three squares over
$\mathbb Q_2(\zeta_n)$. By Springer's Theorem (see Chapter 2, \S 5.3 of
\cite{Sch}), $-1$ is a sum of three squares over $\mathbb Q_2$.
Therefore $-1$ is a sum of two squares by Pfister's Theorem (see
Chapter 2, \S 10.8 of \cite{Sch}) over $\mathbb Q_2$. A contradiction
is derived.

\end{proof}



\begin{bibdiv}
\begin{biblist}

 
 



\bib {Bo1996}{article}{
    author={M. Borovoi},
     title={The Brauer--Manin obstructions for homogeneous spaces with connected or abelian stabilizer},
   journal={J. f\"ur die reine und angew. Math. (Crelle)},
    volume={473},
      date={1996},
    number={},
     pages={181\ndash 194},
} 

\bib {Bo1999}{article}{
    author={M. Borovoi},
     title={The defect of weak approximation for homogeneous spaces.},
   journal={Ann. Fac. Sci. Toulouse},
    volume={8},
      date={1999},
    number={},
     pages={219 \ndash 233},
} 

\bib {B2001}{article}{
    author={M. Borovoi},
     title={On representations of integers by indefinite ternary quadratic forms},
   journal={J. Number Theory},
    volume={90},
      date={2001},
    number={2},
     pages={281\ndash 293},
}

\bib {BR1995}{article}{
    author={M. Borovoi },
    author={ Z. Rudnick},
     title={Hardy-Littlewood varieties and semisimple groups},
   journal={Invent. math.},
    volume={119},
      date={1995},
    number={ },
     pages={37\ndash 66},
}


\bib {Cassels}{book}{
    author={J. W. S. Cassels},
     title={ Rational quadratic forms},
    volume={13},
     publisher={Academic Press},
     place={London},
      date={1978},
      series={London Mathematical Society Monographs},
     pages={\ndash },
}




\bib {CX2004}{article}{
    author= {Wai Kiu Chan},
    author={Fei Xu },
     title={ On representations of spinor genera},
   journal={Compositio Math.},
    volume={140},
      date={2004},
    number={ },
     pages={ 287 \ndash 300},
} 



\bib{CTflasque}{article}{
    author={J.-L. Colliot-Th\'el\`ene},
     title={R\'esolutions flasques des groupes lin\'eaires connexes},
     journal={J. f\"ur die reine und angew. Math. (Crelle)},
    volume={ },
      date={ },
    number={ },
     pages={\`a para\^{\i}tre},
     }

\bib {CTCS1980}{article}{
    author={J.-L. Colliot-Th\'el\`ene},
        author={ D. Coray },
        author={J.-J. Sansuc },
     title={ Descente et principe de Hasse pour certaines
vari\'et\'es rationnelles},
   journal={J. f\"ur die reine und angew. Math. (Crelle)},
    volume={320},
      date={1980},
    number={ },
     pages={150\ndash 191},
}


\bib {CTS1987}{article}{
    author={J.-L. Colliot-Th\'el\`ene},
    author={J.-J. Sansuc },
     title={La descente sur les vari\'et\'es rationnelles II},
   journal={Duke Math. J. },
    volume={54},
      date={1987},
    number={ },
     pages={375\ndash 492},
}


\bib {CTPest}{article}{
    author={J.-L. Colliot-Th\'el\`ene},
         title={Points rationnels sur les fibrations, in Higher dimensional varieties and rational points  (Budapest, 2001)},    
  journal={ },
  publisher={Springer},
  series={Bolyai Soc. Math. Stud.},     volume={12},
      date={2003},
     pages={171\ndash 221},
}  


\bib {CTK}{article}{
    author={J.-L. Colliot-Th\'el\`ene},
    author={B. \`E. Kunyavski\u{\i}},
         title={Groupe de Picard et groupe de Brauer des compactifications lisses d'espaces homog\`enes},    
  journal={J. Algebraic Geometry},
  volume={15},
      date={2006},
    number={ },
     pages={733 \ndash 752},
} 



\bib{Eichler}{article}{
    author={M. Eichler},
 title={Die \"Ahnlichkeitsklassen indefiniter Gitter},
  journal={Mathematische Zeitschrift},
    volume={55 },
      date={1952},
    pages={216\ndash 252},
      }
 
 \bib{ER2006}{article}{
  author={I. V. Erovenko},
  author={A. S. Rapinchuk},
 title={Bounded generation of $S$--arithmetic subgroups of isotropic orthogonal groups over number fields},
  journal={Journal of Number Theory},
    volume={119 },
      date={2006},
    pages={28\ndash 48},
 }



\bib{Estes.Hsia}{article}{
    author={D. R. Estes},
    author={J. S. Hsia},
 title={Sum of three integer squares in complex quadratic fields},
  journal={Proc. A.M.S.},
    volume={89 },
      date={1983},
    pages={211\ndash 214},
 }
 
 
 \bib{Gro} {article}{
    author={A. Grothendieck},
     title={Le groupe de Brauer, I, II, III},
   journal={Dix expos\'es sur la cohomologie des sch\'emas},
    volume={},
     publisher={Masson},
      date={1968},
    number={ },
     pages={ },
}

\bib{Harari} {article} {
    author={D. Harari},
 title={M\'ethode des fibrations et obstruction de Manin},
  journal={Duke Math. J. },
    volume={75 },
      date={1994},
    pages={221\ndash 260},
 }


\bib{Harder} {article} {
    author={G. Harder},
 title={Halbeinfache Gruppenschemata \"uber Dedekindringen},
  journal={Inventiones math.},
    volume={4 },
      date={1967},
    pages={165 \ndash 191},
 }





\bib {HSX1998}{article}{
    author={J.S. Hsia},
    author={Y.Y.  Shao},
    author={ F. Xu},
     title={Representations of indefinite quadratic forms},
   journal={J. f\"ur die reine und angew. Math. (Crelle)},
    volume={494},
      date={1998},
    number={ },
     pages={129\ndash 140},
}

\bib {JWX}{article}{
    author={Chungang Ji},
    author={Yuanhua Wang},
    author={Fei Xu},
     title={Sums of three squares over imaginary 
quadratic fields},
   journal={Forum Math.},
    volume={18},
      date={2006},
    number={ },
     pages={585 \ndash 601},
}

\bib {JW1956}{article}{
    author={B.W. Jones},
    author={G. L. Watson},
    title={On indefinite ternary quadratic forms},
   journal={Canad. J. Math.},
    volume={8},
      date={1956},
    number={ },
     pages={592\ndash 608},
}

\bib {Kn1956}{article}{
    author={M. Kneser},
     title={Klassenzahlen indefiniter quadratischer Formen in drei oder mehr Ver\"anderlichen},
   journal={Archiv der Mathematik},
    volume={VII},
      date={1956},
    number={ },
     pages={323\ndash 332},
} 

\bib {Kn1961}{article}{
    author={M. Kneser},
     title={Darstellungsmasse indefiniter quadratischer
Formen},
   journal={Mathematische Zeitschrift},
    volume={77},
      date={1961},
    number={ },
     pages={188\ndash 194},
} 


\bib {Kn}{article}{
    author={M. Kneser},
     title={Starke Approximation in algebraischen Gruppen,I},
   journal={J. f\"ur die reine und angew. Math. (Crelle)},
    volume={218},
      date={1965},
    number={ },
     pages={190\ndash 203},
} 



\bib{KSS} {book}{
 author={H. Kraft},
     author={P. Slodowy},
    author={T. A. Springer},   
      title={Algebraische Transformationsgruppen und Invariantentheorie},
   publisher={Birkh\"auser},
    series={DMV Seminar},
     volume={13},
      date={1989},
    number={ },
} 

 

\bib {K1986}{article}{
    author={R. E.  Kottwitz},
     title={Stable trace formula: elliptic singular terms},
   journal={Math. Annalen},
    volume={275},
      date={1986},
    number={ },
     pages={365\ndash 399},
}

\bib {Milne1986}{book}{
    author={ J. S. Milne},
     title={Arithmetic duality theorems},
       volume={ },
     publisher={Academic Press},
     place={},
      date={1986},
   journal={ },
    series={Perspectives in Mathematics},  
    volume={1},
    number={ },
}



\bib{Nisnevich} {article} {
    author={Ye. A. Nisnevich},
 title={Espaces homog\`enes principaux rationnellement triviaux et arithm\'etique des sch\'emas en groupes r\'eductifs sur les anneaux de Dedekind. },
  journal={C. R. Acad. Sc. Paris S\'er. I},
    volume={299},
      date={1984},
    pages={5 \ndash 8},
 }

\bib{O'Meara}{book}{
    author={O. T.  O'Meara},
     title={Introduction to Quadratic Forms},
      publisher={Springer Verlag},
   journal={ },
   series={Grundlehren der Mathematik},  
    volume={117 },
      date={ 1971},
    number={ },
}


\bib{PR1994}{book}{
    author={V. P. Platonov},
    author={A. S. Rapinchuk},
     title={Algebraic groups and  number theory},
       volume={ },
     publisher={Nauka},
     place={Moscow},
      date={1991},
    volume={ },
    number={ },
}


 \bib{Raj}{book}{
    author={A. R. Rajwade},
     title={Squares},
       volume={ },
     publisher={Cambridge University Press},
     place={},
      date={1993},
    series={London Mathematical Society Lecture Note Series},  
    volume={171},
    number={ },
}


\bib{S1981}{article}{
    author={J.-J. Sansuc},
    title={Groupe de Brauer et arithm\'etique des groupes alg\'ebriques lin\'eaires sur un corps de nombres},
   journal={ J. f\"ur die reine und angew. Mathematik (Crelle)},
    volume={327 },
      date={1981},
    number={ },
     pages={12 \ndash 80},
}

\bib {Sch}{book}{
    author={W. Scharlau},
    author={},
    author={ },
     title={Quadratic and Hermitian forms},
     publisher={Springer},
     place={},
      date={1986},
   journal={ },
    series={Grundlehren der Mathematik},
    volume={270},
    number={ },
}



\bib {SP1980}{article}{
    author={R. Schulze-Pillot },
     title={Darstellung durch Spinorgeschlechter
tern\"arer quadratischer Formen},
   journal={ J.Number Theory }, 
    volume={12},
      date={1980},
    number={ },
     pages={529\ndash 540},
 }

\bib {SP2000}{article}{
    author={R. Schulze-Pillot },
     title={Exceptional integers for genera of integral ternary positive definite quadratic forms},
   journal={ Duke Math. J. }, 
    volume={102},
      date={2000},
    number={ },
     pages={351\ndash 357},
     }




\bib {SP2004}{article}{
    author={R. Schulze-Pillot },
     title={Representation by integral quadratic forms---a survey.  Algebraic and arithmetic theory of quadratic forms},
   journal={ Contemp. Math.}, 
    volume={344},
      date={2004},
    number={ },
     pages={323\ndash 337},
}

\bib {SPX2004}{article}{
    author={R. Schulze-Pillot  },
    author={F. Xu },
     title={Representations by spinor genera of ternary quadratic forms},
   journal={ Contemporary Mathematics }, 
   volume={344},
      date={2004},
    number={ },
     pages={323\ndash 337},
}

\bib {SerreCG}{book}{
    author={J-P. Serre},
     title={Cohomologie galoisienne},
     publisher={Springer},
     place={Berlin Heidelberg New York},
      journal={ }, 
      series={Lecture Notes in Mathematics},  
    volume={5},
    date={1965},
    number={ },
     pages={},
}


\bib {Sko}{book}{
    author={A. N. Skorobogatov},
     title={Torsors and rational points},
     publisher={Cambridge University Press},
     place={},
      journal={ }, 
            series={Cambridge Tracts in Mathematics},  
    volume={144},
    date={2001},
    number={ },
     pages={},
}

\bib {Watson1961}{article}{
    author={G. L. Watson},
     title={Indefinite quadratic diophantine equations},
   journal={Mathematika},
    volume={8},
      date={1961},
    number={ },
     pages={32\ndash 38},
     }


\bib {Watson1967}{article}{
    author={G. L. Watson},
     title={Diophantine equations reducible to quadratics},
   journal={Proc. London Math. Soc.},
    volume={17},
      date={1967},
     pages={26\ndash 44},
     }
     
\bib {W}{article}{
    author={A. Weil },
     title={Sur la th\'eorie des formes quadratiques},
   journal={Colloque sur la th\'eorie des groupes alg\'ebriques, CBRM, Bruxelles },
      date={1962},
    number={ },
     pages={9\ndash 22},
} 


\bib {X2000}{article}{
    author={F. Xu},
     title={Representations of indefinite ternary quadratic forms
over number fields},
   journal={ Mathematische Zeitschrift},
    volume={234},
      date={2000},
    number={ },
     pages={115\ndash 144},
}

\bib {X2005}{article}{
    author={F. Xu},
     title={On representations of spinor genera II},
   journal={ Math. Annalen},
    volume={332},
      date={2005},
    number={ },
     pages={37\ndash 53},
}



\end{biblist}
\end{bibdiv}
\end{document}